\documentclass[bj,numbers,noshowframe]{imsart}
\usepackage{amsmath,amsthm,amssymb,amsfonts,hyperref,xcolor,bbm,circuitikz,enumitem,graphicx,algorithm,algorithmic,array,verbatim,caption,subcaption}

\startlocaldefs
\theoremstyle{plain}
\newtheorem{theorem}{Theorem}
\newtheorem{lemma}{Lemma}
\newtheorem{corollary}{Corollary}
\theoremstyle{definition}
\newtheorem{assumption}{Assumption}
\newtheorem{definition}{Definition}
\newtheorem{problem}{Class}
\newtheorem{example}{Example}
\newtheorem{remark}{Remark}

\newcommand{\mR}{\mathbb{R}}
\newcommand{\mF}{\mathcal{F}}
\newcommand{\tlambda}{\Tilde{\lambda}}
\newcommand{\hlambda}{\hat{\lambda}}
\newcommand{\clambda}{\check{\lambda}}
\newcommand{\blambda}{\bar{\lambda}}
\DeclareMathOperator*{\argmin}{argmin}
\DeclareMathOperator*{\argmax}{argmax}
\newcommand{\expt}{\mathbb{E}}

\newcommand{\idf}{\mathbbm{1}}
\newcommand{\ir}{r^{\textup{int}}}
\newcommand{\ie}{\textup{ERR}^{\textup{int}}}
\newcommand{\ess}{\textup{ESS}^{\textup{int}}}
\newcommand{\cl}{\check{\ell}}
\newcommand{\hl}{\hat{\ell}}
\newcommand{\bd}[1]{\boldsymbol{#1}}
\newcommand{\prob}{\mathbb{P}}

\newcommand{\cov}{\textup{Cov}}

\newcommand{\bSigma}{\boldsymbol{\Sigma}}
\newcommand{\bell}{\boldsymbol{\theta}}
\newcommand{\BellSP}{\boldsymbol{\Theta}}

\newcommand{\ys}[1]{\textcolor{blue}{(#1)}}
\newcommand{\mW}{\mathcal{W}}
\newcommand{\KL}{\textup{KL}}

\endlocaldefs

\makeatletter
\def\journal@name{} 
\makeatother

\begin{document}
\begin{frontmatter}
\title{Sequential Multiple Testing: A Second-Order Asymptotic Analysis}
\runtitle{Sequential Multiple Testing: Second-Order Analysis}

\begin{aug}
\author{\fnms{Jingyu}~\snm{Liu}\ead[label=e1]{jingyu.liu@queensu.ca}}
\author{\fnms{Yanglei}~\snm{Song}\ead[label=e2]{yanglei.song@queensu.ca}}
\address{Department of Mathematics and Statistics, Queen's University, Kingston, Canada}
\address{\printead[presep={\ }]{e1,e2}}
\end{aug}

\begin{abstract}
We study sequential multiple testing with independent data streams, where the goal is to identify an unknown subset of signals while controlling commonly used error metrics, including generalized familywise rates and false discovery and non-discovery rates. For these problems, procedures that are first-order optimal are known, in the sense that the ratio of their expected sample size (ESS) to the minimal achievable ESS converges to one as the error tolerance levels vanish. In this work, we develop a unified theory of second-order asymptotic optimality. We establish general sufficient conditions under which second-order Bayesian optimality implies second-order frequentist optimality for broad classes of sequential testing procedures. As a consequence, several procedures previously known to be first-order optimal are shown to be second-order optimal: for every signal configuration, the difference between their ESS and the minimal achievable ESS remains uniformly bounded as the error tolerance levels tend to zero. In addition, we derive a second-order asymptotic expansion of the minimal achievable ESS, which refines the classical first-order approximation by identifying the second-order correction term arising from a boundary-crossing problem for a multidimensional random walk. We apply this result to several commonly used error metrics.
\end{abstract}


\end{frontmatter}


\section{Introduction}
Multiple testing concerns the simultaneous assessment of $K$ hypotheses and remains a central topic in statistics. Classical formulations assume a fixed sample size and focus on controlling prespecified Type-I error criteria. Representative examples include the familywise error rate (FWER), which bounds the probability of at least one false rejection \cite{holm1979simple,hommel1988stagewise}; the generalized FWER, which controls the probability of at least $k \ge 1$ false rejections \cite{Lehmann:2005:GFE}; and the false discovery rate (FDR), defined as the expected proportion of false rejections among all discoveries \cite{benjamini1995controlling}.

In many applications, such as industrial process control, clinical trials, and multichannel signal detection, data are collected sequentially and decisions must adapt to the accumulated observations \cite{tartakovsky2014sequential}. 
Sequential procedures can substantially reduce sampling effort while maintaining prescribed \emph{Type-I and Type-II} error guarantees. This framework originates from Wald’s seminal work on the sequential probability ratio test (SPRT) \cite{wald1945sequential}. Subsequent research extended Wald’s ideas to multiple hypotheses observed through multiple data streams, leading to two main paradigms: sampling may terminate at different times across streams \cite{malloy2014sequential,bartroff2014sequential,bartroff2018multiple,xing2025asymptotically}, or simultaneously across all streams \cite{de2012sequential,de2012step,song2017asymptotically,song2019sequential,he2021asymptotically,xing2023signal,chaudhuri2024joint}. In this work, we focus on the simultaneous stopping paradigm. 

Specifically, we consider $K$ independent data streams, with $\{\mathbf{X}_1^{k}, \mathbf{X}_2^{k}, \ldots\}$ denoting the observations from the $k$th stream, for $k \in [K] := \{1,\ldots,K\}$. For each stream $k$, we test two simple hypotheses $H_0^{k}$ versus $H_1^{k}$. Let $A \subset [K]$ denote the (unknown) set of signals, that is, the streams for which the alternative hypothesis holds. We allow $A$ to take values in a prescribed collection $\mathcal{A} \subset 2^{[K]}$, where $2^{[K]}$ is the power set of $[K]$. If $\mathcal{A}$ is a strict subset of $2^{[K]}$, then it encodes structural information about the signal subset. For example, if $\mathcal{A} = \{A \subset [K] : |A| = m\}$, then the number of signals is known to be $m$. In this work, we focus on the case of independent streams and refer to \cite{chaudhuri2024joint} for extensions to dependent settings.

At each time $t \ge 1$, samples $\{\mathbf{X}_t^{1},\ldots,\mathbf{X}_t^{K}\}$ are observed. Sampling continues until a stopping rule determines that sufficient evidence has been collected, at which point a decision about the unknown signal subset is made. Formally, a sequential multiple testing rule consists of a stopping time $T$ and a terminal decision rule $D \in \mathcal{A}$, where $T$ is the total number of observations collected and $D$ estimates $A$ based on the data observed up to time $T$. The objective is to construct $(T,D)$ so as to control prescribed error criteria while keeping the expected sample size (ESS) $\expt_A[T]$ small for each $A \in \mathcal{A}$.

In contrast to classical fixed-sample formulations, sequential procedures typically aim to control both Type-I and Type-II errors. Procedures have been developed to control  familywise error rates (FWER) \cite{bartroff2010multistage,de2012sequential,de2012step,bartroff2014sequential}, generalized familywise error rates \cite{bartroff2018multiple,de2015sequential}, generalized misclassification rates \cite{li2014universal,malloy2014sequential}, and false discovery and non-discovery rates \cite{javanmard2018online,bartroff2020sequential,he2021asymptotically}; see Subsection \ref{sec: multi testing error} for precise definitions. For these criteria, first-order asymptotic optimality results are available, in some cases under additional structural information such as a known number of signals \cite{song2017asymptotically,song2019sequential,he2021asymptotically}.

Specifically, denote by $\Delta(\bell)$ a general class of sequential procedures indexed by $\bell \in (0,1)^r$, where $r \ge 1$. In specific applications, $\Delta(\bell)$ may represent the collection of procedures that satisfy a prescribed error criterion at level $\bell$, in which case $r$ is typically either $1$ or $2$. Given $\Delta(\bell)$ and $A \in \mathcal{A}$, define
\begin{equation}\label{def: small ESS}
T_A^{\textup{min}}(\Delta(\bell))
:=
\inf_{(T,D)\in \Delta(\bell)} \expt_A[T],
\end{equation}
the minimal achievable expected sample size (ESS) over $\Delta(\bell)$ under signal configuration $A$.

Let $\{\delta(\bell) = (T(\bell), D(\bell)) : \bell \in (0,1)^r\}$ be a family of procedures such that $\delta(\bell) \in \Delta(\bell)$ for each $\bell$. We say that $\{\delta(\bell)\}$ is \emph{first-order asymptotically optimal} with respect to $\{\Delta(\bell)\}$ if, for each $A \in \mathcal{A}$,
\[
\expt_A[T(\bell)]
=
T_A^{\min}(\Delta(\bell))(1+o(1))
\quad \text{as } \bell \to \bd{0},
\]
that is, the ratio of the ESS of the procedure to the minimal achievable ESS converges to one as $\bell$ approaches the zero vector $\bd{0}$.
Moreover, for a collection of real numbers $\{Q_1(\bell) : \bell \in (0,1)^r\}$, we say that $Q_1(\bell)$ provides a first-order asymptotic approximation to $T_A^{\min}(\Delta(\bell))$ if
\[
T_A^{\min}(\Delta(\bell))
=
Q_1(\bell)(1+o(1)).
\]

For the error control criteria mentioned above, both first-order optimal procedures and corresponding first-order approximations of the minimal achievable ESS have been established \cite{song2017asymptotically,song2019sequential,he2021asymptotically}. We refer to Section \ref{sec: application} for further details. Here, for concreteness, we consider the class of sequential procedures 
$\Delta^{\textup{gmr}}_{m_0}(\alpha)$ that controls the probability of committing at least $m_0$ mistakes, regardless of Type-I or Type-II, at level $\alpha \in (0,1)$ under each signal configuration $A \subset [K]$, that is, $\mathcal{A} = 2^{[K]}$.  In \cite{song2019sequential}, the Sum-Intersection rule $\{\delta_S(b_\alpha) = (T_S(b_\alpha), D_S(b_\alpha))\}$ is proposed and shown to be first-order asymptotically optimal, where $b_{\alpha}$ is a threshold so that $\delta_S(b_{\alpha}) \in \Delta^{\textup{gmr}}_{m_0}(\alpha)$ (see Subsection \ref{sec: gmr}). Moreover, the minimal achievable ESS is characterized to the first-order. That is, for each $A \subset [K]$, as $\alpha \to 0$,
\begin{align*}
\expt_A[T_S(b_\alpha)]  =
T_A^{\min}(\Delta^{\textup{gmr}}_{m_0}(\alpha))(1+o(1)), \quad \text{ and }\quad
T_A^{\min}(\Delta^{\textup{gmr}}_{m_0}(\alpha))
= \frac{|\log(\alpha)|}{\kappa_A}(1+o(1)),
\end{align*}
where $\kappa_A$ is a positive constant independent of $\alpha$ (see Subsection \ref{sec: gmr} for its definition).

In Fig. \ref{fig: intro 1}, we present a numerical study with $K=20$, $m_0=1$, and $A=\emptyset$; see Section \ref{sec:numerical_results} for details. The left panel displays, using square markers, the ESS of the Sum--Intersection rule,
$\expt_A[T_S(b_\alpha^*)]$, where $b_\alpha^*$ is a non-conservative threshold. 
It also shows, using triangle markers, the first-order approximation 
$|\log(\alpha)|/\kappa_A$. Both quantities are plotted against 
$|\log_{10}(\alpha)|$ as $\alpha \to 0$. The middle and right panels show the difference $\expt_A[T_S(b_\alpha^*)] - |\log(\alpha)|/\kappa_A$ and the ratio $\expt_A[T_S(b_\alpha^*)] / (|\log(\alpha)|/\kappa_A)$, respectively. As predicted by first-order asymptotic theory, the ratio converges to one as $\alpha \to 0$. However, the middle panel indicates that the difference is not bounded; instead, it diverges as $\alpha \to 0$.

This raises the question of whether the Sum-Intersection rule enjoys a higher-order form of optimality, namely, whether the difference $\expt_A[T_S(b_\alpha)] - T_A^{\min}(\Delta^{\textup{gmr}}_{m_0}(\alpha))$ remains bounded as $\alpha \to 0$. Here the comparison is made with the minimal achievable ESS $T_A^{\min}(\Delta^{\textup{gmr}}_{m_0}(\alpha))$, rather than with its approximations. Equally important is the problem of obtaining a sharper asymptotic characterization of $T_A^{\min}(\Delta^{\textup{gmr}}_{m_0}(\alpha))$ beyond the leading logarithmic term. More generally, these questions arise for other error metrics and classes of sequential procedures.

\begin{figure}[!t]
\centering
\includegraphics[width=0.95\textwidth]{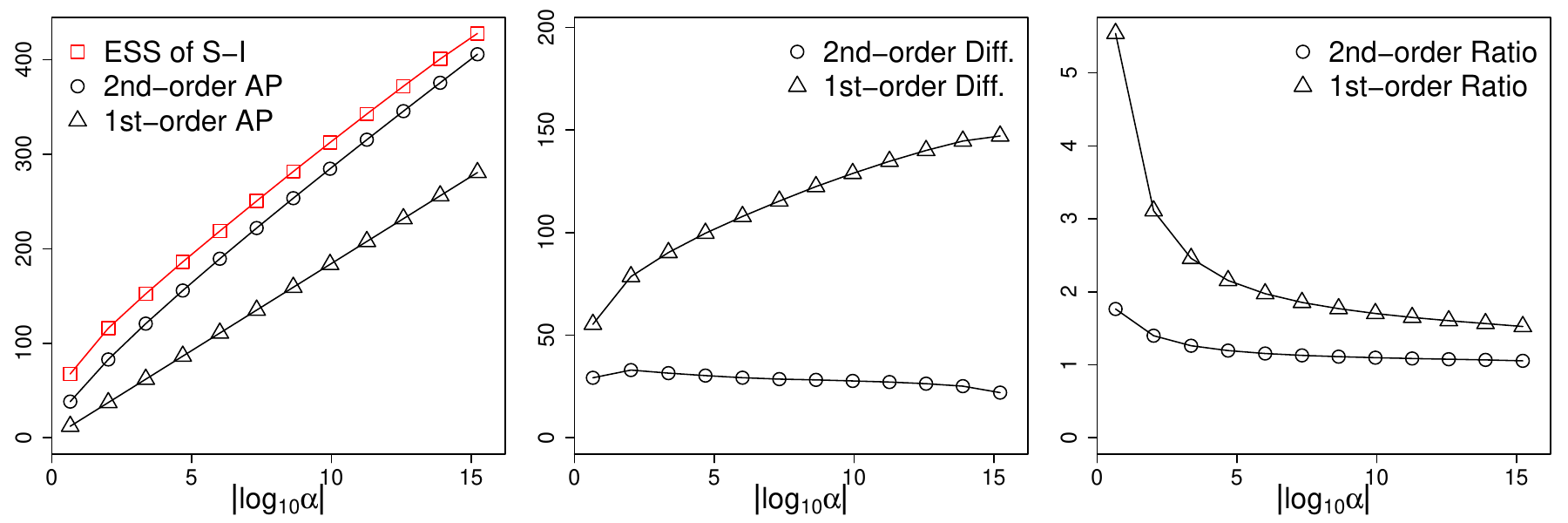}
\caption{Symmetric case: $K=20, m_0=1$. The x-axis in all panels is $|\log_{10}\alpha|$. In the leftmost plots, ``S-I'' denotes the Sum-Intersection rule (see \eqref{eq: sum intersection rule}) and ``AP'' denotes the asymptotic approximation to the smallest achievable ESS. In the middle plots, ``Diff.'' denotes the difference between the ESS of the S-I rule and the two approximations. In the rightmost plots, ``Ratio'' denotes the ratio of the ESS of the S-I rule to each of the two approximations.}
\label{fig: intro 1}
\end{figure}

\smallskip
\noindent \textbf{Our contributions.}
This paper establishes second-order frequentist optimality for several sequential multiple testing procedures and provides a second-order characterization of the minimal achievable ESS under a broad range of error metrics, including the generalized misclassification rate, generalized familywise error rates, false discovery/non-discovery rates, and settings with structural information on the signal configuration.

First, we develop in Theorem \ref{thm: main result} a general framework that provides explicit and verifiable sufficient conditions under which a family of rules is second-order asymptotically optimal. 
Recall that for each $\bell \in (0,1)^r$, $\Delta(\bell)$ denotes a class of sequential procedures. 
A family $\{\delta(\bell) = (T(\bell), D(\bell)) \in \Delta(\bell)\}$ is said to be \emph{second-order asymptotically optimal} with respect to $\{\Delta(\bell)\}$ if, for each $A \in \mathcal{A}$,
\[
\expt_A[T(\bell)]
=
T_A^{\min}(\Delta(\bell)) + O(1)
\quad \text{as } \bell \to \bd{0}.
\]
That is, the excess ESS remains uniformly bounded as the error tolerance level vanishes.

Our framework builds on the Bayesian optimality theory developed in 
\cite{kiefer1963asymptotically,lorden1967integrated,lorden1977nearly}. 
In the associated Bayesian formulation, the signal configuration $A$ is endowed with a prior distribution, and performance is measured by an integrated risk consisting of a sampling cost and a decision loss, defined through a suitably chosen cost parameter and loss function tailored to the class $\Delta(\bell)$. 
We show that if, for each $\bell \in (0,1)^r$, the stopping time $T(\bell)$ does not exceed that of an appropriately constructed Bayes rule almost surely, and if the induced integrated loss due to incorrect decisions is uniformly controlled over $\Delta(\bell)$ at the relevant asymptotic scale, then the procedure is second-order asymptotically optimal in the frequentist sense.

Second, we derive in Theorem~\ref{thm: min ESS characterization} a second-order accurate asymptotic expansion of $T_A^{\min}(\Delta(\bell))$, which characterizes the fundamental efficiency limit within this class.
Under suitable conditions, we show that 
\[
T_A^{\min}(\Delta(\bell))
=
\frac{\log(c_{\bell}^{-1})}{\kappa_A}
+
\kappa_A'
\sqrt{\log(c_{\bell}^{-1})}
+
O\left((\log(c_{\bell}^{-1}))^{1/4+\epsilon}\right),\quad \text{ as }\; \bell \to \bd{0},
\]
for any fixed $\epsilon \in (0,1/4)$, where the cost parameter $c_{\bell}>0$ is a function of $\bell$, and $\kappa_A$ and $\kappa_A'$ are constants independent of $\bell$. 
This refines the existing first-order approximation 
\[
T_A^{\min}(\Delta(\bell))
=
\frac{\log(c_{\bell}^{-1})}{\kappa_A}(1+o(1)),
\]
by identifying the explicit second-order correction term $\kappa_A'\sqrt{\log(c_{\bell}^{-1})}$. 
This term arises from a boundary-crossing problem for a multidimensional random walk. 
Our analysis relies on nonlinear renewal theory \cite{nagai2006nonlinear}, and appears to be new in the present context; see Remark \ref{rk:draglin_issue}.

Finally, we apply the above framework to several sequential multiple testing procedures previously known to be first-order asymptotically optimal under various error metrics and information structures, and establish their second-order optimality. In addition, we specialize the second-order asymptotic approximation of the minimal achievable ESS for the corresponding classes of procedures. 

For concreteness, consider the class $\Delta^{\textup{gmr}}_{m_0}(\alpha)$ and the Sum-Intersection rule $\delta_S(b_{\alpha})$ mentioned above. In the left panel of Fig. \ref{fig: intro 1}, the circle markers represent the second-order asymptotic approximation of $T_A^{\min}(\Delta^{\textup{gmr}}_{m_0}(\alpha))$ for $A=\emptyset$, which has the form
${|\log(\alpha)|}/{\kappa_A} + \kappa_A' \sqrt{|\log(\alpha)|}$, 
where $\kappa_A' > 0$ is independent of $\alpha$. The middle and right panels display the corresponding difference and ratio between $\expt_A[T_S(b_{\alpha}^*)]$ and this second-order approximation, respectively. The boundedness of the excess ESS as $\alpha \to 0$, shown in the middle panel, confirms the second-order optimality of $\{\delta_S(b_{\alpha})\}$ established in this paper.

\smallskip
\noindent \textbf{Organization.} Section \ref{sec: problem formulation} introduces the sequential multiple testing framework, including the error metrics, notions of asymptotic optimality, and the asymptotic characterization of the minimal achievable ESS. 
Section \ref{sec: bayesian formulation} presents the main theoretical results: we formulate the associated Bayesian problem, establish the link between second-order Bayesian and frequentist optimality, and derive the second-order expansion of the minimal ESS. 
Section \ref{sec: application} applies these results to several problem instances and classes of sequential procedures. 
Section \ref{sec:numerical_results} contains numerical studies under the generalized misclassification rate, and
section \ref{sec:conclusion} concludes the paper. Omitted proofs are given in the appendices, together with additional discussions and numerical results.

\section{Frequentist formulation, Optimality, and Minimum ESS}
\label{sec: problem formulation}
Consider $K \ge 2$ independent streams of random vectors $\{\mathbf{X}^k : k \in [K]\}$, where $[K] = \{1,\ldots,K\}$. 
For each $k \in [K]$,   $\mathbf{X}^k = \{\mathbf{X}_t^k \in \mathcal{X}^k : t = 1,2,\ldots\}$ is an independent and identically distributed (i.i.d.) sequence of $\mathcal{X}^k$-valued random vectors, where $\mathbf{X}_t^k$ denotes the observation from the $k$-th stream at time $t$. 
At each time $t \ge 1$, define
$$
\mathbf{X}_t^{[K]} := (\mathbf{X}_t^1,\ldots,\mathbf{X}_t^K)^\top
\in \mathcal{X}^1 \times \cdots \times \mathcal{X}^K
$$
to be the vector of observations across all streams at time $t$, and let
$
\mathcal{F}_t := \sigma(\mathbf{X}_1^{[K]},\ldots,\mathbf{X}_t^{[K]})
$
denote the $\sigma$-field generated by the observations up to time $t$.

For each $k \in [K]$, let $\mu^k$ be a $\sigma$-finite measure, and let $f^k$ denote the density of the marginal distribution of $X_1^k$ with respect to $\mu^k$. We consider the pair of simple hypotheses
\begin{equation}\label{def:simple_vs_simple}
H_0^k: f^k = f_0^k
\quad \text{versus} \quad
H_1^k: f^k = f_1^k.
\end{equation}
We impose the following assumption, which requires that the Kullback-Leibler divergences between $f_0^k$ and $f_1^k$ are positive and that the log-likelihood ratio has a finite second moment under both the null and the alternative.

\begin{assumption}\label{assumption: lorden}
For each stream $k \in [K]$, we assume that for $(i,i') \in \{(0,1),(1,0)\}$,
$$
\int \log\left(f_i^k/f_{i'}^k\right) f_i^k \, d\mu^k > 0,
\quad \text{and} \quad
\int \left[\log\left(f_i^k/f_{i'}^k\right)\right]^2 f_i^k \, d\mu^k < \infty.
$$
\end{assumption}

In addition, for each $k \in [K]$, let $P_0^k$ and $P_1^k$ denote the corresponding distributions of the \emph{data stream} $\mathbf{X}^k$ under $H_0^k$ and $H_1^k$, respectively. For any subset $A \subset [K]$, interpret $A$ as the set of indices for which the alternative hypothesis is true, referred to as the \emph{signal subset}, and let $\prob_A$ denote the joint distribution of $\{\mathbf{X}^k : k \in [K]\}$ under this configuration. In particular, by independence across streams,
\begin{equation*} 
\prob_A = \bigotimes_{k \in [K]} P^k,
\quad \text{ where } \quad 
P^k =
\begin{cases}
P_1^k, & k \in A,\\
P_0^k, & k \in A^c,
\end{cases}
\end{equation*}
where $\bigotimes$ denotes the product measure. Denote by $\expt_A$ the expectation with respect to the measure $\prob_A$.  
Under $\prob_A$, the distribution of $\mathbf{X}_1^{[K]}$ admits a density with respect to  $\mu^1 \otimes \cdots \otimes \mu^K$,  given by
\begin{equation}\label{eq: density}
f_A(x^1,\ldots,x^K)
=
\prod_{i \in A} f_1^i(x^i)
\prod_{j \in A^c} f_0^j(x^j),
\quad \text{ for } \;\;
x^k \in \mathcal{X}^k,\ k \in [K].
\end{equation}

Let $\mathcal{A} \subset 2^{[K]}$ denote the collection of all possible signal subsets, where $2^{[K]}$ is the power set of $[K]$. The case $\mathcal{A} = 2^{[K]}$ corresponds to the absence of any prior information about the signal configuration, whereas $\mathcal{A} \subsetneq 2^{[K]}$ reflects the availability of prior knowledge on the structure of the signals.

A sequential multiple testing procedure $\delta$ is defined as a pair $(T,D)$, where $T$ is an $\{\mathcal{F}_t\}$-stopping time at which sampling in all streams is terminated, and $D = (D^1,\ldots,D^K)^\top$ is an $\mathcal{F}_T$-measurable decision rule. Each component $D^k$, $k \in [K]$, is a Bernoulli random variable, with $D^k = 0$ corresponding to the selection of $H_0^k$ and $D^k = 1$ corresponding to the selection of $H_1^k$. Accordingly, upon stopping and based on the information available up to time $T$, we declare the presence of a signal (resp.~noise) in the $k$-th stream if $D^k = 1$ (resp.~$D^k = 0$). With a slight abuse of notation, we also write
\begin{equation*}
D := \{k \in [K] : D^k = 1\}
\end{equation*}
to denote the subset of streams declared to contain signals. We restrict attention to decision rules satisfying $D \in \mathcal{A}$, and denote by $\Delta$ the class of \emph{all} sequential testing procedures.

Our objective is to minimize the expected sample size (ESS) $\expt_A[T]$ for all $A \in \mathcal{A}$, while simultaneously controlling prescribed notions of error associated with the decision rule $D$.

\subsection{Multiple Testing Error Rates}\label{sec: multi testing error}

In this subsection, we define several subclasses of sequential testing procedures corresponding to different notions of error metrics and information structures.

Given a decision $D \in \mathcal{A}$ and the true signal subset $A \in \mathcal{A}$, there are two natural types of errors associated with each stream $k \in [K]$: a Type-I error occurs if the true hypothesis for the $k$-th stream is $H_0^k$ but $D^k = 1$ (i.e., a false positive), while a Type-II error occurs if the true hypothesis is $H_1^k$ but $D^k = 0$ (i.e., a false negative).

We first recall three subclasses of procedures from the literature that assume no prior information, i.e., $\mathcal{A} = 2^{[K]}$, and control different multiple testing error criteria.

\begin{problem}\label{problem: gmr}
Given a user-specified integer $1\leq m_0 \leq K$ and a tolerance level $\alpha \in (0,1)$, we define the class of sequential procedures that control the generalized misclassification rate (GMR)   \cite{li2014universal,malloy2014sequential,song2019sequential}  as 
\begin{equation*}
\mathcal{A} = 2^{[K]}, \quad \text{ and } \quad 
\Delta_{m_0}^{\mathrm{gmr}}(\alpha)
=
\left\{
(T,D) \in \Delta :
\max_{A \in \mathcal{A}} \prob_A\bigl(|D \triangle A| \ge m_0\bigr)
\le \alpha
\right\}.
\end{equation*}
\end{problem}

\begin{problem}\label{problem: gfr}
Given two user-specified integers $1 \leq m_1,m_2 \leq K$ such that $m_1+m_2 \leq K$ and tolerance levels $\alpha,\beta \in (0,1)$, we define the class of sequential procedures that control the generalized familywise error rates (GFWER)  \cite{bartroff2018multiple,de2015sequential,song2019sequential}, as $\mathcal{A} = 2^{[K]}$ and
\begin{equation*}
\Delta_{m_1,m_2}^{\mathrm{gfr}}(\alpha,\beta)
=
\left\{
(T,D) \in \Delta :
\max_{A \in \mathcal{A}} \prob_A\bigl(|D \setminus A| \ge m_1\bigr) \le \alpha,
\;
\max_{A \in \mathcal{A}} \prob_A\bigl(|A \setminus D| \ge m_2\bigr) \le \beta
\right\}.
\end{equation*}
\end{problem}

\begin{problem}\label{problem: fdr}
Given tolerance levels $\alpha,\beta \in (0,1)$, we define the class of sequential procedures that control the false discovery rate (FDR) and the false non-discovery rate (FNR)   \cite{javanmard2018online,bartroff2020sequential,he2021asymptotically} as $\mathcal{A} = 2^{[K]}$ and
\begin{equation*}
\Delta^{\mathrm{fdr}}(\alpha,\beta)
=
\left\{
(T,D) \in \Delta :
\max_{A \in \mathcal{A}} \expt_A\left[\frac{|D \setminus A|}{\max\{|D|,1\}}\right] \le \alpha,
\;
\max_{A \in \mathcal{A}} \expt_A\left[\frac{|A \setminus D|}{\max\{|D^c|,1\}}\right] \le \beta
\right\}. 
\end{equation*}
\end{problem}

A few remarks are in order. First, for each of the three classes above, error control is required uniformly over all possible signal subsets $A \subset [K]$. Second, in contrast to Class \ref{problem: gmr}, the GFWER and FDR/FNR criteria distinguish between Type-I and Type-II errors. Third, for GFWER, when $m_1 = m_2 = 1$, the criterion reduces to the classical familywise error rates \cite{bartroff2010multistage,de2012sequential,de2012step,bartroff2014sequential}; more generally, GFWER allows up to $m_1 - 1$ Type-I errors and $m_2 - 1$ Type-II errors. A similar discussion applies to the \emph{generalized} misclassification rate. Finally, Class \ref{problem: fdr} focuses on controlling the expected proportions of Type-I and Type-II errors.

In some applications, prior knowledge about the structure of the signal subset may be available. For instance, the exact number of signals may be known, or lower and upper bounds on the number of signals may be specified. We next recall a subclass of sequential procedures from the literature designed for settings in which $\mathcal{A} \subsetneq 2^{[K]}$  \cite{song2017asymptotically,he2021asymptotically}.

\begin{problem}\label{problem: kns}
Suppose the number of signals is known to be $m \in [1,K)$. For a tolerance level $\alpha \in (0,1)$, we define the class of sequential procedures that control the probability of a wrong decision as
\begin{equation}\label{eq: problem kns}
\mathcal{A} = \{A \subset [K] : |A| = m\},\qquad
\Delta_m^{\textup{kns}}(\alpha)
=
\left\{
(T,D) \in \Delta :
\max_{A \in \mathcal{A}} \prob_A\bigl(D \neq A\bigr)
\le \alpha
\right\}.
\end{equation}
\end{problem}

We review additional subclasses of sequential procedures in Appendix \ref{app:further}. For a given subclass, the objective is to identify a procedure within the class that minimizes the ESS in an appropriate asymptotic sense. In the next subsection, we review several notions of asymptotic optimality.

\subsection{Asymptotic Optimality} \label{subsec:asymptic_optimality}
Let $\Delta(\bell) \subset \Delta$ denote a general subclass of procedures indexed by an error tolerance vector $\bell \in (0,1)^r$, where $r \in \{1,2\}$. For example, $\Delta(\bell)$ may correspond to one of the four classes of procedures defined above, with $r = 1$ for Class \ref{problem: gmr} and $r = 2$ for Class \ref{problem: gfr}. 
Recall from \eqref{def: small ESS} that $T_A^{\textup{min}}(\Delta(\bell))$ denotes the minimum achievable expected sample size (ESS) over the class $\Delta(\bell)$ under the true signal subset $A \in \mathcal{A}$.


For many classes of sequential procedures, such as Classes \ref{problem: gmr}--\ref{problem: kns}, the minimum achievable ESS diverges as the error probability requirements $\bell$ become more stringent, that is,
\begin{align}\label{def:diverge}
T_A^{\textup{min}}(\Delta(\bell))\to\infty \;\; \text{ as } \;\;\bell\to\bd{0}, \quad \text{for each } A\in\mathcal{A},
\end{align}
where $\bd{0}$ denotes the zero vector in $\mathbb{R}^r$. We now formally define two notions of asymptotic optimality.


\begin{definition}[Asymptotic Optimality]\label{def: optimal}
Let $\{\delta(\bell) = (T(\bell), D(\bell)): \bell \in (0,1)^r\}$ be a collection of sequential multiple testing procedures such that $\delta(\bell) \in \Delta(\bell)$ for all $\bell \in (0,1)^r$. Then:
\begin{enumerate}[label = (\alph*)]
\item $\{\delta(\bell)\}$ is called \emph{first-order asymptotically optimal} in $\{\Delta(\bell)\}$ if, for each $A \in \mathcal{A}$,
\begin{equation*}
\lim_{\bell \to \boldsymbol{0}}
\frac{\expt_A[T(\bell)]}
{T_A^{\textup{min}}(\Delta(\bell))}
= 1.
\end{equation*}

\item $\{\delta(\bell)\}$ is called \emph{second-order asymptotically optimal} in $\{\Delta(\bell)\}$ if, for each $A \in \mathcal{A}$,
\begin{equation*}
\limsup_{\bell \to \boldsymbol{0}}
\Bigl(
\expt_A[T(\bell)]
-
T_A^{\textup{min}}(\Delta(\bell))
\Bigr)
< \infty.
\end{equation*}
\end{enumerate}
\end{definition}

We emphasize that the definition of asymptotic optimality (AO) requires the stated limits to hold for all $A \in \mathcal{A}$. While first-order AO ensures that the ESS is asymptotically optimal in a relative sense, it does not preclude the possibility that, for some $A \in \mathcal{A}$, the excess
$
\expt_A[T(\bell)]
-
T_A^{\textup{min}}(\Delta(\bell))
$
diverges as $\bell \to \boldsymbol{0}$; see Figure \ref{fig: intro 1} for an illustration. When \eqref{def:diverge} holds, the notion  of second-order AO imposes stronger requirements by controlling the excess in an additive sense.  We refer to Subsection \ref{subsec:discussion} for further discussion of higher-order optimality.

For Classes \ref{problem: gmr}--\ref{problem: kns} and additional classes in Appendix \ref{app:further}, the existing literature proposes testing procedures and establishes their first-order asymptotic optimality \cite{song2017asymptotically,song2019sequential,he2021asymptotically}; see Section \ref{sec: application} and Appendix \ref{app:kns}. However, it remains open whether these first-order optimal procedures also enjoy the stronger notion of second-order optimality. In this work, we provide an \emph{affirmative} answer to this question.

The above formulations are frequentist in nature, aiming to control the relevant error probabilities uniformly over $A \in \mathcal{A}$, as well as the ESS for each $A \in \mathcal{A}$.  To establish the \emph{second-order} optimality of existing procedures, our strategy is to leverage optimality results from associated \emph{Bayesian} formulations, which average errors and ESS with respect to a prior distribution over $\mathcal{A}$, and then translate these results into their frequentist counterparts. To this end, in Subsection \ref{subsec:Bayes_formulation} we introduce the Bayesian formulation and develop a framework linking Bayesian and frequentist optimality, which we then apply to general subclasses of sequential procedures.

\subsection{Minimum ESS Characterization}\label{subsec:mim_ESS}
The second-order optimality results make it possible to derive a more accurate asymptotic approximation to the minimum ESS than what is currently available in the literature.
Next, we define two notions of asymptotic approximation.

\begin{definition}[Asymptotic approximation]\label{def: optimal approximation}
Let $A \in \mathcal{A}$, and $\BellSP \subset (0,1)^r$. Let $\{Q_{1}(\bell), Q_2(\bell): \bell \in \BellSP\}$ be a collection of positive numbers indexed by $\bell$.
\begin{enumerate}[label=(\alph*)]
\item $\{Q_1(\bell): \bell \in \BellSP\}$ is called a \emph{first-order accurate approximation} to $T_A^{\textup{min}}(\Delta(\bell))$ if, as $\bell \to \boldsymbol{0}$ with $\bell \in \BellSP$,
\begin{equation*}
T_A^{\textup{min}}(\Delta(\bell)) = Q_1(\bell)\bigl(1+o(1)\bigr).
\end{equation*}

\item $\{Q_1(\bell), Q_2(\bell): \bell \in \BellSP\}$ is called a \emph{second-order accurate approximation} to $T_A^{\textup{min}}(\Delta(\bell))$ if, as $\bell \to \boldsymbol{0}$ with $\bell \in \BellSP$,
\begin{equation*}
T_A^{\textup{min}}(\Delta(\bell)) = Q_1(\bell) + Q_2(\bell)\bigl(1+o(1)\bigr),
\qquad \text{and} \qquad Q_2(\bell)=Q_1(\bell)o(1).
\end{equation*}
\end{enumerate}
\end{definition}

\begin{remark}
In later applications, when $r=1$, we let $\BellSP=(0,1)$. When $r=2$, we may impose conditions on the way the two components of $\bell$ approach zero. For this reason, we introduce $\BellSP$ in the above definition.
\end{remark}

From the above definition, if $\{Q_1(\bell)\}$ is first-order accurate, then $T_A^{\textup{min}}(\Delta(\bell))$ is characterized up to an error term that is negligible compared to $Q_1(\bell)$. Thus, $Q_1(\bell)$ is the first-order dominant term in the asymptotic approximation.

On the other hand, if $\{Q_1(\bell), Q_2(\bell): \bell \in \BellSP\}$ is second-order accurate, the error of the approximation $Q_1(\bell)+Q_2(\bell)$ is negligible compared to $Q_2(\bell)$, which is more precise than an $o(Q_1(\bell))$ error since $Q_2(\bell)=Q_1(\bell)o(1)$. This also implies that $Q_2(\bell)$ is the second-order dominant term in the asymptotic approximation, which justifies the terminology.

As detailed in Section \ref{sec: application}, for various classes of sequential procedures there exist first-order accurate approximations in the literature, which we strengthen to second order. To give some preview, for now assume $r=1$ and denote $\bell$ by $\alpha$. For several classes, the existing first-order approximation to $T_A^{\textup{min}}(\Delta(\alpha))$ is of the form $Q_1(\alpha)=|\log(\alpha)|/\kappa_A$, where $\kappa_A > 0$ is a constant. In this work, we show that as $\alpha\to 0$,
\[
T_A^{\textup{min}}(\Delta(\alpha))
=
\frac{|\log(\alpha)|}{\kappa_A}
+
\kappa_A' \sqrt{|\log(\alpha)|}
+
O\left(|\log(\alpha)|^{\epsilon_0}\right),
\]
where $\kappa_A' \geq 0$ is another constant, and $\epsilon_0\in(1/4,1/2)$ is a fixed constant.

Note that when $\kappa_A' > 0$, the first- and second-order approximations have different forms. Furthermore, roughly speaking, the existing results are accurate up to $o(|\log(\alpha)|)$, while our results are accurate up to $O(|\log(\alpha)|^{\epsilon_0})$, where $\epsilon_0$ can be made arbitrarily close to $1/4$ under appropriate moment conditions. Finally, when $\kappa_A' = 0$, the approximation ${|\log(\alpha)|}/{\kappa_A}$ in fact achieves an $O(1)$ error.

\section{A Framework for Frequentist Optimality via Bayesian Formulations} \label{sec: bayesian formulation}

In this section, we first introduce a general Bayesian formulation of the sequential multiple testing problem   and review a procedure that is second-order optimal in the Bayesian sense (see Subsection \ref{subsec:Bayes_formulation}). We then establish sufficient conditions under which a procedure can be shown to be second-order optimal in the frequentist sense by comparison with this second-order Bayesian optimal procedure (see Subsection \ref{subsec:conversion}). Finally, we derive a second-order accurate asymptotic approximation to the minimum achievable expected sample size (ESS) under the frequentist formulation (see Subsection \ref{subsec:second-order ESS}).

\subsection{Bayesian Formulation and Bayesian Optimality}
\label{subsec:Bayes_formulation}
Recall that $\mathcal{A} \subset 2^{[K]}$ denotes the collection of all possible signal subsets, and let $\pi_0$ be the uniform distribution on $\mathcal{A}$, that is,  $\pi_0(A) = |\mathcal{A}|^{-1}$ for each $A \in \mathcal{A}$. 
In the Bayesian formulation, the true signal subset $A$ is treated as a random variable drawn from the prior distribution $\pi_0$.

Let $\delta = (T,D) \in \Delta$ be a sequential procedure, and recall that the decision rule $D$ takes values in $\mathcal{A}$. We define the integrated risk of $\delta$ as the sum of the cost due to sampling and the loss due to incorrect decisions.  
Specifically, let $c > 0$ denote the cost incurred per additional set of observations. The integrated sampling cost is defined as
\begin{equation}\label{eq: int ess def}
\ess(\delta; c)
:= c \,\expt[T]
= c \sum_{A \in \mathcal{A}} \pi_0(A)\,\expt_A[T],
\end{equation}
where the expectation $\expt$ is taken with respect to both the prior distribution of $A$ and the randomness of the observations. In addition, let
\begin{equation*}
\mW(\cdot \mid \cdot): \mathcal{A} \times \mathcal{A} \to [0,\infty)
\end{equation*}
be a nonnegative loss function, where $\mW(D \mid A)$ denotes the loss incurred by declaring $D$ when the true signal subset is $A$. We assume that $\mW(A \mid A) = 0$ for all $A \in \mathcal{A}$. The integrated error loss associated with incorrect decisions is then defined as
\begin{equation}\label{eq: integrated error}
\ie(\delta; \mW)
:= \expt\left[\mW(D \mid A)\right]
= \sum_{A \in \mathcal{A}} \pi_0(A)
\sum_{d \in \mathcal{A}} \prob_A(D = d)\,\mW(d \mid A).
\end{equation}

Thus, given $c > 0$ and $\mW$, the integrated cost is defined as
\begin{equation}\label{eq: integrated risk}
\begin{aligned}
\ir(\delta; c, \mW) := \ess(\delta; c) + \ie(\delta; \mW).
\end{aligned}
\end{equation}

\begin{remark}
To connect frequentist optimality with Bayesian optimality, it is necessary to specify the cost parameter $c$ and the loss function $\mW$. 
\end{remark}

Next, we review the notion of asymptotic optimality in the Bayesian sense. Let
\begin{equation}\label{smallest_Bayes_risk}
\ir_{\textup{min}}(c,\mW)
:=
\inf_{\delta \in \Delta}\, \ir(\delta; c, \mW)
\end{equation}
denote the minimum achievable integrated risk.

\begin{definition}[Second-order Bayesian optimality]\label{def: bayesian optimality}
Let $\{\delta(c) : c > 0\} \subset \Delta$ be a collection of sequential procedures. The family $\{\delta(c)\}$ is called \emph{second-order Bayesian optimal} with respect to $\mW$ if there exists a constant $M > 0$ such that 
\begin{equation*}
\ir(\delta(c); c, \mW)
\le
\ir_{\textup{min}}(c, \mW) + M c,\quad \text{ for all } c > 0.
\end{equation*}
\end{definition}
We defer discussion of this definition to Remark \ref{remark:bayesian_optimality}.

Finally, we review a second-order Bayesian optimal procedure from the literature. Recall the definition of the density $f_A(\cdot)$ in \eqref{eq: density}. For $t \ge 1$, the posterior distribution $\pi_t : \mathcal{A} \to [0,1]$ of the signal subset $A$ given the $\sigma$-field $\mathcal{F}_t$ is defined as
\begin{equation}\label{eq: posterior}
\pi_t(A)
:=
\frac{f_{A,[t]}}{\sum_{B \in \mathcal{A}} f_{B,[t]}},
\quad \text{for each } A \in \mathcal{A},
\qquad
\text{where }
f_{A,[t]} := \prod_{s=1}^t f_A(\mathbf{X}_s^{[K]}).
\end{equation}

Given $c > 0$ and a loss function $\mW$, the following procedure $\delta_{\mathrm{Ld}}(c,\mW) = (T_{\mathrm{Ld}}(c,\mW), D_{\mathrm{Ld}}(c,\mW))$, studied in \cite{schwarz1962asymptotic,kiefer1963asymptotically,lorden1967integrated}, is defined as
\begin{equation}\label{eq: lorden procedure}
T_{\mathrm{Ld}} 
=
\min\Bigl\{
t \ge 1 :
\min_{D \in \mathcal{A}}\;
\sum_{A \in \mathcal{A}} \pi_t(A)\,\mW(D \mid A)
< c
\Bigr\},
\quad\; 
D_{\mathrm{Ld}} 
\in 
\argmin_{D \in \mathcal{A}}
\sum_{A \in \mathcal{A}} \pi_{T_{\mathrm{Ld}}}(A)\,\mW(D \mid A),
\end{equation}
where ties are broken arbitrarily, and the dependence on $(c,\mW)$ is omitted to simplify the notation. In words, the procedure stops at the first time $t$ when there exists a decision $D \in \mathcal{A}$ whose posterior expected loss falls below the threshold $c$, and upon stopping selects a decision that minimizes the posterior expected loss. It is shown in Lorden's paper \cite{lorden1967integrated} that the family $\{\delta_{\mathrm{Ld}}(c,\mW)\}$ is second-order Bayesian optimal, which motivates the subscript ``Ld''.

\begin{remark}\label{remark:bayesian_optimality}
In \cite{kiefer1963asymptotically}, it is shown that $\ir_{\textup{min}}(c, \mW)$ is of order $c|\log c|$ as $c \to 0$, and that the family $\{\delta_{\mathrm{Ld}}(c,\mW)\}$ is \emph{first-order} asymptotically optimal in the sense that
\begin{equation*}
\lim_{c \to 0}
\frac{\ir(\delta_{\mathrm{Ld}}(c,\mW); c, \mW)}
{\ir_{\textup{min}}(c, \mW)}
= 1.
\end{equation*}
The second-order optimality of $\{\delta_{\mathrm{Ld}}(c,\mW)\}$, as defined in Definition \ref{def: bayesian optimality} and established in \cite{lorden1967integrated}, is therefore a strictly stronger notion of optimality.
\end{remark}

\subsection{Second-Order Frequentist Optimality via Bayesian Optimality}\label{subsec:conversion}

In this subsection, we present one of the main results, which provides a technical scheme for establishing second-order frequentist optimality using Bayesian optimality as a tool. Specifically, for each $\bell \in (0,1)^{r}$ with $r \in \{1,2\}$, let $\Delta(\bell) \subset \Delta$ be a subclass of sequential multiple testing procedures. Recall the definition of $T_A^{\min}(\Delta(\bell))$ in \eqref{def: small ESS}, which is the minimum achievable ESS over this class.

Now, let $\delta_0(\bell) = (T_0(\bell),D_0(\bell))$, $\bell \in (0,1)^{r}$, be a collection of procedures such that $\delta_0(\bell) \in \Delta(\bell)$ for each $\bell$. Our first result establishes sufficient conditions under which $\{\delta_0(\bell)\}$ is second-order frequentist optimal with respect to $\{\Delta(\bell)\}$. The proof is given in Appendix \ref{sec: proof of thm: main result}.

\begin{theorem}\label{thm: main result}
Suppose that Assumption \ref{assumption: lorden} holds. Assume that there exists a cost $c_{\bell} > 0$ for each $\bell \in (0,1)^r$, a loss function $\mW : \mathcal{A} \times \mathcal{A} \to [0,\infty)$, and a constant $L > 0$ such that the following conditions hold:
\begin{align}
\label{assumption: lorden 1}
&\max_{A \in \mathcal{A}} \mW(D \mid A) > 0, \quad \text{for each } D \in \mathcal{A}; \\
\label{eq: second-order optimal sufficient condition 1}
&T_0(\bell) \le T_{\mathrm{Ld}}(c_{\bell}, \mW)\;\; \prob_A\text{-almost surely}, \;\; \text{for each } A \in \mathcal{A} \text{ and } \bell \in (0,1)^r; \\
\label{eq: second-order optimal sufficient condition 2}
&\ie(\delta; \mW) \le L c_{\bell} , \quad \text{for any procedure } \delta \in \Delta(\bell) \text{ and each } \bell \in (0,1)^r.
\end{align}
Then there exists a constant $M > 0$, depending only on $\mW$ and not on $\bell$, such that  
$$
\expt_A[T_0(\bell)] - T_A^{\textup{min}}(\Delta(\bell))
\le (2L + M)\,|\mathcal{A}|,  \quad \text{ for each } A \in \mathcal{A} \text{ and }  \bell \in (0,1)^r.
$$
\end{theorem}

Note that Theorem \ref{thm: main result} is non-asymptotic. Nevertheless, letting $\bell \to \bd{0}$ immediately implies the second-order frequentist optimality of $\{\delta_0(\bell)\}$ with respect to $\{\Delta(\bell)\}$, in view of Definition \ref{def: optimal}.

The key to applying Theorem \ref{thm: main result} is to specify a cost $c_{\bell}$ that depends on $\bell$, and a loss function $\mW(\cdot \mid \cdot)$ that is independent of $\bell$, such that the three conditions therein can be verified.
In particular, condition \eqref{assumption: lorden 1} requires that for each decision $D \in \mathcal{A}$, there exists a signal subset $A \in \mathcal{A}$ under which $D$ is incorrect, in the sense that $\mW(D \mid A) > 0$. Condition \eqref{eq: second-order optimal sufficient condition 1} requires that the stopping time $T_0(\bell)$ under consideration stops no later than $T_{\mathrm{Ld}}(c_{\bell}, \mW)$ defined in \eqref{eq: lorden procedure}. Finally, condition \eqref{eq: second-order optimal sufficient condition 2} requires that the integrated error loss of \emph{any} procedure in $\Delta(\bell)$ is bounded above by $L c_{\bell}$. In Section \ref{sec: application} and Appendix \ref{app:kns}, we apply Theorem \ref{thm: main result} to Classes \ref{problem: gmr}--\ref{problem: kns} to show that the existing procedures in the literature are second-order frequentist optimal.

\begin{remark}\label{rk:idea_discuss}
The idea of deriving frequentist optimality from Bayesian optimality has appeared previously in the sequential analysis literature. In particular, some arguments in the proof of Theorem \ref{thm: main result} are similar in spirit to those in the proof of Theorem~4 of \cite{lorden1977nearly}. However, the existing results are tailored to specific procedures and settings. Our contribution is to formulate explicit and verifiable sufficient conditions, stated in Theorem \ref{thm: main result}, under which second-order frequentist optimality follows from second-order Bayesian optimality. The theorem is presented at a level of generality that accommodates multiple error metrics, including Class \ref{problem: gmr} and Class \ref{problem: gfr}, and does not require the assumption $\prob_A(D(\bell)\neq A)\to 0$ as $\bell\to 0$ for all $A\in\mathcal{A}$, as assumed in \cite{lorden1977nearly}.
\end{remark}

\subsection{Second-Order Characterization of the Minimum ESS under the Frequentist Formulation} \label{subsec:second-order ESS}
In this subsection, we characterize $T_A^{\textup{min}}(\Delta(\bell))$ up to the second-order as $\bell \to \bd{0}$ for the family of subclasses $\Delta(\bell) \subset \Delta$ indexed by $\bell \in (0,1)^r$. We consider a general loss function $\mW:\mathcal{A}\times\mathcal{A}\to[0,\infty)$ and introduce some notation and assumptions.

 For each $k \in [K]$, define
\begin{equation}
\label{def:KLs_k}
\mathcal{I}_1^k
:=
\int \log \left({f_1^{k}}/{f_0^{k}}\right) f_1^{k} \, d\mu^k,
\quad
\mathcal{I}_0^k
:=
\int \log \left({f_0^{k}}/{f_1^{k}}\right) f_0^{k} \, d\mu^k,
\end{equation}
which are the Kullback--Leibler (KL) divergences between the densities $f_1^{k}$ and $f_0^{k}$.   Furthermore, for $A, C \in \mathcal{A}$ and given that the true signal subset is $A$, we denote the KL divergence between the densities $f_A$ and $f_C$ (see \eqref{eq: density}) by
\begin{equation}
\label{def:KL_fA_fC}
\KL(f_A \mid f_C)
:=
\expt_A\left[\log\left(\frac{f_A}{f_C}(\mathbf{X}_1^{[K]})\right)\right]
=
\sum_{k \in A \setminus C} \mathcal{I}_1^k
+
\sum_{k \in C \setminus A} \mathcal{I}_0^k.
\end{equation}

For each $D \in \mathcal{A}$,  denote the collection of signal subsets under which  $D$ is incorrect with respective to the loss function $\mW$ by
\begin{equation}
    \label{def:H_D_W}
\mathcal{H}_D^{\mW} := \{C \in \mathcal{A} : \mW(D \mid C) > 0\}.
\end{equation}
  In addition, for $A,D \in \mathcal{A}$,  define
\begin{equation}
\label{def:KL_A_D_star}
\KL_{A,D}^{\mW}
:=
\min_{C \in \mathcal{H}_D^{\mW}}
\KL(f_A \mid f_C), \quad \text{ and } \quad
\KL_{A,*}^{\mW} := \max_{B \in \mathcal{A}} \; \KL_{A,B}^{\mW}.
\end{equation}
That is, given the true signal subset $A$, $\KL_{A,D}^{\mW}$ is the minimum KL divergence between $f_A$ and $f_C$ over all signal subsets $C$ for which the decision $D$ is incorrect, and $\KL_{A,*}^{\mW}$ is the maximum of $\KL_{A,D}^{\mW}$ over all $D \in \mathcal{A}$.

\begin{definition} \label{def: dragalin 1}
Let $A \in \mathcal{A}$. We say that $A$ admits a \emph{unique most favorable subset} with respect to $\mW$ if
$\KL_{A,*}^{\mW}$ is attained at a unique subset $D_A^{\mW} \in \mathcal{A}$, that is,
$$
\KL_{A,D_A^{\mW}}^{\mW} = \KL_{A,*}^{\mW} > \KL_{A,D}^{\mW},
\quad \text{for all } D \in \mathcal{A} \setminus \{D_A^{\mW}\}.
$$
\end{definition}

\begin{remark}\label{remark: zero-one loss}
We call $\mW$ a \emph{zero-one loss} if $\mW(D \mid A) > 0$ for all $D \neq A$ with $A,D \in \mathcal{A}$; that is, the loss is positive unless $D = A$. If $\mW$ is a zero-one loss, then clearly $\mathcal{H}_D^{\mW} = \mathcal{A} \setminus \{D\}$. 
Now fix $A \in \mathcal{A}$. Under Assumption \ref{assumption: lorden}, we have
\[
\KL_{A,A}^{\mW} > 0,
\quad
\KL_{A,D}^{\mW} \le \KL(f_A \mid f_A) = 0,
\qquad \text{for all } D \neq A.
\]
As a result, under a zero-one loss $\mW$, each $A \in \mathcal{A}$ admits a \emph{unique most favorable subset} with respect to $\mW$, namely $D_A^{\mW} = A$.
\end{remark}

Let $A\in\mathcal{A}$ admit a \emph{unique most favorable subset} with respect to $\mW$. We denote by
\begin{equation}\label{eq: C_A}
\mathcal{C}_A^{\mW}:=\left\{C\in\mathcal{H}_{D_A^{\mW}}^{\mW}:\KL(f_A\mid f_C)=\KL_{A,*}^{\mW}\right\}  
\quad \text{ and } \quad
r_A^{\mW}:=\left|\mathcal{C}_A^{\mW}\right|
\end{equation}
the collection of signal subsets attaining the minimum KL divergence over $\mathcal{H}_{D_A^{\mW}}^{\mW}$ and its cardinality, respectively.
The resulting asymptotic behavior under $\prob_A$ differs qualitatively depending on whether $r_A^{\mW} = 1$ or $r_A^{\mW} \ge 2$. We refer to the former as the \emph{asymmetric} case and the latter as the \emph{symmetric} case.



Now, let $\left(\mathcal{C}_{A,1}^{\mW},\ldots,\mathcal{C}_{A,r_A^{\mW}}^{\mW}\right)$ be an ordered listing of the elements of $\mathcal{C}_A^{\mW}$. We define an $r_A^{\mW}$-dimensional random vector $\bd{R}_A^{\mW}$, where for each $1 \leq j\leq r_A^{\mW}$,
$$
\bd{R}_{A,j}^{\mW} :=  \log  f_A\left(\mathbf{X}_1^{[K]}\right) -\log f_{\mathcal{C}_{A,j}^{\mW}}\left(\mathbf{X}_1^{[K]}\right) - \KL_{A,*}^{\mW}.
$$
We note that $\expt_A(\bd{R}_{A}^{\mW})=\bd{0}$, and denote by
\begin{align} \label{def:Sigma_A_Z}
    \bSigma_{A}^{\mW} := \expt_A\left[ \bd{R}_{A}^{\mW}(\bd{R}_{A}^{\mW})^\top\right],
\end{align}
the $r_A^{\mW} \times r_A^{\mW}$ covariance matrix of $\bd{R}_{A}^{\mW}$ when $A$ is the true signal subset. Finally, let $\bd{Y}_{A}^{\mW}\sim \mathcal{N}(\mathbf{0},\bSigma_{A}^{\mW})$, the $r_A^{\mW}$-dimensional normal distribution with mean vector $\mathbf{0}$ and covariance matrix $\bSigma_{A}^{\mW}$. Denote the expected value of its largest component by
\begin{align}
\label{def:h_A}
h_{A}^{\mW}:=\expt_A\left[\max\left\{\bd{Y}_{A,1}^{\mW},\ldots,\bd{Y}_{A,r_A^{\mW}}^{\mW} \right\}\right].
\end{align}

\begin{remark}
When $r_A^{\mW}=1$, $h_{A}^{\mW}=0$. For the symmatric case $r_A^{\mW}\ge 2$, we have $h_{A}^{\mW} \geq 0$, which  can be estimated via a Monte Carlo method by sampling random vectors from $\mathcal{N}(\mathbf{0},\bSigma_{A}^{\mW})$ and taking the average of the largest component in each sample.
\end{remark}

In the following, we characterize the second-order approximation of the minimum achievable expected sample size (ESS) for $\Delta(\bell), \bell \in (0,1)^r$. For some $q\ge 2$ to be determined later, we assume 
\begin{equation}\label{eq: bounded moment assumption}
\int \bigl|\log(f_i^k / f_{i'}^k)\bigr|^q f_i^k \, d\mu^k < \infty,
\quad \text{ for } (i,i') \in \{(0,1),(1,0)\} \text{ and } k\in[K].   
\end{equation}

In addition, we consider the case where, for each $\bell\in(0,1)^r$, whether a procedure belongs to $\Delta(\bell)$ depends only on its decision rule and not on its stopping rule.

\begin{assumption}
\label{cond:class_depends_on_decision}
Let $\delta=(T,D)\in\Delta$ and $\delta'=(T',D')\in\Delta$ be arbitrary sequential procedures. If $\delta\in\Delta(\bell)$ for some $\bell\in(0,1)^r$ and $D=D'$, then $\delta'\in\Delta(\bell)$.
\end{assumption}

This assumption is satisfied by classes of sequential procedures, such as Class \ref{problem: gmr}--\ref{problem: kns}, which are defined in terms of error probabilities induced by incorrect decisions. The proof of the following theorem is provided in Appendix \ref{sec: proof of min ESS}.

\begin{theorem}\label{thm: min ESS characterization}
Suppose Assumptions \ref{assumption: lorden} and \ref{cond:class_depends_on_decision} hold. Let $\delta_0(\bell) = (T_0(\bell),D_0(\bell))$, $\bell \in (0,1)^{r}$, be a collection of procedures such that $\delta_0(\bell) \in \Delta(\bell)$ for each $\bell$. Assume that there exist a cost $c_{\bell} > 0$ for each $\bell \in (0,1)^r$, a loss function $\mW : \mathcal{A} \times \mathcal{A} \to [0,\infty)$, and a constant $L > 0$ such that  \eqref{assumption: lorden 1}, \eqref{eq: second-order optimal sufficient condition 1}, and \eqref{eq: second-order optimal sufficient condition 2} hold.  Furthermore, assume $c_{\bell} \to 0$ as $\bell \to \boldsymbol{0}$.

Let $A \in \mathcal{A}$, and assume that $A$ admits a unique most favorable subset with respect to $\mW$.

\begin{enumerate}[label=(\alph*), ref=\theassumption $(\alph*)$]
\item 
Consider the case $r_A^{\mW} = 1$. Assume \eqref{eq: bounded moment assumption} holds for 
$q=3$. Then, as $\bell\rightarrow \bd{0}$,
$$
T_A^{\textup{min}}(\Delta(\bell))= \frac{\log (c_{\bell}^{-1})}{\KL_{A,*}^{\mW}}+O(1).
$$
\item
Consider the case $r_A^{\mW} \geq 2$. For $\epsilon \in (0,1/2)$,  assume
\eqref{eq: bounded moment assumption} holds for some $q$ satisfying
\begin{equation}\label{eq: main text q_cond symmetric}
q > \epsilon^{-1} \quad \textup{ and } \quad  q\ge 3.
\end{equation}
Then, as $\bell\rightarrow \bd{0}$, 
$$
T_A^{\textup{min}}(\Delta(\bell))= \frac{\log (c_{\bell}^{-1})}{\KL_{A,*}^{\mW}}+\frac{h_{A}^{\mW}\sqrt{\log (c_{\bell}^{-1}})}{(\KL_{A,*}^{\mW})^{3/2}}+O\left((\log (c_{\bell}^{-1}))^{1/4+ \epsilon/2}\right).
$$
\end{enumerate}
\end{theorem}

A few remarks are in order. First, Assumptions \ref{assumption: lorden} and \eqref{eq: bounded moment assumption} impose moment conditions on the log-likelihood ratios. In particular, the second part of Assumption \ref{assumption: lorden} is implied by \eqref{eq: bounded moment assumption} when $q\ge2$.  Second, both the condition $c_{\bell}\to0$ as $\bell\to0$ and Assumption \ref{cond:class_depends_on_decision} are natural and hold for most classes of procedures, in particular for Classes \ref{problem: gmr}--\ref{problem: bns fdr fnr} considered in this paper. Third, before applying Theorem \ref{thm: min ESS characterization}, we typically consider a collection of procedures $\delta_0(\bell)\in\Delta(\bell)$ for $\bell\in(0,1)^r$ and establish their second-order optimality via Theorem \ref{thm: main result} by verifying that \eqref{assumption: lorden 1}, \eqref{eq: second-order optimal sufficient condition 1}, and \eqref{eq: second-order optimal sufficient condition 2} hold for suitable choices of $c_{\bell}$, $\mathcal{W}$, and $L$. As this verification is already carried out in Theorem \ref{thm: main result}, no additional work is required at this stage to apply Theorem \ref{thm: min ESS characterization}.

Thus, to apply Theorem \ref{thm: min ESS characterization}, the key step is to verify that a given $A\in\mathcal{A}$ admits a unique most favorable subset with respect to $\mW$, which is tailored to some error metric of interest. In Section \ref{sec: application} and Appendix \ref{app:kns}, we show that for Classes \ref{problem: gmr}, \ref{problem: fdr}, and \ref{problem: kns}, each $A\in\mathcal{A}$ admits a unique most favorable subset for the corresponding $\mW$. For Class \ref{problem: gfr}, however, we establish via a counterexample that this property need not hold, and we further derive sufficient conditions under which it is ensured.

\begin{remark}
In the proof of Theorem \ref{thm: min ESS characterization}, we define a procedure $\delta_1(\bell)$ that shares the same stopping rule as $\delta_{\mathrm{Ld}}(c_{\bell}, \mW)$ and the same decision rule as $\delta_0(\bell)$.  This serves two purposes: (i) the shared decision rule allows us to avoid verifying that $\delta_{\mathrm{Ld}}(c_{\bell}, \mW) \in \Delta(\bell)$; (ii) the shared stopping rule ensures, by Theorem \ref{thm: main result}, that the minimum ESS over $\Delta(\bell)$ coincides with $\expt_A[T_{\mathrm{Ld}}(c_{\bell},\mW)]$ up to an $O(1)$ term. Consequently, the problem reduces to deriving a second-order approximation for the ESS of the specific rule $\delta_{\mathrm{Ld}}(c_{\bell},\mW)$.

Moreover, since $A$ admits a unique most favorable subset $D_A^{\mW}$, we show that $\expt_A[T_{\mathrm{Ld}}(c_{\bell},\mW)] = \expt_{A}[\tau_{A}^{\mW}(c_{\bell})] + O(1)$, where for $c > 0$,
\begin{equation*}
\tau_{A}^{\mW}(c) =\min\left\{t\ge1:\bd{R}_{t,G}\ge\log(c^{-1})\ \text{for all }G\in\mathcal{H}_{D_A^{\mW}}^{\mW}\right\},
\end{equation*}
and $\bd{R}_{t,G}:=\sum_{s=1}^t\log\left({f_A(\mathbf{X}_s^{[K]})}/{f_G(\mathbf{X}_s^{[K]})}\right)$. Thus, $\tau_{A}^{\mW}(c)$ is the first time at which all components of the multidimensional random walk $\{\bd{R}_t:t\ge1\}$ cross the level $\log(c^{-1})$. Finally, we apply nonlinear renewal theory as developed in \cite{nagai2006nonlinear} to approximate $\expt_A[\tau_{A}^{\mW}(c)]$ (see Appendix \ref{sec: Asymptotic ESS of stopping times associated with multidimensional random walks}).
\end{remark}

\begin{remark} \label{rk:draglin_issue}
For the asymmetric case $r_A^{\mW}=1$, Theorem 3.1 in \cite{dragalin2000multihypothesis} provides an asymptotic approximation to $\expt_A[\tau_{A}^{\mW}(c)]$ that is accurate up to an $o(1)$ term, under the assumption that the increment $\bd{R}_{1,G}$ is $\prob_A$-nonarithmetic for each $G$. In the present work, however, since we focus on second-order optimality, it suffices to characterize the error term up to $O(1)$, and the nonarithmetic condition is not required.

More importantly, in the symmetric case $r_A^{\mW}\ge 2$, the applicability of Theorem~3.3 in \cite{dragalin2000multihypothesis} appears questionable. The argument therein invokes nonlinear renewal theory via Theorem~3 of \cite{zhang1988nonlinear}. However, as explained in Appendix~\ref{sec: Theorem 3.3 in dragalin is not correct}, a key regularity condition required in \cite{zhang1988nonlinear} is not satisfied in the present setting. To our knowledge, a rigorous second-order expansion for $\expt_A[\tau_{A}^{\mW}(c)]$ in the symmetric case has not been established, and obtaining an $O(1)$-accurate expansion remains open.
\end{remark}

\section{Applications under Different Error Metrics}\label{sec: application}
In this section, we apply our main results to Classes \ref{problem: gmr}--\ref{problem: fdr}, which correspond to different error metrics and assume no prior knowledge, that is, $\mathcal{A}=2^{[K]}$. In Appendix \ref{app:kns}, we apply these results to Class \ref{problem: kns}, which assumes that the number of signals is known. Additional classes of sequential procedures are discussed in Appendix \ref{app:further}.

For each class, we first review a procedure from the literature that provides explicit error control and has been shown to achieve \emph{first-order} asymptotic optimality in the frequentist sense (see Definition \ref{def: optimal}). We then establish its \emph{second-order} optimality by applying Theorem \ref{thm: main result}. In addition, using Theorems \ref{thm: min ESS characterization}, we characterize the second-order minimal expected sample size (ESS) under each signal subset.  

Recall the densities $\{f_0^k, f_1^k : k \in [K]\}$ defined in \eqref{def:simple_vs_simple}. 
For each stream $k \in [K]$ and time $t \ge 1$, define the log-likelihood ratio (LLR) 
statistic up to time $t$ (that is, restricted to $\mF_t$) by
\begin{equation}\label{eq:LLR}
\lambda_t^k 
:= \log \frac{dP_1^k}{dP_0^k}\bigg|_{\mF_t}
= \sum_{s=1}^t \log \frac{f_1^k(\mathbf{X}_s^k)}{f_0^k(\mathbf{X}_s^k)}.
\end{equation}



\subsection{Generalized Misclassification Rate}\label{sec: gmr}
In this subsection, we focus on Class \ref{problem: gmr},  $\Delta_{m_0}^{\mathrm{gmr}}(\alpha)$, which controls the generalized misclassification rate below level $\alpha \in (0,1)$ and assumes no prior information, that is, $\mathcal{A} = 2^{[K]}$. 
We start by reviewing the Sum-Intersection rule, denoted by 
$\delta_S(b) = (T_S(b), D_S(b))$, proposed in 
\cite{song2019sequential}, where $b > 0$ is a threshold parameter.
For each $t \ge 1$,  denote by
\begin{equation}\label{eq: ordered statistics sum intersection}
    0 < \tlambda_t^{(1)} \le \cdots \le \tlambda_t^{(K)}
\end{equation}
the order statistics of the absolute LLRs
$\{ |\lambda_t^k| : k \in [K] \}$. 
Note that if $\lambda_t^{k}$ is small in absolute value, it provides 
weak evidence in favor of either the null or the alternative at time $t$. 
Then, for $b > 0$,
\begin{equation}\label{eq: sum intersection rule}
\begin{aligned}
T_S(b) 
&:= \min\left\{ t \ge 1 : \sum_{i=1}^{m_0} \tlambda_t^{(i)} \ge b \right\}, \quad
D_S(b) 
:= \left\{ k \in [K] : \lambda_{T_S(b)}^k > 0 \right\}.
\end{aligned}
\end{equation}
That is, the procedure stops at the first time when the sum of the 
$m_0$ least significant LLRs exceeds $b$, and upon stopping, the null 
hypothesis is rejected in every stream whose LLR is positive.

It is shown in Theorem 3.1 of \cite{song2019sequential} that, for $\alpha > 0$, if the threshold is selected as
\begin{equation}\label{eq: second-order optimal of sum intersection rule 0}
    b_{\alpha} = |\log \alpha| + \log \binom{K}{m_0},
\end{equation}
then $\delta_S(b_{\alpha}) \in \Delta_{m_0}^{\mathrm{gmr}}(\alpha)$. With this choice of threshold, the family $\{\delta_S(b_{\alpha})\}$ is first-order asymptotically optimal as $\alpha \to 0$ (see Theorem~3.3 in \cite{song2019sequential}).
The next theorem establishes the second-order optimality of $\{\delta_S(b_{\alpha})\}$. We include the proof here to illustrate the application of Theorem \ref{thm: main result}.

\begin{theorem}\label{thm: second-order optimal of sum intersection rule}
Suppose Assumption \ref{assumption: lorden} holds. For each $\alpha > 0$, let the threshold be chosen as $b_{\alpha}$ in \eqref{eq: second-order optimal of sum intersection rule 0}. Then the Sum-Intersection rule $\{\delta_S(b_{\alpha})\}$ is second-order asymptotically optimal in $\{\Delta_{m_0}^{\mathrm{gmr}}(\alpha)\}$ as $\alpha \to 0$, that is, for each $A \subset [K]$,
\begin{equation*}
\limsup_{\alpha \to {0}}
\Bigl(
\expt_A[T_{S}(b_{\alpha})]
-
T_A^{\textup{min}}\left(\Delta_{m_0}^{\mathrm{gmr}}(\alpha)\right) 
\Bigr)
< \infty.
\end{equation*}
\end{theorem}

\begin{proof}
Let $\pi_0$ be the uniform distribution on $2^{[K]}$, that is, $\pi_0(A) = 1/2^K$ for $A \subset [K]$. In addition, we define the following loss function: for $A,D \subset [K]$,
\begin{equation}\label{eq: loss gmr}
\begin{aligned}
&\mW(D\mid A)=
\begin{cases}
    1 &\text{ if }\;\; |D \triangle A|\ge m_0 \\
    0 &\text{else}
\end{cases}.
\end{aligned}
\end{equation}
That is, $W(D|A) =  1$ if and only if the decision $D$ differs with $A$ in at least $m_0$ components. Furthermore, we define
\begin{equation}\label{eq: second-order optimal of sum intersection rule 2}
    c_{\alpha} := (2^Ke^{b_{\alpha}})^{-1} = 2^{-K}\binom{K}{m_0}^{-1}\alpha,\quad
    L := 2^{K}\binom{K}{m_0}.
\end{equation}
Then the conclusion follows immediately from Theorem \ref{thm: main result} once we verify condition \eqref{assumption: lorden 1}--\eqref{eq: second-order optimal sufficient condition 2} hold with the above loss function $\mW$, cost $c_{\alpha}$ and constant $L$.

We start with condition \eqref{assumption: lorden 1}. Fix an arbitrary $D \subset [K]$. Let $A = D^c$, and then $D \triangle A = [K]$. Since $m_0 \leq K$ (see Definition \ref{problem: gmr}), we have $\mW(D|A) = 1$. Thus condition \eqref{assumption: lorden 1} holds.

The verification of condition \eqref{eq: second-order optimal sufficient condition 1} requires showing that $T_S(b_{\alpha})$ stops no later than $T_{\mathrm{Ld}}(c_{\alpha}; \mW)$. By definition, for each $\alpha \in (0,1)$, we have $b_{\alpha}= \log(2^{-K}c_{\alpha}^{-1})$ and $c_{\alpha} < 2^{-K}$. Then condition \eqref{eq: second-order optimal sufficient condition 1} follows immediately from Lemma \ref{lemma: stop earlier sum intersection} in the Appendix \ref{app: proof of gmr}.

Finally, we verify condition \eqref{eq: second-order optimal sufficient condition 2}.
 Fix an arbitrary procedure $\delta=(T,D)\in\Delta_{m_0}^{\text{gmr}}(\alpha)$.
 By the definition of Class \ref{problem: gmr}, we have
 $\prob_A(|A \triangle D|\ge m_0) \leq \alpha$ for each $A \subset [K]$. Then, due to 
 the definition of $\mW$, 
\begin{equation*}
    \ie(\delta; \mW) = \sum_{A\in\mathcal{A}}\frac{1}{2^K}\prob_A(|A \triangle D|\ge m_0) \le \sum_{A\in\mathcal{A}}\frac{\alpha}{2^K} = \alpha = L c_{\alpha}.
\end{equation*} 
Thus, condition \eqref{eq: second-order optimal sufficient condition 2} holds. The proof
 is complete.
\end{proof}

Next, we consider asymptotic approximations to the smallest achievable ESS,
$T_A^{\textup{min}}\left(\Delta_{m_0}^{\mathrm{gmr}}(\alpha)\right)$, for each $A \subset [K]$ as $\alpha \to 0$.
Specifically, Theorem~3.3 of \cite{song2019sequential} shows that for each $A \subset [K]$, the first-order performance satisfies
\begin{equation}
    \label{eq: first_order_gmr}
T_A^{\textup{min}}\left(\Delta_{m_0}^{\mathrm{gmr}}(\alpha)\right)
=
\frac{|\log \alpha|}{\KL_{A,A}^{\mW}}\,(1+o(1)),
\end{equation}
where $\KL_{A,A}^{\mW}$ is defined in \eqref{def:KL_A_D_star} with $\mW$ given in \eqref{eq: loss gmr}, namely,
\[
\KL_{A,A}^{\mW}
=
\min\{\KL(f_A \mid f_C) : C \subset [K], \;\; |A \triangle C| \ge m_0\}.
\]

In this work, we derive a more accurate second-order approximation by applying Theorem \ref{thm: min ESS characterization}. The main step is to verify that each $A\subset [K]$ admits a unique most favorable subset with respect to $\mW$ defined in \eqref{eq: loss gmr}, as established in the following lemma.

\begin{remark}
Note that when $m_0 \ge 2$, $\mW$ in \eqref{eq: loss gmr} is \emph{not} a zero--one loss as defined in Remark \ref{remark: zero-one loss}.
\end{remark}

Recall the definition of  $\KL_{A,D}^{\mW}$ and $\KL_{A,*}^{\mW}$ in \eqref{def:KL_A_D_star}.

\begin{lemma}\label{lemma: misclassification rate satisfies assumption}
Suppose Assumption \ref{assumption: lorden} holds.
Let $m_0 \in [1,K]$ be an integer, and let $\mW$ be defined in \eqref{eq: loss gmr}. For each $A \subset [K]$,
\[
\KL_{A,A}^{\mW} > \KL_{A,D}^{\mW}, \quad \text{for any } D \subset [K] \text{ such that } D \ne A.
\]
Thus, each $A \subset [K]$ admits a unique most favorable subset with respect to $\mW$  with $D_A^{\mW} = A$, and
$\KL_{A,*}^{\mW} = \KL_{A,A}^{\mW}$.
\end{lemma}
\begin{proof}
    See Appendix \ref{app: proof of gmr}.
\end{proof}

\begin{theorem}\label{thm: second-order min ESS gmr}
Suppose Assumption \ref{assumption: lorden} holds. 
Let $\mW$ be defined in \eqref{eq: loss gmr}, and let $A\subset [K]$ be an arbitrary signal subset.

\begin{enumerate}[label=(\alph*), ref=\theassumption $(\alph*)$]
\item 
Consider the asymmetric case $r_A^{\mW} = 1$. Assume \eqref{eq: bounded moment assumption} holds for $q=3$.
Then, as $\alpha\rightarrow 0$,  
$$
T_A^{\textup{min}}\left(\Delta_{m_0}^{\mathrm{gmr}}(\alpha)\right)= \frac{|\log \alpha|}{\KL_{A,A}^{\mW}}+O(1).
$$
\item 
Consider the symmetric case $r_A^{\mW} \geq 2$. Let $\epsilon \in (0,1/2)$, and assume
\eqref{eq: bounded moment assumption} holds for some $q$ satisfying \eqref{eq: main text q_cond symmetric}. Recall $h_{A}^{\mW}$ defined in \eqref{def:h_A}. Then, as $\alpha\rightarrow 0$,  
$$
T_A^{\textup{min}}\left(\Delta_{m_0}^{\mathrm{gmr}}(\alpha)\right)= \frac{|\log \alpha|}{\KL_{A,A}^{\mW}}+\frac{h_{A}^{\mW}\sqrt{|\log \alpha|}}{(\KL_{A,A}^{\mW})^{3/2}}+O\left((\log \alpha)^{1/4+ \epsilon/2}\right).
$$
\end{enumerate}
\end{theorem}
\begin{proof}
As mentioned above, with $b_{\alpha}$ in \eqref{eq: second-order optimal of sum intersection rule 0},
$\delta_S(b_{\alpha}) \in \Delta_{m_0}^{\mathrm{gmr}}(\alpha)$ for each $\alpha \in (0,1)$. Moreover, we have shown in the proof of Theorem \ref{thm: second-order optimal of sum intersection rule} that condition \eqref{assumption: lorden 1}--\eqref{eq: second-order optimal sufficient condition 2} hold with the loss function $\mW$ in \eqref{eq: loss gmr}, cost $c_{\alpha}$ in \eqref{eq: second-order optimal of sum intersection rule 2} and a constant $L$. It is clear that $c_{\alpha} \to 0$ as $\alpha \to 0$, and that Assumption \ref{cond:class_depends_on_decision} holds.

Finally, Lemma \ref{lemma: misclassification rate satisfies assumption} shows that each $A$ admits a unique favorable subset with $D^{\mW}_A = A$ and  $\KL_{A,*}^{\mW} = \KL_{A,A}^{\mW}$. The proof is then finished by Theorem \ref{thm: min ESS characterization}.
\end{proof}

\subsection{Generalized Familywise Error Rates} \label{sec: leap rule}

In this subsection, we focus on Class \ref{problem: gfr}, $\Delta_{m_1,m_2}^{\mathrm{gfr}}(\alpha,\beta)$, which controls the generalized familywise false positive and false negative rates below levels $\alpha$ and $\beta$, respectively, and assumes no prior information, that is, $\mathcal{A} = 2^{[K]}$.
We start by reviewing the Leap rule, denoted by 
$\delta_L(a,b) = (T_L(a,b), D_L(a,b))$, proposed in 
\cite{song2019sequential}, where $a, b > 0$ are threshold parameters.

Recall the LLRs $\{\lambda_t^k : k \in [K],\, t = 1,2,\ldots\}$ defined in \eqref{eq:LLR}. For each time $t \ge 1$, denote by
\begin{equation*}
0 < \hlambda_t^{(1)} \le \cdots \le \hlambda_t^{(p(t))}
\end{equation*}
the ordered statistics of the positive LLRs $\{\lambda_t^k : \lambda_t^k > 0,\, k \in [K]\}$, where $p(t)$ is the number of strictly positive LLRs at time $t$. Similarly, denote by
\begin{equation*}
0 < \clambda_t^{(1)} \le \cdots \le \clambda_t^{(K-p(t))}
\end{equation*}
the ordered statistics of the absolute values of the nonpositive LLRs $\{|\lambda_t^k| : \lambda_t^k \le 0,\, k \in [K]\}$.

We also denote the corresponding stream indices by $\{\hl_1(t),\ldots,\hl_{p(t)}(t)\}$ such that $\hlambda_t^{\hl_i(t)} = \hlambda_t^{(i)}$ for $i \in [p(t)]$, and by $\{\cl_1(t),\ldots,\cl_{K-p(t)}(t)\}$ such that $|\lambda_t^{\cl_j(t)}| = \clambda_t^{(j)}$ for $j \in [K-p(t)]$. For the statistic $\hlambda_t^{(i)}$, a smaller index $i$ indicates weaker evidence for rejecting the $k$-th null hypothesis (and thus a higher risk of a Type-I error). Similarly, for $\clambda_t^{(i)}$, a smaller index $i$ indicates weaker evidence for accepting the $k$-th null hypothesis (and thus a higher risk of a Type-II error).

For thresholds $a,b>0$, we first define a collection of rules as follows: for $i=0,\ldots,m_1-1$,  
\begin{align*}
\hat{\tau}_i := \min\left\{ t\ge 1:
\sum_{j=1}^{m_1-i}\hlambda_t^{(j)} \ge b,\ 
\sum_{j=i+1}^{i+m_2}\clambda_t^{(j)} \ge a \right\}, \;\; 
\hat{D}_i := \{k\in[K]: \lambda_{\hat{\tau}_i}^k>0\}\ \cup\ \{\cl_1(\hat{\tau}_i), \ldots, \cl_i(\hat{\tau}_i)\},
\end{align*}
and for $i=1,\ldots,m_2-1$,  
\begin{align*}
\check{\tau}_i := \min\left\{ t\ge 1:
\sum_{j= i+1}^{i+m_1}\hlambda_t^{(j)} \ge b,\ 
\sum_{j=1}^{m_2- i}\clambda_t^{(j)} \ge a \right\}, \;\; 
\check{D}_i := \{k\in[K]: \lambda_{\check{\tau}_i}^k>0\}\ \setminus\ \{\hl_1(\check{\tau}_i),\ldots, \hl_i(\check{\tau}_i)\}.
\end{align*}
The dependence on $(a,b)$ is suppressed for simplicity. Furthermore, we adopt the convention that $\hlambda_t^{k}=\infty$ for $k>p(t)$ and $\clambda_t^{k}=\infty$ for $k>K-p(t)$.

Then, the leap rule $\delta_L(a,b)$ with parameters $a,b>0$ is defined by
\begin{align}\label{eq: leap rule}
T_L(a,b) := \min\{\hat{\tau}_0,\cdots,\hat{\tau}_{m_1-1},\check{\tau}_1,\ldots,\check{\tau}_{m_2-1}\}, \quad
D_L(a,b) =
\begin{cases}
\hat{D}_i, & \text{if } \hat{\tau}_i = T_L(a,b),\\
\check{D}_i, & \text{if } \check{\tau}_i = T_L(a,b),
\end{cases}
\end{align}
where, if the minimum is attained by multiple stopping rules, $D_L$ is defined as the union of the corresponding decisions. We refer the interested reader to Section 4 of \cite{song2019sequential} for motivation and discussion of the Leap rule.

It is shown in Theorem 4.1 of \cite{song2019sequential} that for any $\alpha,\beta \in (0,1)$, if the thresholds are selected as
\begin{equation}\label{eq: leap_thresholds}
a_{\alpha,\beta} = |\log\beta| + \log \left[2^{m_2}\binom{K}{m_2}\right], \quad
b_{\alpha,\beta} = |\log\alpha| + \log \left[2^{m_1}\binom{K}{m_1}\right],
\end{equation}
then $\delta_L(a_{\alpha,\beta},b_{\alpha,\beta})$ belongs to the class $\Delta_{m_1,m_2}^{\mathrm{gfr}}(\alpha,\beta)$. Moreover, with this choice of thresholds, the family $\{\delta_L(a_{\alpha,\beta},b_{\alpha,\beta})\}$ is first-order asymptotically optimal as $\alpha \vee \beta \to 0$ (see Theorem~4.3 in \cite{song2019sequential}).

The next theorem establishes its second-order optimality when $\alpha$ and $\beta$ vanish at the same rate; specifically, we assume
\begin{equation}\label{eq:alpha_beta_same_rate}
C^{-1}\alpha \le \beta \le C\alpha, \quad \text{for some absolute constant } C > 0.
\end{equation}

\begin{theorem}\label{thm: second-order optimal of leap rule}
Suppose Assumption \ref{assumption: lorden} holds and $\alpha,\beta \in (0,1)$ satisfy condition \eqref{eq:alpha_beta_same_rate}. Then the Leap rule $\{\delta_L(a_{\alpha,\beta},b_{\alpha,\beta})\}$ is second-order asymptotically optimal in $\{\Delta_{m_1,m_2}^{\textup{gfr}}(\alpha,\beta)\}$ 
as $\alpha,\beta \to 0$,  that is, for each $A \subset [K]$,
\begin{equation*}
\limsup_{\alpha,\beta \to {0}}
\Bigl(
\expt_A[T_L(a_{\alpha,\beta},b_{\alpha,\beta})]
-
T_A^{\textup{min}}\left(\Delta_{m_1,m_2}^{\mathrm{gfr}}(\alpha,\beta)\right) 
\Bigr)
< \infty.
\end{equation*}
\end{theorem}

\begin{proof}
The proof is in Appendix \ref{sec: proof of gfr}.
\end{proof}

\begin{remark}\label{remark:gmr_same_rate}
For the first-order asymptotic optimality, $\alpha$ and $\beta$ may vanish at arbitrary relative rates; however, in the above theorem they are required to vanish at the same rate.
\end{remark}

In the proof of Theorem \ref{thm: second-order optimal of leap rule}, we apply Theorem \ref{thm: main result} with the following loss function: for $A,D \subset [K]$,
\begin{equation}\label{eq: loss gfr}
\begin{aligned}
\mW(D\mid A)=
\begin{cases}
1, & \text{ if }\;\; |D\setminus A|\ge m_1 \;\text{ or }\; |A\setminus D|\ge m_2,\\
0, & \text{otherwise}.
\end{cases}
\end{aligned}
\end{equation}
This equals $1$ if and only if the decision $D$ incurs at least $m_1$ Type-I errors or at least $m_2$ Type-II errors (or both) when the true signal subset is $A$.

Next, we consider asymptotic approximations to the smallest achievable ESS,
$T_A^{\textup{min}}\left(\Delta_{m_1,m_2}^{\mathrm{gfr}}(\alpha,\beta)\right)$, for each $A \subset [K]$ as $\alpha,\beta \to 0$ in such a way that \eqref{eq:alpha_beta_same_rate} holds.
Specifically, under \eqref{eq:alpha_beta_same_rate}, Theorem~4.3 of \cite{song2019sequential} shows that for each $A \subset [K]$, the first-order performance satisfies
\begin{equation*}
T_A^{\textup{min}}\left(\Delta_{m_1,m_2}^{\mathrm{gfr}}(\alpha,\beta)\right)
=
\frac{|\log \alpha|}{\KL_{A,*}^{\mW}}\,(1+o(1)),   
\end{equation*}
where $\KL_{A,*}^{\mW}$ is defined in \eqref{def:KL_A_D_star} with $\mW$ given in \eqref{eq: loss gfr}, that is,
\begin{align*}
    \KL_{A,*}^{\mW} = \max_{D \subset [K]} \KL_{A,D}^{\mW}, \text{ with }
    \KL_{A,D}^{\mW} = \min\{\KL(f_A|f_C): C \subset [K], \; |D\setminus C| \geq m_1 \text{ or } |C \setminus D| \geq m_2\}.
\end{align*}

In this work, we aim to derive a more accurate second-order approximation by applying Theorem \ref{thm: min ESS characterization}.
\begin{theorem}\label{thm: second-order min ESS gfr}
Suppose that Assumption \ref{assumption: lorden} holds, and that $\alpha,\beta \in (0,1)$ satisfy condition \eqref{eq:alpha_beta_same_rate}. Let $A \in \mathcal{A}$, and assume that $A$ admits a unique most favorable subset with respect to $\mW$ in \eqref{eq: loss gfr}.
\begin{enumerate}[label=(\alph*), ref=\theassumption $(\alph*)$]
\item 
Consider the asymmetric case $r_A^{\mW} = 1$. Assume \eqref{eq: bounded moment assumption} holds for $q=3$.
Then, as $\alpha\rightarrow 0$,  
$$
T_A^{\textup{min}}\left(\Delta_{m_1,m_2}^{\mathrm{gfr}}(\alpha,\beta)\right)= \frac{|\log (\alpha)|}{\KL_{A,*}^{\mW}}+O(1).
$$
\item 
Consider the symmetric case $r_A^{\mW} \geq 2$. Let $\epsilon \in (0,1/2)$, and assume
\eqref{eq: bounded moment assumption} holds for some $q$ satisfying \eqref{eq: main text q_cond symmetric}. Then, as $\alpha\rightarrow 0$,  
$$
T_A^{\textup{min}}\left(\Delta_{m_1,m_2}^{\mathrm{gfr}}(\alpha,\beta)\right)= \frac{|\log (\alpha)|}{\KL_{A,*}^{\mW}}+\frac{h_{A,*}^{\mW}\sqrt{|\log (\alpha)|}}{(\KL_{A,*}^{\mW})^{3/2}}+O\left((\log (\alpha))^{1/4+ \epsilon/2}\right).
$$
\end{enumerate}
\end{theorem}
\begin{proof}
As mentioned above, with $a_{\alpha,\beta}, b_{\alpha,\beta}$ in \eqref{eq: leap_thresholds}, $\delta_L(a_{\alpha,\beta}, b_{\alpha,\beta}) \in \Delta_{m_1,m_2}^{\mathrm{gfr}}(\alpha,\beta)$ for each $\alpha,\beta \in (0,1)$. Moreover, we show in the proof of Theorem \ref{thm: second-order optimal of leap rule} that condition \eqref{assumption: lorden 1}--\eqref{eq: second-order optimal sufficient condition 2} hold with the loss function $\mW$ in \eqref{eq: loss gfr}, and cost $c_{\alpha,\beta}$ and constant $L$ defined in \eqref{eq: leap_constants_def}. It is clear that $c_{\alpha,\beta} \to 0$ as $\alpha,\beta \to 0$, and that Assumption \ref{cond:class_depends_on_decision} holds. Since $A$ is assumed to admit a unique most favorable subset with respect to $\mW$, the proof is then finished by Theorem \ref{thm: min ESS characterization}.
\end{proof}

Unlike Class \ref{problem: gmr}, which controls the generalized misclassification rate (see Lemma \ref{lemma: misclassification rate satisfies assumption}), for $\mW$ defined in \eqref{eq: loss gfr} it is not guaranteed that each $A\subset[K]$ admits a unique most favorable subset in the sense of Definition \ref{def: dragalin 1}, as illustrated by the following counterexample.

\begin{example}
Let $K=3$ and $m_1=m_2=2$. Consider the completely homogeneous case
$\mathcal{I}_i^{j} = \mathcal{I} > 0$ for $i \in \{0,1\}$ and $j \in \{1,2,3\}$,
where the information numbers are defined in \eqref{def:KLs_k}. Now consider $A=\emptyset$. By the definition in \eqref{def:KL_A_D_star} and $\mW$ in \eqref{eq: loss gfr},
\[
\KL_{A,*}^{\mW} = 3\mathcal{I}
= \KL_{A,\{1\}}^{\mW}
= \KL_{A,\{2\}}^{\mW}
= \KL_{A,\{3\}}^{\mW}.
\]
Thus, the three subsets $\{1\}$, $\{2\}$, and $\{3\}$ attain the maximum in $\max_{D\subset[K]}\KL_{A,D}^{\mW}$.
\end{example}

Next, we provide sufficient conditions under which $A\subset[K]$ admits a unique most favorable subset under $\mW$ defined in \eqref{eq: loss gfr}.

\begin{lemma}\label{gfr:sufficient_conditions_unique}
Suppose Assumption \ref{assumption: lorden} holds. Then $A \subset [K]$ admits a unique most favorable subset with respect to $\mW$  in \eqref{eq: loss gfr} if one of the following conditions holds: 
\begin{equation*}
\begin{aligned}
&\textup{ (a) } m_1=m_2=1.
\\&\textup{ (b) } m_1=m_2:=m^*, \ \mathcal{I}_0^k=\mathcal{I}_1^k:=\mathcal{I}^* \textup{ for } k\in[K], \;\textup{ and }\; \min\{|A|,|A^c|\}\ge m^*.     
\end{aligned}   
\end{equation*}
In either case,  $D_A^{\mW}=A$, and $\KL_{A,*}^{\mW}=\KL_{A,A}^{\mW}$.
\end{lemma}
\begin{proof}
The proof is in Appendix \ref{sec: proof of gfr}.
\end{proof}
\begin{remark}
When $m_1=m_2=1$, corresponding to the classical familywise error rates, Theorem \ref{thm: second-order min ESS gfr} therefore provides a second-order accurate asymptotic approximation. However, for generalized familywise error rates, this remains an open problem in general.    
\end{remark}

\subsection{False Discovery and Non-discovery Rates}\label{sec: fdr}

In this subsection, we consider Class \ref{problem: fdr}, $\Delta^{\mathrm{fdr}}(\alpha,\beta)$, which controls the false discovery and false non-discovery rates, and assume $\mathcal{A}=2^{[K]}$. We start by reviewing the Intersection rule, denoted by 
$\delta_I(a,b) = (T_I(a,b), D_I(a,b))$, proposed in 
\cite{de2012sequential}, where $a, b > 0$ are threshold parameters. Specifically,
\begin{equation}
\label{def:intersection_rule}
    T_I(a,b) = \min\{t\ge 1: \lambda_t^k\notin (-a,b) \text{ for all }k\in[K]\},
    \qquad D_I(a,b) = \{k\in[K]: \lambda_{T_I(a,b)}^k>0\}.
\end{equation}
It is shown in Corollary 4.1 of \cite{he2021asymptotically} that for $\alpha,\beta \in (0,1)$, if
\begin{equation}
    \label{fdr:a_b}
    a_{\alpha,\beta} = |\log(\beta)| + \log(K),\quad b_{\alpha,\beta} = |\log(\alpha)|+ \log(K),
\end{equation}
then $\delta_I(a_{\alpha,\beta},b_{\alpha,\beta}) \in \Delta^{\mathrm{fdr}}(\alpha,\beta)$. Furthermore, by the same corollary, with this choice of thresholds, the family $\{\delta_L(a_{\alpha,\beta},b_{\alpha,\beta})\}$ is first-order asymptotically optimal as $\alpha \vee \beta \to 0$.

The next theorem establishes its second-order optimality when $\alpha$ and $\beta$ vanish at the same rate, that is, when \eqref{eq:alpha_beta_same_rate} holds.

\begin{theorem}\label{thm: second-order optimal of intersection rule}
Suppose Assumption \ref{assumption: lorden} holds,  and $\alpha,\beta \in (0,1)$ satisfy condition \eqref{eq:alpha_beta_same_rate}. Then the Intersection rule $\{\delta_I(a_{\alpha,\beta},b_{\alpha,\beta})\}$ is second-order asymptotically optimal in $\{\Delta^{\mathrm{fdr}}(\alpha,\beta)\}$ 
as $\alpha,\beta \to 0$,  that is, for each $A \subset [K]$,
\begin{equation*}
\limsup_{\alpha,\beta \to {0}}
\Bigl(
\expt_A[T_I(a_{\alpha,\beta},b_{\alpha,\beta})]
-
T_A^{\textup{min}}\left(\Delta^{\mathrm{fdr}}(\alpha,\beta)\right) 
\Bigr)
< \infty.
\end{equation*}
\end{theorem}
\begin{proof}
The proof is in Appendix \ref{sec: proof of fdr}. 
\end{proof}

\begin{remark}\label{rk:relationship}
In the proof of Theorem \ref{thm: second-order optimal of intersection rule}, we use the following relationship between Class \ref{problem: fdr} and Class \ref{problem: gfr} with $m_1=m_2=1$, which was first established in \cite{he2021asymptotically}:
\begin{equation*}
    \Delta_{1,1}^{\mathrm{gfr}}(\alpha,\beta)\subset \Delta^{\textup{fdr}}(\alpha,\beta)\subset \Delta_{1,1}^{\mathrm{gfr}}(K\alpha,K\beta).
\end{equation*}
Further, when $m_1 = m_2 = 1$, the Intersection rule $\delta_I(a,b)$ in \eqref{def:intersection_rule} coincides with the Leap rule $\delta_L(a,b)$ in \eqref{eq: leap rule}, and the thresholds $(a_{\alpha,\beta}, b_{\alpha,\beta})$ in \eqref{fdr:a_b} are the same as $(a_{\alpha,\beta}, b_{\alpha,\beta})$ in \eqref{eq: leap_thresholds}. 
\end{remark}

Next, we consider asymptotic approximations to the smallest achievable ESS,
$T_A^{\textup{min}}\left(\Delta^{\mathrm{fdr}}(\alpha,\beta)\right)$, for each $A \subset [K]$ as $\alpha,\beta \to 0$ in such a way that \eqref{eq:alpha_beta_same_rate} holds. Specifically, Corollary 4.1 of \cite{he2021asymptotically} shows that under \eqref{eq:alpha_beta_same_rate}, for each $A \subset [K]$, the first-order performance can be characterized by:
\begin{equation*}
T_A^{\textup{min}}\left(\Delta^{\mathrm{fdr}}(\alpha,\beta)\right)
=
\frac{|\log \alpha|}{\KL_{A,A}^{\mW}}\,(1+o(1)),   
\end{equation*}
where $\KL_{A,A}^{\mW}$ is defined in \eqref{def:KL_A_D_star} with $\mW$ given in \eqref{eq: loss gfr} with $m_1 = m_2 = 1$, namely,
\begin{align*}
    \KL_{A,A}^{\mW} = \min\{\KL(f_A|f_C): C \subset [K], \; |A\setminus C| \geq 1 \text{ or } |C \setminus A| \geq 1\}.
\end{align*}
In this work, we obtain a second-order approximation by first exploiting the relationship described in Remark \ref{rk:relationship} and then applying Theorem \ref{thm: second-order min ESS gfr} with $m_1=m_2=1$. Note that when $m_1=m_2=1$, each $A\subset[K]$ admits a unique most favorable subset under $\mW$ defined in \eqref{eq: loss gfr}.

\begin{theorem}\label{thm: second-order min ESS fdr}
Suppose that Assumption \ref{assumption: lorden} holds and that $\alpha,\beta \in (0,1)$ satisfy condition \eqref{eq:alpha_beta_same_rate}. Let $\mW$ be defined in \eqref{eq: loss gfr}, and $A \subset [K]$ be an arbitrary signal subset.

\begin{enumerate}[label=(\alph*), ref=\theassumption $(\alph*)$]
\item 
Consider the asymmetric case $r_A^{\mW} = 1$. Assume \eqref{eq: bounded moment assumption} holds for $q=3$.
Then, as $\alpha\rightarrow 0$,  
$$
T_A^{\textup{min}}\left(\Delta^{\mathrm{fdr}}(\alpha,\beta)\right)= \frac{|\log (\alpha)|}{\KL_{A,A}^{\mW}}+O(1).
$$
\item 
Consider the symmetric case $r_A^{\mW} \geq 2$. Let $\epsilon \in (0,1/2)$, and assume
\eqref{eq: bounded moment assumption} holds for some $q$ satisfying \eqref{eq: main text q_cond symmetric}. Then, as $\alpha\rightarrow 0$,  
$$
T_A^{\textup{min}}\left(\Delta^{\mathrm{fdr}}(\alpha,\beta)\right)= \frac{|\log (\alpha)|}{\KL_{A,A}^{\mW}}+\frac{h_{A,A}^{\mW}\sqrt{|\log (\alpha)|}}{(\KL_{A,A}^{\mW})^{3/2}}+O\left((\log (\alpha))^{1/4+ \epsilon/2}\right).
$$
\end{enumerate}
\end{theorem}
\begin{proof}
The proof follows from the preceding discussion.
\end{proof}

\begin{remark}
We emphasize that Theorem \ref{thm: second-order min ESS fdr} applies to every signal subset $A\subset[K]$. However,  $\alpha$ and $\beta$ need to vanish at the same rate. 
\end{remark}

\subsection{Discussion of Optimality}\label{subsec:discussion}
It is instructive to compare our second-order results with classical optimality results in related settings.

In the single stream case ($K=1$), Wald's sequential probability ratio test (SPRT) \cite{wald1945sequential}, which coincides with the Intersection rule in \eqref{def:intersection_rule}, is \emph{exactly} optimal in Class \ref{problem: gfr} with $m_1=m_2=1$, i.e., $\Delta_{1,1}^{\mathrm{gfr}}(\alpha,\beta)$. When the thresholds $(a,b)$ are chosen so that the Type-I and Type-II error probabilities equal $(\alpha,\beta)$, the SPRT simultaneously minimizes $\expt_0[T]$ and $\expt_1[T]$ over all sequential tests satisfying the same error constraints; see, e.g., \cite{tartakovsky2014sequential,wald1948optimum}. Thus, in this one stream setting, the minimal achievable ESS is attained exactly by the SPRT. 


Next, we review the classical multi-hypothesis testing problem 
\cite{lorden1977nearly,draglia1999multihypothesis,dragalin2000multihypothesis}, 
which is closely related to the multiple-stream setting considered in this paper.
In this formulation, one observes i.i.d.\ data $\{X_t\}_{t\ge1}$ whose common distribution belongs to a finite collection $\{f_i : i \in [M]\}$ of competing $\mu$-densities. A sequential procedure consists of a stopping time $T$ and a terminal decision $D \in [M]$. Error control is typically expressed either through the family of error probabilities $\{\mathbb{P}_i(D=j) : i \neq j \in [M]\}$ or through weighted risks of the form $\{\sum_{j\in[M]} \omega_{i,j}\mathbb{P}_i(D=j) : i \in [M]\}$, where $\omega_{i,j}\ge 0$ are fixed weights and $\mathbb{P}_i$ denotes the distribution under hypothesis $i$.

In contrast to the single hypothesis case, exact optimality is generally unavailable once $M\ge3$. Lorden \cite{lorden1977nearly} proposed the \emph{matrix SPRT}, a multidimensional extension of Wald’s SPRT, and established its \emph{third-order} asymptotic optimality under each $\mathbb{P}_i$, $i\in[M]$. More precisely, with appropriately chosen thresholds, the ESS of the matrix SPRT differs by $o(1)$ from the minimal achievable ESS among procedures whose error probabilities are no larger than those of the matrix SPRT, either in terms of the full collection $\{\mathbb{P}_i(D=j)\}$ or the weighted risks.

It is important to note that this optimality result holds within a restricted comparison class of procedures satisfying $\mathbb{P}_i(D=j)\to 0$ for all $i\neq j$ as the error control becomes more stringent. Such a property does not hold for several of the multiple testing classes considered in this paper, including Class \ref{problem: gmr} when $m_0>1$, Class \ref{problem: gfr} when $m_1+m_2>2$, and Class \ref{problem: fdr}. The procedures studied in the previous subsections, while computationally efficient and second-order optimal, are not expected to achieve third-order optimality within their respective classes.

\section{Numerical Results under Generalized Misclassification Rate}\label{sec:numerical_results}

In this section, we present a numerical study for Class \ref{problem: gmr}, $\Delta_{m_0}^{\mathrm{gmr}}(\alpha)$, which controls the generalized misclassification rate at level $\alpha\in(0,1)$ without assuming prior information. The goal is to corroborate the results in Section \ref{sec: gmr}.


Specifically, recall that for the Sum-Intersection rule in \eqref{eq: sum intersection rule}, if the threshold is chosen as $b_{\alpha}$ in \eqref{eq: second-order optimal of sum intersection rule 0}, then $\delta_S(b_{\alpha})\in\Delta_{m_0}^{\mathrm{gmr}}(\alpha)$ for each $\alpha\in(0,1)$, and by Theorem \ref{thm: second-order optimal of sum intersection rule} the family $\{\delta_S(b_{\alpha})\}$ is second-order optimal in $\{\Delta_{m_0}^{\mathrm{gmr}}(\alpha)\}$ as $\alpha \to 0$.

As discussed in \cite{song2019sequential}, the threshold $b_{\alpha}$ provides strong theoretical guarantees for the Sum-Intersection rule in terms of error control, but it is often conservative in practice. In our numerical study, for each given $\alpha\in(0,1)$, we first determine via simulation a threshold $b_{\alpha}^*$ such that the maximal error probability equals $\alpha$, that is, $\max_{B \subset [K]}\prob_B\bigl(|D_S(b_{\alpha}^*)\ \triangle \ B|\ge m_0\bigr)=\alpha$. We then compute, again by simulation, the ESS under the empty signal subset, namely $\expt_{\emptyset}[T_S(b_{\alpha}^*)]$. Note that $b_{\alpha}^*$ is a non-conservative critical value. To estimate the error probabilities, we employ importance sampling; see \cite{song2016logarithmically,song2025efficient}.

Recall the densities $\{f_0^k,f_1^k:k\in[K]\}$  in \eqref{def:simple_vs_simple} and the information numbers $\{I_0^k,I_1^k:k\in[K]\}$ in \eqref{def:KLs_k}. 
We consider the problem of testing Normal means, where for the $k$th stream,
\[
H_0^{k}:\ f_0^k=\mathcal{N}(0,\sigma_k^2)\qquad \text{versus}\qquad H_1^{k}:\ f_1^k=\mathcal{N}(\mu_k,\sigma_k^2).
\]
For each $k \in [K]$, $I_1^{k} = I_0^{k} = \mu_k^{2}/(2\sigma_k^2)$.

We consider the loss function $\mW$ defined in \eqref{eq: loss gmr}. By Lemma \ref{lemma: misclassification rate satisfies assumption}, each $A\subset[K]$ admits a unique most favorable subset with respect to $\mW$, with $D_A^{\mW}=A$ and $\KL_{A,*}^{\mW}=\KL_{A,A}^{\mW}$. Moreover, recall the definition of $\mathcal{C}_A^{\mW}$  and $r_A^{\mW}=|\mathcal{C}_A^{\mW}|$ in \eqref{eq: C_A}. Here we consider a symmetric case with $r_A^{\mW}\ge 2$, while Appendix \ref{app:addition_simulation} presents an asymmetric case with $r_A^{\mW}=1$.

Specifically, we consider the following setup
$$\mu_1=\cdots=\mu_K = \mu,\quad \text{ and } \quad  \sigma_1^2 = \cdots = \sigma_K^2 = \sigma^2,$$
where $\mu \in \mathbb{R}$ and $\sigma^2 > 0$. By definition \eqref{def:KL_A_D_star} and \eqref{eq: C_A},
\begin{align*}
\KL_{\emptyset,*} = m_0\frac{\mu^2}{2\sigma^2},\quad 
\left\{\mathcal{C}_{\emptyset,j}^{\mW}:\; 1 \leq j\leq r_{\emptyset}^{\mW}\right\} = \{B \subset [K]:\;|B|=m_0\}, \quad \text{ and }\quad r_\emptyset^{\mW} = \binom{K}{m_0}.
\end{align*}
Recall the definition of $\bSigma_{\emptyset}^{\mW}$ in \eqref{def:Sigma_A_Z}. For $1 \leq i, j \leq r_{\emptyset}^{\mW}$, by elementary calculation,
\begin{align*}
   (\bSigma_{\emptyset}^{\mW})_{ij} = \cov_{\emptyset}\left(\sum_{k \in \mathcal{C}_{\emptyset,i}^{\mW}} \log\frac{f_1^{k}(\mathbf{X}_1^{k})}{f_0^{k}(\mathbf{X}_1^{k})},\; \sum_{k \in \mathcal{C}_{\emptyset,j}^{\mW}} \log\frac{f_1^{k}(\mathbf{X}_1^{k})}{f_0^{k}(\mathbf{X}_1^{k})},\right) = \left|\mathcal{C}_{\emptyset,i}^{\mW}\cap \mathcal{C}_{\emptyset,j}^{\mW}\right|\frac{\mu^2}{\sigma^2},
\end{align*}
where we denote by $\cov_{\emptyset}$ the covariance with respect to the probability $\prob_{\emptyset}$. Thus, $h_{\emptyset}^{\mW}$ defined in \eqref{def:h_A} is the expected maximum of a Gaussian random vector with distribution $\mathcal{N}(\bd{0},\bSigma_{\emptyset}^{\mW})$, and it can be computed via Monte Carlo simulation. Since it depends on $m_0$ through $\mW$, we denote it by $h_{\emptyset}^{(m_0)}$ instead. 
By \eqref{eq: first_order_gmr} and Theorem \ref{thm: second-order min ESS gmr}(b), the first- and second-order asymptotic approximations for the smallest achievable ESS of Class \ref{problem: gmr} are, respectively,
\begin{align*}
\textup{FO}_{\alpha}:=
\frac{2\sigma^2}{m_0\mu^2}|\log\alpha|,\quad \text{ and } \quad 
\textup{SO}_{\alpha}:=\frac{2\sigma^2}{m_0\mu^2}|\log\alpha|
+\left(\frac{2\sigma^2}{m_0\mu^2}\right)^{3/2} h_{\emptyset}^{(m_0)}\sqrt{|\log\alpha|}.
\end{align*}


In Fig. \ref{fig: intro 1} and Fig. \ref{fig:numerical_fig}, we consider three cases: $(K,m_0)\in\{(20,1),(20,2),(50,2)\}$. For each case, we present three plots. The leftmost plot displays the ESS $\expt_{\emptyset}[T_S(b_{\alpha}^*)]$ (square markers), together with the first-order approximation $\textup{FO}_{\alpha}$ (triangle markers) and the second-order approximation $\textup{SO}_{\alpha}$ (circle markers), as functions of $|\log_{10}(\alpha)|$ as $\alpha\in(0,1)$ varies. The middle plot shows the differences $\expt_{\emptyset}[T_S(b_{\alpha}^*)]-\textup{FO}_{\alpha}$ and $\expt_{\emptyset}[T_S(b_{\alpha}^*)]-\textup{SO}_{\alpha}$ as functions of $|\log_{10}(\alpha)|$. The rightmost plot presents the ratios $\expt_{\emptyset}[T_S(b_{\alpha}^*)]/\textup{FO}_{\alpha}$ and $\expt_{\emptyset}[T_S(b_{\alpha}^*)]/\textup{SO}_{\alpha}$ as functions of $|\log_{10}(\alpha)|$.

As established in Theorem \ref{thm: second-order optimal of sum intersection rule}, the Sum-Intersection rule with the non-conservative threshold $b_{\alpha}^*$ is second-order optimal. Accordingly, as predicted by Theorem \ref{thm: second-order min ESS gmr}(b) and evident from the plots, the second-order approximation $\textup{SO}_{\alpha}$ is substantially more accurate than the first-order approximation $\textup{FO}_{\alpha}$. In particular, the difference $\expt_{\emptyset}[T_S(b_{\alpha}^*)]-\textup{FO}_{\alpha}$ diverges as $\alpha\to0$. Moreover, although Theorems \ref{thm: second-order optimal of sum intersection rule} and \ref{thm: second-order min ESS gmr}(b) only imply that $\expt_{\emptyset}[T_S(b_{\alpha}^*)]-\textup{SO}_{\alpha}=O(|\log(\alpha)|^{1/4+\epsilon})$ for any fixed $\epsilon \in (0,1/4)$, the middle panels suggest that this difference remains bounded, i.e., is $O(1)$. It would be of interest to establish this refinement rigorously. Moreover, $K$ is fixed in our analysis. From Fig. \ref{fig:numerical_fig}, the second-order approximation appears more accurate for smaller $K$. A systematic study of the impact of $K$ is left for future work.

\begin{figure}[!t]
\centering
\begin{subfigure}[b]{\textwidth}
\includegraphics[width=0.95\textwidth]{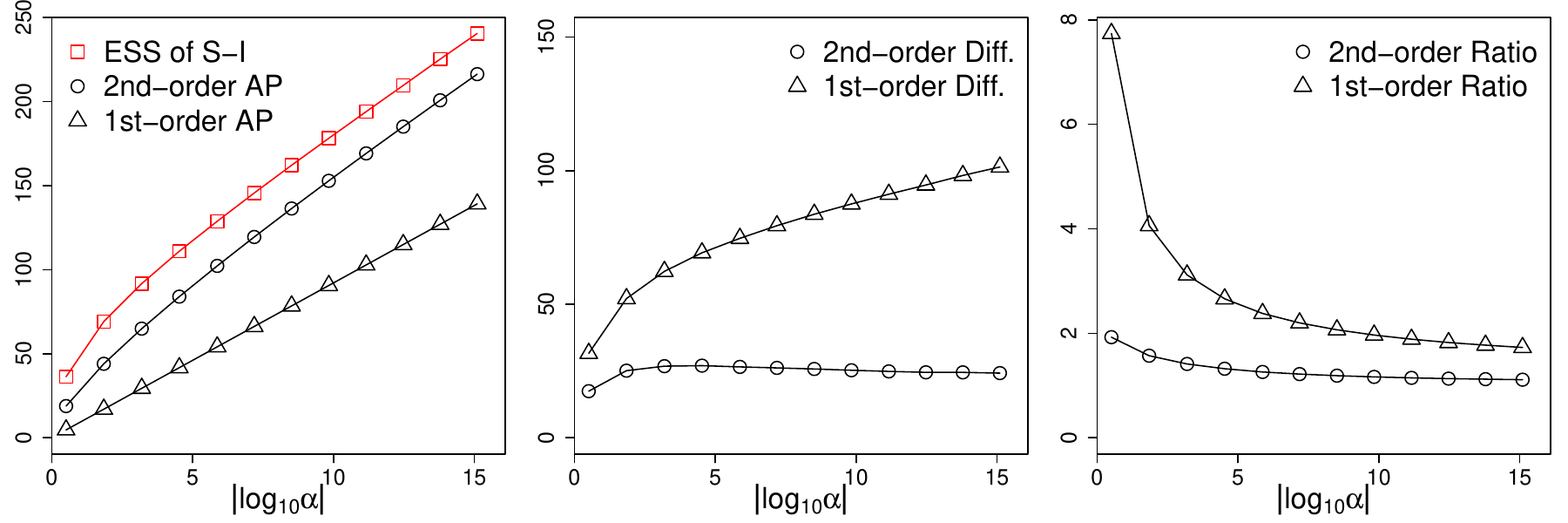}
\caption{Symmetric Case: $K=20, m_0=2$}
\end{subfigure}
\begin{subfigure}[b]{\textwidth}
\includegraphics[width=0.95\textwidth]{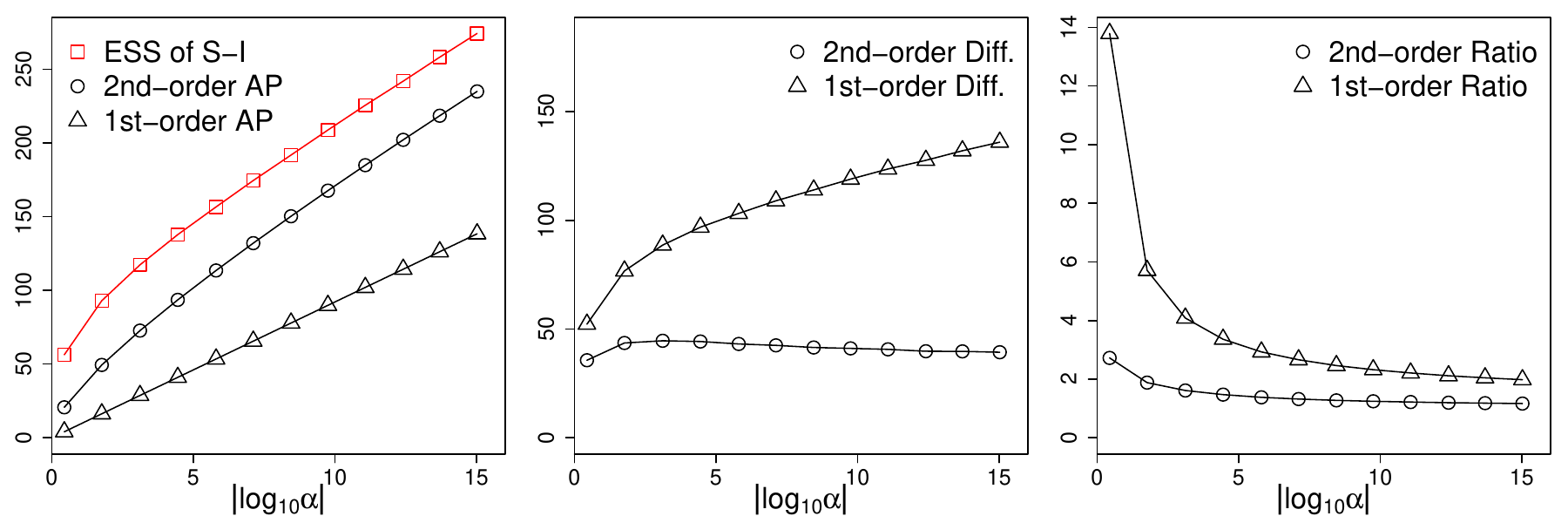} 
\caption{Symmetric Case: $K=50, m_0=2$}
\end{subfigure}
\caption{The x-axis in all panels is $|\log_{10}\alpha|$. In the leftmost plots, ``S-I'' denotes the Sum-Intersection rule and ``AP'' denotes the asymptotic approximation  to the smallest achievable ESS. In the middle plots, ``Diff.'' denotes the difference between the ESS of the S-I rule and the two approximations. In the rightmost plots, ``Ratio'' denotes the ratio of the ESS of the S-I rule to each of the two approximations.}
\label{fig:numerical_fig}
\end{figure}

\section{Conclusions and Future Work} \label{sec:conclusion}
We establish a unified second-order asymptotic theory for sequential multiple testing with simultaneous stopping across independent data streams. Through an associated Bayesian formulation with a properly specified cost parameter and loss function, we provide sufficient conditions under which second-order Bayesian optimality implies second-order frequentist optimality. Under these conditions, several procedures previously known to be first-order optimal under various error metrics and information structures are in fact second-order optimal: for every signal configuration, their excess ESS over the minimal achievable one remains uniformly bounded as the error tolerances vanish.

We also obtain a second-order asymptotic expansion of the minimal achievable ESS. Under suitable moment conditions and a uniqueness condition on the most favorable alternative induced by the loss, we identify the second-order correction to the leading logarithmic term. Its form depends on whether the associated boundary-crossing problem is asymmetric or symmetric. These results sharpen the classical first-order theory.

Several directions remain open. First, for all classes considered in this paper except Class \ref{problem: gfr}, each signal subset $A \in \mathcal{A}$ admits a unique most favorable subset under a properly chosen loss function $\mW$. For Class \ref{problem: gfr}, we provide a counterexample showing that this uniqueness may fail and give sufficient conditions under which it holds. A complete second-order characterization of the minimal achievable ESS for Class \ref{problem: gfr}  remains open.

Second, in the symmetric case, the remainder term in the ESS expansion is of order $O((\log(c_{\bell}^{-1}))^{\epsilon})$ for arbitrary $\epsilon \in (1/4,1/2)$. Numerical evidence suggests that the approximation error may in fact remain bounded, indicating that sharper analysis could improve the remainder bound. Third, for several classes, including Class \ref{problem: gfr} and Class \ref{problem: fdr}, first-order asymptotic optimality does not require the error tolerances (e.g., $\alpha$ and $\beta$) to converge to zero at the same rate (see condition \eqref{eq:alpha_beta_same_rate}). Extending the second-order theory to such more general regimes is of interest. Finally, it would be natural to investigate extensions to composite hypotheses and to settings with dependent data streams as in \cite{chaudhuri2024joint}.

\bibliography{ref.bib}       
\newpage
\begin{appendix}

\section{Proof of Theorem \ref{thm: main result}} \label{sec: proof of thm: main result}
Recall that $\pi_0$ is the uniform distribution on $\mathcal{A}$, and that $\ie(\cdot;\cdot)$ denotes the integrated error loss in \eqref{eq: integrated error}. Recall the definition of $\delta_{\mathrm{Ld}}(c,\mW) = (T_{\mathrm{Ld}}(c,\mW), D_{\mathrm{Ld}}(c,\mW))$ in \eqref{eq: lorden procedure}, $\ir(\delta;c,\mW)$ in \eqref{eq: integrated risk} and   $\ir_{\textup{min}}(c,\mW)$ in \eqref{smallest_Bayes_risk}. 

\begin{proof}[Proof of Theorem \ref{thm: main result}]
Fix $\bell \in (0,1)^r$. Let $\delta(\bell) = (T(\bell), D(\bell)) \in \Delta(\bell)$ be an arbitrary procedure in $\Delta(\bell)$.

By part (b) of Lemma \ref{lemma: thm 4 in lordon} ahead, there exists a constant $M > 0$ not depending on $\bell$, such that for any signal subset $A\in\mathcal{A}$, 
\begin{equation}
    \label{dis:theorem 1}
\expt_A[T_0(\bell)]-\expt_A[T(\bell)]\leq
\left( \frac{\ie(\delta_0(\bell); \mW) + \ie(\delta(\bell); \mW)}{c_{\bell}} + 
M \right)|\mathcal{A}|.
\end{equation}
 Since $\delta_0(\bell), \delta(\bell) \in \Delta(\bell)$, due to condition \eqref{eq: second-order optimal sufficient condition 2}, we have
\begin{align*}
    \expt_A[T_0(\bell)]-\expt_A[T(\bell)] \leq (2L+M) |\mathcal{A}|.
\end{align*}
Since $\delta(\bell)$ is an arbitrary procedure in $\Delta(\bell)$, it follows that
\begin{align*}
   \expt_A[T_0(\bell)]- T_A^{\textup{min}}(\Delta(\bell))\leq (2L+M) |\mathcal{A}|. 
\end{align*}
The proof is complete.
\end{proof}

The following lemma is non-asymptotic.

\begin{lemma}\label{lemma: thm 4 in lordon}
Suppose Assumption \ref{assumption: lorden} holds. Let $c>0$ be a cost,   and $\mW(\cdot\mid\cdot):\mathcal{A}\times\mathcal{A}\mapsto [0,\infty)$ be a loss function satisfying condition \eqref{assumption: lorden 1}. If a procedure $\delta_0 =(T_0,D_0) \in \Delta$ satisfies 
$$T_0 \le T_{\mathrm{Ld}}(c,\mW) \text{ ($i.e.$ $\delta_0$ stops earlier than $\delta_{\mathrm{Ld}}$), }$$
then there exists a constant $M>0$ not depending on $c$, such that 
\begin{enumerate}[label = (\alph*)]
\item we have
$$ \ir(\delta_0; c, \mW)-\ir_{\textup{min}}(c,\mW)  \le \ie(\delta_0; \mW)  + Mc;$$
\item for any sequential procedure $\delta = (T,D) \in \Delta$, we have
$$\expt_A (T_0  - T) \le \left(\frac{\ie(\delta_0;  \mW)+\ie(\delta; \mW) }{c } +M \right)|\mathcal{A}|, \quad \text{ for any } A\in\mathcal{A}.$$
\end{enumerate}
\end{lemma}
\begin{proof} 
First, we prove part (a). Recall the definition of $\ess$ in \eqref{eq: int ess def} and $\ie$ in \eqref{eq: integrated error}. 
Since $T_0 \le T_{\mathrm{Ld}}(c,\mW)$, we have 
$$
\ess(\delta_0;c) \leq \ess(\delta_{\mathrm{Ld}}(c,\mW);c).
$$
Since $\mW$ is non-negative, we have
\begin{align*}
    &\ir(\delta_0; c,\mW)-\ir(\delta_{\mathrm{Ld}}(c,\mW);c, \mW) \\&=  \ess(\delta_0;c) -\ess(\delta_{\mathrm{Ld}};c) 
+\ie(\delta_0; \mW) -\ie(\delta_{\mathrm{Ld}}(c,\mW);\mW)
    \le \ie(\delta_0; \mW).
\end{align*}
Due to Theorem 2.1 in \cite{lorden1967integrated}, in view of Assumption \ref{assumption: lorden} and condition \eqref{assumption: lorden 1}, $\delta_{\mathrm{Ld}}$ is second-order Bayesian optimal, that is, there exists a constant $M>0$ not depending on $c$, such that 
$$\ir(\delta_{\mathrm{Ld}}(c,\mW);c,\mW)-\ir_{\textup{min}}(c,\mW) \le Mc.$$Summing the above two inequalities, we have 
\begin{equation*}
    \ir(\delta_0;c,\mW)-\ir_{\textup{min}}(c,\mW) \le \ie(\delta_0; \mW) + Mc,
\end{equation*}
which completes the proof of part (a).

Next, we prove part (b). Fix any procedure $\delta = (T,D) \in \Delta$. We define an auxiliary procedure  $\tilde{\delta} = (\tilde{T},\tilde{D})$, as follows
\begin{equation*}
    \tilde{T}= \min\{T,T_0\}, \quad \tilde{D} = 
\begin{cases}
    D &\text{ if } T \leq T_0\\
    D_0 &\text{ otherwise }
\end{cases}.
\end{equation*}
It is clear that $\tilde{D}$ is $\mathcal{F}_{\tilde{T}}$ measurable. Thus, $\tilde{\delta}$ is a valid sequential procedure, that is, $\tilde{\delta} \in \Delta$. 
Due to the definition in \eqref{eq: integrated error},
\begin{equation}\label{eq: thm 4 in lordon 2}
\begin{aligned}
\ie(\tilde{\delta}; \mW)&=\sum_{A\in\mathcal{A}}\pi_0(A)\expt_A[\mW(D_0\mid A)\idf\{T_0 <  T\} + \mW(D\mid A)\idf\{T\le T_0\}]
\\&\le \ie(\delta_0; \mW) + \ie(\delta;  \mW).
\end{aligned}    
\end{equation}
Moreover, due to definition \ref{smallest_Bayes_risk}, we have $\ir_{\textup{min}}(c,\mW)\le \ir(\tilde{\delta};c,\mW)$, then it follows from part (a) that
\begin{equation}\label{eq: thm 4 in lordon 3}
\ir(\tilde{\delta};c,\mW) \ge \ir(\delta_0;c,\mW)-[\ie(\delta_0; \mW) + Mc] . 
\end{equation}
Then due to \eqref{eq: thm 4 in lordon 2} and \eqref{eq: thm 4 in lordon 3}, we have
\begin{align*}
\ess(\delta_0;c)- \ess(\tilde{\delta};c) &=
[\ir(\delta_0;c,\mW)-\ie(\delta_0; \mW)]-[\ir(\tilde{\delta};c,\mW)-\ie(\tilde{\delta};\mW)] \\
&\le \ie(\delta;  \mW)+\ie(\delta_0; \mW)+ Mc,
\end{align*}
which, due to definition \eqref{eq: int ess def}, implies that
$$ \frac{c}{|\mathcal{A}|}\sum_{A\in\mathcal{A}}\expt_A (T_0 - \tilde{T})\le \ie(\delta; \mW)+\ie(\delta_0; \mW)+ Mc. $$
Due to the definition of $\tilde{T}$, we have  $\tilde{T} \leq T$ and $\tilde{T} \leq T_0$. Thus, for any $A\in\mathcal{A}$, we have 
$$\frac{c}{|\mathcal{A}|} \expt_A (T_0 - T) \le \frac{c}{|\mathcal{A}|} \expt_A (T_0 - \tilde{T}) \le \ie(\delta;   \mW)+\ie(\delta_0; \mW)+ Mc,$$
which completes the proof of part (b).
\end{proof}

\section{Proof of Theorem \ref{thm: min ESS characterization}}\label{sec: proof of min ESS} 

\begin{proof}[Proof of Theorem \ref{thm: min ESS characterization}]
Recall the procedures $\delta_{\mathrm{Ld}}(c_{\bell}, \mW)$ in \eqref{eq: lorden procedure}. 

\noindent \textbf{Step 1: we show that  
$T_A^{\textup{min}}(\Delta(\bell))=\expt_A[T_{\mathrm{Ld}}(c_{\bell},\mW)]+O(1)$ as $\bell \to \bd{0}$.}

We construct a procedure, denoted by $\delta_1(\bell) = (T_1(\bell), D_1(\bell))$, by modifying $\delta_{\mathrm{Ld}}(c_{\bell})$ and $\delta_0(\bell)$: 
$$
T_1(\bell) := T_{\mathrm{Ld}}(c_{\bell},\mW), \quad \text{ and } \quad D_1(\bell):=D_0(\bell).
$$
By definition and due to \eqref{eq: second-order optimal sufficient condition 1},
$$
D_1(\bell)= D_0(\bell)\;\in\; \mF_{T_0(\bell)}\;\subset\;\mF_{T_{\mathrm{Ld}}(c_{\bell},\mW)} = \mF_{T_1(\bell)}.
$$
Thus, $\delta_1(\bell)$ is a valid sequential procedure, that is, $\delta_1(\bell) \in \Delta$  for each $\bell \in (0,1)^r$.

Since $\delta_0(\bell)\in\Delta(\bell)$ and  $\delta_1(\bell)$ share the same decision rule, by Assumption \ref{cond:class_depends_on_decision}, we have
$\delta_1(\bell)\in\Delta(\bell)$ for each $\bell \in (0,1)^r$.

Note that we assume \eqref{assumption: lorden 1} and \eqref{eq: second-order optimal sufficient condition 2} hold for $\mW$ and $\Delta(\bell)$. Moreover, $T_1(\bell) = T_{\mathrm{Ld}}(c_{\bell}, \mW)$ satisfies condition \eqref{eq: second-order optimal sufficient condition 1} by definition. Thus, by Theorem \ref{thm: main result}, we have $\{\delta_1(\bell)\}$ is second-order optimal in $\{\Delta(\bell)\}$, which in particular implies that as $\bell\rightarrow \bd{0}$,
$$
\expt_A[T_{\mathrm{Ld}}(c_{\bell},\mW)] -T_A^{\textup{min}}(\Delta(\bell)) \;=\;
\expt_A[T_1(\bell)]-T_A^{\textup{min}}(\Delta(\bell))\;=\; O(1).
$$ 

\noindent \textbf{Step 2: we establish a second-order approximation of $\expt_A[T_{\mathrm{Ld}}(c_{\bell},\mW)]$.}

Denote by $m_{\mW}$ is the smallest \emph{nonzero} entry of the matrix $\mW(\cdot|\cdot)$, and by $M_{\mW}$ the largest, that is,
\begin{align}\label{aux:m_M}
\begin{split}
m_{\mW}&:=\min\{\mW(D|A):A,D\in\mathcal{A}\text{ and }\mW(D|A)>0\},\;\;\; 
M_{\mW} :=\max\{\mW(D|A):A,D\in\mathcal{A}\}.
\end{split}
\end{align}
Since $\mathcal{A}$ is finite, we have $0< m_{\mW} \leq M_{\mW} < \infty$. 
Moreover, for each $c > 0$, define
\begin{equation}\label{eq: tilde T}
\begin{aligned}
&\tilde{T}(c,\mW) := \min_{B,D\in\mathcal{A}}\tilde{T}_{B,D}(c,\mW), \quad \text{ where }
\\&\tilde{T}_{B,D}(c,\mW) := \min\left\{t\ge 1:  f_{G,[t]}\; < \; c f_{B,[t]},\; \text{ for each } G\in\mathcal{H}_{D}^{\mW}\right\}.
\end{aligned}
\end{equation}
That is, $\tilde{T}(c,\mW)$ stops at the first time $t$ when there exist $B_0,D_0\in\mathcal{A}$ such that $f_{B_0,[t]}/f_{G,[t]}> c^{-1}$ for all $G\in\mathcal{H}_{D_0}^{\mW}$, where we recall $\mathcal{H}_{D_0}^{\mW}$ is defined in \eqref{def:H_D_W}.

By Lemma \ref{lemma: MSPRT equivalent with Lorden}, it follows that
$$
\tilde{T}\left(c_{\bell}\frac{|\mathcal{A}|}{m_{\mW}},\;\mW \right)\le T_{\mathrm{Ld}}(c_{\bell},\;\mW)\le \tilde{T}\left(c_{\bell} \frac{1}{|\mathcal{A}| M_{\mW}},\;\mW \right).
$$

Furthermore, Lemma \ref{lemma: min ESS dragalin} provides an asymptotic approximation to $\expt_A[\tilde{T}(c,\mW)]$ as $c \to \infty$, via nonlinear renewal theory. 
Since $|\mathcal{A}|$, $m_{\mW}$ and $M_{\mW}$ are all constants, the proof is thus complete.
\end{proof}

For a cost $c > 0$ and a loss function $\mW:\mathcal{A}\times\mathcal{A}\to[0,\infty)$,  the stopping time $\tilde{T}(c,\mW)$ is defined in \eqref{eq: tilde T}, and $m_{\mW}, M_{\mW}$ in \eqref{aux:m_M}.

\begin{lemma}\label{lemma: MSPRT equivalent with Lorden}
Let $\mW:\mathcal{A} \times \mathcal{A} \to [0,\infty)$ be a loss function and $c > 0$. Then
$$
T_1:= \tilde{T}\left(\frac{c|\mathcal{A}|}{m_{\mW}},\mW \right)\le T_{\mathrm{Ld}}(c,\mW)\le \tilde{T}\left(\frac{c}{|\mathcal{A}| M_{\mW}},\mW \right):= T_2.
$$
\end{lemma}
\begin{proof}
We begin by establishing the second inequality. At time $T_2$, by definition in \eqref{eq: tilde T}, there exist $B_0,D_0\in\mathcal{A}$ such that 
\begin{align*}
\frac{c}{|\mathcal{A}| M_{\mW}}&\ge \frac{\max_{G\in\mathcal{H}_{D_0}^{\mW}}f_{G,[T_2]}}{f_{B_0,[T_2]}},
\end{align*}
where $D_{D_0}^{\mW}$ is as defined in \eqref{def:H_D_W}. On one hand, by definition of $M_\mW$,
\begin{align*}
 \max_{G\in\mathcal{H}_{D_0}^{\mW}}f_{G,[T_2]} \geq \frac{1}{|\mathcal{A}|} \sum_{G \in\mathcal{H}_{D_0}^{\mW}} f_{G,[T_2]} \geq \frac{1}{|\mathcal{A}| M_{\mW}} \sum_{A\in\mathcal{H}_{D_0}^{\mW}} f_{G,[T_2]} W(D|G).
\end{align*}
On the other hand, we have
$$
f_{B_0,[T_2]} \leq \sum_{B\in\mathcal{A}}f_{B,[T_2]}.
$$
Then, in view of the definition of $\pi_t$ in \eqref{eq: posterior}, the above three inequalities imply that
\begin{align*}
    c > \sum_{A\in\mathcal{H}_{D_0}^{\mW}} \pi_{T_2}(G) W(D|G).
\end{align*}
Thus, the stopping criterion for $T_{\mathrm{Ld}}(c,\mW)$ is met at time $T_2$, which implies that $T_{\mathrm{Ld}}(c,\mW)\le T_2$.

Next, we focus on the first inequality. Since $\mW$ is fixed in this proof, we write $T_{\mathrm{Ld}}(c)$ for $T_{\mathrm{Ld}}(c,\mW)$. At time $T_{\mathrm{Ld}}(c)$, due to the definitions in \eqref{eq: posterior} and \eqref{eq: lorden procedure}, there exists $D_0\in\mathcal{A}$ such that
\begin{align*}
\sum_{G\in\mathcal{A}}f_{G,[T_{\mathrm{Ld}}(c)]} \mW(D_0\mid G) <  c\sum_{B\in\mathcal{A}}f_
{B,[T_\mathrm{Ld}(c)]}.
\end{align*}
Let $B_0 := \argmax_{B\in\mathcal{A}} f_
{B,[T_\mathrm{Ld}(c)]}$. Then, in view of the definition of $m_{\mW}$, it follows that
$$c|\mathcal{A}| f_
{B_0,[T_\mathrm{Ld}(c)]}> \sum_{G \in\mathcal{H}_{D_0}^{\mW}}f_{G,[T_{\mathrm{Ld}}(c)]}\mW(D_0\mid G)\ge m_{\mW}\max_{G\in\mathcal{H}_{D_0}^{\mW}}f_{G,[T_{\mathrm{Ld}}(c)]}.$$
Thus, the stopping criterion for $T_1$ is met at time $T_{\mathrm{Ld}}(c)$. Therefore we have $T_1\le T_{\mathrm{Ld}}(c,\mW)$. The proof is complete.
\end{proof}

Next, we provide an asymptotic approximation to $\expt_A[\tilde{T}(c,\mW)]$ as $c \to \infty$, where $A$ appears in Theorem \ref{thm: min ESS characterization}.  Recall  $h_{A}^{\mW}$ defined in \eqref{def:h_A} and the stopping time $\tilde{T}(c,\mW)$  in \eqref{eq: tilde T}. Define
\begin{align}
    \label{aux:L_c}
    L_c := \log(c^{-1}), \quad \text{ for } c \in (0,1).
\end{align}

\begin{lemma}\label{lemma: min ESS dragalin}
Suppose Assumption \ref{assumption: lorden} holds.
Let $A \in \mathcal{A}$, and assume that $A$ admits a unique most favorable subset with respect to $\mW$.
\begin{enumerate}[label=(\alph*), ref=\theassumption $(\alph*)$]
\item Consider the case $r_A^{\mW} = 1$. Assume \eqref{eq: bounded moment assumption} holds for $q=3$. Then, as $c\rightarrow 0$
$$\expt_A\left[\tilde{T}(c,\mW)\right] = \frac{L_c}{\KL_{A,*}^{\mW}}+O(1).$$

\item Consider the case $r_A^{\mW} \ge 2$. For $\epsilon\in(0,1/2)$, assume \eqref{eq: bounded moment assumption} holds for some $q$ satisfying \eqref{eq: main text q_cond symmetric}.
Then, as $c\rightarrow 0$
$$\expt_A\left[\tilde{T}(c,\mW)\right] = \frac{L_c}{\KL_{A,*}^{\mW}}+\frac{h_{A}^{\mW}\sqrt{L_c}}{(\KL_{A,*}^{\mW})^{3/2}}+O(L_c ^{1/4+ \epsilon/2}).$$
\end{enumerate}
\end{lemma}

\begin{proof}
Recall $D_{A}^{\mW}$ defined in \eqref{def:H_D_W}. For simplicity, denote $$
\tilde{T}_{A,*}:=\tilde{T}_{A,D_{A}^{\mW}} \quad \text{ and } \quad \mathcal{H}_{A,*}^{\mW}:=\mathcal{H}_{D_A^{\mW}}^{\mW}.$$ 
Since $\mW$ is fixed throughout the proof, we suppress its dependence in the notation and write $\tilde{T}_{A,D}(c)$ for $\tilde{T}_{A,D}(c,\mW)$, $\tilde{T}_{A,*}(c)$ for $\tilde{T}_{A,*}(c,\mW)$, and $\tilde{T}(c)$ for $\tilde{T}(c,\mW)$

We note the following decomposition
\begin{equation}\label{eq: min ESS dragalin 1}
\expt_A \left[\tilde{T}(c)\right] = \expt_A\left[\tilde{T}_{A,*}(c)\right] - \expt_{A}\left[(\tilde{T}_{A,*}(c)-\tilde{T}(c))\idf\{\tilde{T}(c)<\tilde{T}_{A,*}(c)\}\right]. 
\end{equation}
By definition \eqref{eq: tilde T}, $0 < \tilde{T}(c) \leq \tilde{T}_{A,*}(c)$ for each $c \in (0,1)$.
By Holder’s inequality, we have
\begin{align*}
0 &\leq \expt_{A}\left[(\tilde{T}_{A,*}(c)-\tilde{T}(c))\idf\{\tilde{T}(c)<\tilde{T}_{A,*}(c)\}\right]\\ 
&\leq \expt_A\left[\tilde{T}_{A,*}(c)\idf\{\tilde{T}(c)<\tilde{T}_{A,*}(c)\}\right]\le \left(\expt_A\left[(\tilde{T}_{A,*}(c))^3\right]\right)^{1/3} \left(\prob_A\left(\tilde{T}(c)<\tilde{T}_{A,*}(c) \right)\right)^{2/3}.  
\end{align*}
Theorem 4.1 in \cite{draglia1999multihypothesis} shows that 
$$
\left(\expt_A\left[(\tilde{T}_{A,*}(c))^3\right]\right)^{1/3} = O(L_c).
$$
In Lemma \ref{lemma:aux_A_D_A_W}, we show that  $\prob_A(\tilde{T}_{B,D}(c) < \tilde{T}_{A,*}(c)) = O\left((L_c)^{-3/2}\right)$ for any $(B,D) \neq (A,D_A^{W})$. Then, since $\mathcal{A}$ is finite, by the union bound, it follows that
\begin{align*}
    \prob_A \left(\tilde{T}(c) < \tilde{T}_{A,*}(c) \right) \leq \sum_{(B,D) \in \mathcal{A}^2\setminus \{(A,D_A^{W})\}} \prob_A\left(\tilde{T}_{B,D}(c) < \tilde{T}_{A,*}(c)\right) = O\left((L_c)^{-3/2}\right).
\end{align*}
Therefore, we have
$$
\left|\expt_{A}\left[(\tilde{T}_{A,*}(c)-\tilde{T}(c))\idf\{\tilde{T}(c)<\tilde{T}_{A,*}(c)\}\right]\right| = O(1),
$$
which, in view of \eqref{eq: min ESS dragalin 1}, implies
$$
\expt_A \left[\tilde{T}(c)\right] = \expt_A\left[\tilde{T}_{A,*}(c)\right] + O(1).
$$
Finally,  note that by definition,
\begin{align*}
\tilde{T}_{A,*}(c)
=\min\left\{t\ge1:\min_{C\in\mathcal{H}_{A,*}^{\mW}}\left(\sum_{s=1}^t\log\frac{f_A(\mathbf{X}_s^{[K]})}{f_C(\mathbf{X}_s^{[K]})}\right)\ge L_c\right\}.
\end{align*}
Furthermore, recall from \eqref{def:KL_A_D_star} that $\KL_{A,*}^{\mW} = \min_{C \in \mathcal{H}_{A,*}^{\mW}} \KL(f_A|f_C)$.

\medskip

\noindent \textbf{Part (a) $r_A = 1$.} 
Due to Theorem \ref{thm: rw_asymm}, as $c\rightarrow 0$ we have
$$\expt_A\left[\tilde{T}_{A,*}(c)\right] = \frac{L_c}{\KL_{A,*}^{\mW}}+O(1).$$

\medskip

\noindent \textbf{Part (b) $r_A \geq 2$.} 
Due to Theorem \ref{thm:rw_sym}, as $c\rightarrow 0$ we have
$$\expt_A(\tilde{T}_{A,*}) = \frac{L_c}{\KL_{A,*}^{\mW}}+\frac{h_{A}^{\mW}\sqrt{L_c}}{(\KL_{A,*}^{\mW})^{3/2}}+O(L_c^{1/4+ \epsilon/2}).$$
The proof is complete.
\end{proof}

\subsection{Supporting Lemmas}
Recall $L_c := \log(c^{-1})$ in \eqref{aux:L_c}, and the stopping time $\tilde{T}_{B,D}(c,\mW)$  in \eqref{eq: tilde T}. For $B,G \in \mathcal{A}$, we define
\begin{align}\label{def:check_T}
\widehat{T}_{B,G}(c) := \min\left\{t\ge 1: f_{G,[t]}\le c f_{B,[t]}\right\}.
\end{align}
For $B,G,D \in \mathcal{A}$,  we define the following quantity
\begin{equation*}
\textup{Div}_A(f_B, f_G) := \expt_A\left[\log\left(\frac{f_B(\mathbf{X}_1^{[K]})}{f_G(\mathbf{X}_1^{[K]})}\right) \right] = \int \log(f_B/f_G) f_A\; d(\prod_{k\in [K]}\mu_k),
\end{equation*}
and further denote by 
\begin{equation}
    \label{Div_A_W_B_D}
\textup{Div}_{A}^{\mW}(B,D) :=  \min_{G \in \mathcal{H}_{D}^{\mW}} \textup{Div}_A(f_B, f_G).
\end{equation}



\begin{lemma}\label{lemma:aux_A_D_A_W}
Suppose Assumption \ref{assumption: lorden} holds. 
Let $A \in \mathcal{A}$, and assume that $A$ admits a unique most favorable subset with respect to $\mW$. Let $(B,D) \in \mathcal{A} \times \mathcal{A}$ such that $(B,D) \neq (A,D_A^{\mW})$, where  $D_A^{\mW}$ appears in Definition \ref{def: dragalin 1}. Assume \eqref{eq: bounded moment assumption} holds for 
$q=3$. Then
\begin{align*}
    \prob_A\left(\tilde{T}_{B,D}(c,\mW) < \tilde{T}_{A,D_A^{\mW}}(c,\mW) \right) = O\left((L_c)^{-3/2}\right).
\end{align*}
\end{lemma}
\begin{proof}
If $B \neq A$,  by definition and due to Assumption \ref{assumption: lorden},  for any $G \in \mathcal{A}$, 
$$
\textup{Div}_A(f_B, f_G) = \KL(f_A\mid f_G)-\KL(f_A\mid f_B) < \textup{Div}_A(f_A, f_G) = \KL(f_A|f_G).
$$
Since $\mathcal{A}$ is finite and due to \eqref{def:KL_A_D_star}, we have
\begin{align*}
   \textup{Div}_A^{\mW}(B,D) <  \min_{G \in \mathcal{H}_{D}^{\mW}} \KL(f_A|f_G) = \KL_{A,D}^{\mW} \leq \KL_{A,*}^{\mW},
\end{align*}
If  $B = A$, since $(B,D) \neq (A,D_A^{\mW})$, we must have $D \neq D_A^{\mW}$. Then since  $A$ admits a unique most favorable subset with respect to $\mW$, we have
\begin{align*}
 \textup{Div}_A^{\mW}(B,D) = \textup{Div}_A^{\mW}(A,D)  = \KL_{A,D}^{\mW} < \KL_{A,*}^{\mW},
\end{align*}
where we use the fact that $D_A^{\mW}$ is the \emph{unique} maximizer of $\max_{G \in \mathcal{A}} \KL_{A,G}^{\mW}$.

Thus, in either case,  $\KL_{A,*}^{\mW}  > \textup{Div}^{\mW}_{A}(B,D)$, implying that
 there exists  $\mu^* > 0$ such that 
 \begin{equation*} 
     \KL_{A,*}^{\mW} > \mu^* > \textup{Div}^{\mW}_{A}(B,D).
 \end{equation*}

Due to union bound, we have 
\begin{align*}
\prob_A\left(\tilde{T}_{B,D}(c) < \tilde{T}_{A,D_A^{\mW}}(c)\right)\le \prob_A\left(\tilde{T}_{B,D}(c) \le \frac{L_c}{\mu^*}\right)  + \prob_A\left(\tilde{T}_{A,D_A^{\mW}}(c) \ge \frac{L_c}{\mu^*}\right).
\end{align*}

We start with  the first term in the upper bound above. By definition \eqref{Div_A_W_B_D}, there exists some $G^* \in \mathcal{H}_D^{\mW}$ such that
\begin{equation}
    \label{aux:G_star}
\textup{Div}_A(f_B, f_{G^*}) = \textup{Div}^{\mW}_{A}(B,D) < \mu^*.
\end{equation}
By definition \eqref{def:check_T} and \eqref{eq: tilde T}, $\tilde{T}_{B,D}(c) \geq \hat{T}_{B,G^*}(c)$, and thus
\begin{align*}
    \prob_A\left(\tilde{T}_{B,D}(c) \le \frac{L_c}{\mu^*}\right)  \leq 
    \prob_A\left(\hat{T}_{B,G^*}(c) \le \frac{L_c}{\mu^*}\right)  
\end{align*}
Due to \eqref{aux:G_star} and by Lemma \ref{lemma: aux 0}(a), we have $\prob_A\left(\tilde{T}_{B,D}(c) \le \frac{L_c}{\mu^*}\right) = O(L_c^{-3/2})$.

Finally, by Lemma \ref{lemma: aux 2}, $\prob_A\left(\tilde{T}_{A,D_A^{\mW}}(c) \ge \frac{L_c}{\mu^*}\right) = O\left((L_c)^{-3/2}\right)$.
The proof is then complete.
\end{proof}

Recall the definition of $\hat{T}_{B,G}$ in \eqref{def:check_T}.

\begin{lemma}\label{lemma: aux 0}
Suppose that Assumption \ref{assumption: lorden} holds, and that \eqref{eq: bounded moment assumption} holds for $q=3$. Let $A, B,G \in \mathcal{A}$ with $B \neq G$.
\begin{enumerate}[label=(\alph*)]
    \item For each $x > \max\{0,\textup{Div}_A(f_B,f_{G})\}$, as $c\rightarrow 0$,
    \begin{align*}
 \prob_A\left(\widehat{T}_{B,G}(c)\le\frac{L_c}{x}\right) = O\left((L_c)^{-3/2}\right).
\end{align*}
\item Assume $G \neq A$. Then for each $x \in (0, \KL(f_A|f_G))$, as $c\rightarrow 0$,
\begin{align*}
    \prob_A\left(\widehat{T}_{A,G}(c)\ge\frac{L_c}{x}\right) = O\left((L_c)^{-3/2}\right).
\end{align*}
\end{enumerate}
\end{lemma}
\begin{proof}
We note that $\{\mathbf{X}_t^{[K]}: t\ge 1\}$ are collections of $i.i.d.$ random vectors with 
$$\expt_A\left[\log\frac{f_B(\mathbf{X}_1^{[K]})}{f_{G}(\mathbf{X}_1^{[K]})}\right] = \textup{Div}_A(f_B,f_{G})$$
and since \eqref{eq: bounded moment assumption} holds for $q=3$, we have
$\expt_A\left|\log\left({f_B(\mathbf{X}_1^{[K]})}/{f_{G}(\mathbf{X}_1^{[K]})}\right)\right|^{3} < \infty$.

\medskip

\noindent \textbf{Part (a).} Let $\theta$ be any constant such that
$\max\left\{0, \text{Div}_A(f_B, f_G)\right\} < \theta < x$. Further, define
$$
\tau_{B,G}(c) := \min\left\{t \geq 1: \sum_{s=1}^t \left(\log\frac{f_B(\mathbf{X}_s^{[K]})}{f_G(\mathbf{X}_s^{[K]})} + \theta  -\text{Div}_A(f_B, f_G) \right)\ge L_c \right\}.
$$
By the definition of $\theta$, it is clear that $\tau_{B,G}(c) \leq \widehat{T}_{B,G}(x)$, and that
$$
\expt_A\left[\log\frac{f_B(\mathbf{X}_1^{[K]})}{f_G(\mathbf{X}_1^{[K]})} + \theta  -\text{Div}_A(f_B, f_G)\right] = \theta > 0.
$$
As a result, due to Theorem 8.4 in Chapter 3 of \cite{gut2009stopped}, as $c\rightarrow 0$,
\begin{align*}
\expt_A\left|{\tau}_{B,G}-\frac{L_c}{\theta}\right|^{3} = O(L_c^{3/2}).
\end{align*}    
Finally, note that
\begin{align*}
\prob_A\left(\widehat{T}_{B,G}\le\frac{L_c}{x}\right) \leq   \prob_A\left({\tau}_{B,G}\le\frac{L_c}{x}\right) 
\leq \prob_A\left(\left|{\tau}_{B,G}-\frac{L_c}{\theta}\right|\ge L_c\left(\frac{1}{\theta}-\frac{1}{x}\right)\right).
\end{align*}
Since $0 < \theta < x$, $1/\theta -1/x > 0$. By Markov inequality,
$$
\prob_A\left(\left|{\tau}_{B,G}-\frac{L_c}{\theta}\right|\ge L_c\left(\frac{1}{\theta}-\frac{1}{x}\right)\right) = O((L_c)^{-3/2}),
$$
which completes the proof of Part (a).

\medskip
\noindent\textbf{Part (b).} Note that $\textup{Div}_A(f_A,f_G)=\KL(f_A\mid f_G)>0$. Thus, the argument is similar to, and simpler than, that of part (a), and we omit the details.
\end{proof}

Recall the stopping time $\tilde{T}_{B,D}(c,\mW)$  in \eqref{eq: tilde T} and the definition of $\hat{T}_{B,G}$ in \eqref{def:check_T}.

\begin{lemma}\label{lemma: aux 2}
Suppose that Assumption \ref{assumption: lorden} holds, and that \eqref{eq: bounded moment assumption} holds for $q=3$.  Let $A \in \mathcal{A}$, and assume that $A$ admits a unique most favorable subset with respect to $\mW$. Then, for each $0 < \mu^* < \KL_{A,*}^{\mW}$, as $c\rightarrow 0$,
$$
\prob_A\left(\tilde{T}_{A,D_A^{\mW}} (c)\ge \frac{L_c}{\mu^*}\right) = O((L_c)^{-3/2}).
$$
\end{lemma}
\begin{proof}
The proof follows similar lines to that of Theorem 2 in \cite{mei2008asymptotic}.
Recall $D^{\mW}_A$ in Definition \ref{def: dragalin 1}. For simplicity, we define $\mathcal{H}_{A,*}^{\mW} :=\mathcal{H}_{A,D_A^{\mW}}^{\mW}$.

First, for each integer $s_0 \geq 0$ and $G \in \mathcal{A}$, define
\begin{equation*}
    \tau_{G}(s_0) :=  \max\left\{t\ge 1: \sum_{s=s_0 +1}^{s_0+t}\log\frac{f_A(\mathbf{X}_s^{[K]})}{f_{G}(\mathbf{X}_s^{[K]})}\leq 0\right\}.
\end{equation*}

For each $G\in \mathcal{H}_{A,*}^{\mW}$, by definition,
\begin{align*}
 \sum_{s=1}^{t}\log\frac{f_A(\mathbf{X}_s^{[K]})}{f_{G}(\mathbf{X}_s^{[K]})}\leq L_c, \text{ for all } t \geq  \widehat{T}_{A,G}(c)+ \tau_G\left(\widehat{T}_{A,G}(c)\right)+1.
 \end{align*}
As a result, 
\begin{align*}
 \tilde{T}_{A,D_A^{\mW}}(c) \le \max_{G \in\mathcal{H}_{A,*}^{\mW}} \left\{\widehat{T}_{A,G}(c)+\tau_{G}\left(\widehat{T}_{A,G}(c)\right) + 1 \right\}. 
\end{align*}

Let $\theta \in \mathbb{R}$ be such that $\mu^* < \theta < \KL_{A,*}^{\mW}$. Due to union bound,
\begin{equation}\label{eq: hat vs check xi_n 2 1}
\begin{aligned}
\prob_A\left(\tilde{T}_{A,D_A^{\mW}}(c) \ge \frac{L_c}{\mu^*}\right) &\le \sum_{G \in \mathcal{H}_{A,*}^{\mW}} \prob_A \left(\widehat{T}_{A,G}(c) \ge \frac{L_c}{\theta}\right) 
+ \sum_{G \in \mathcal{H}_{A,*}^{\mW}} \prob_A \left(\tau_{G}(\widehat{T}_{A,G}(c)) + 1 \ge \frac{L_c}{\mu^*} - \frac{L_c}{\theta}\right).
\end{aligned}
\end{equation}
First, we consider the first term on the right hand side of \eqref{eq: hat vs check xi_n 2 1}.
By definition \eqref{def:KL_A_D_star},
$$
\theta<\KL_{A,*}^{\mW} = \KL_{A,D_A^{\mW}}^{\mW} \leq \KL(f_A|f_G), \quad \text{ for each } G \in \mathcal{H}_{A,*}^{\mW},
$$
which, by Lemma \ref{lemma: aux 0}(b), implies that 
for each $G \in \mathcal{H}_{A,*}^{\mW}$, as $c\rightarrow 0$,
\begin{equation*}
\prob_A\left(\widehat{T}_{A,G}(c) \geq \frac{L_c}{\theta}\right) = O(L_c^{-3/2}).
\end{equation*}

Next, we consider the second part on the right hand side of \eqref{eq: hat vs check xi_n 2 1}. Since $\{\mathbf{X}_t^{[K]}:t\ge 1\}$ is a collection of $i.i.d.$ random variables, for each $G \in \mathcal{H}_{A,*}^{\mW}$,  $\tau_{G}(\widehat{T}_{A,G}(c))$ and $\tau_{G}(0)$ have the same distribution under $\prob_A$. Since
$$\expt_A\left[\log\frac{f_A(\mathbf{X}_s^{[K]})}{f_{G}(\mathbf{X}_s^{[K]})}\right] = \KL(f_A|f_G) >0,$$
by Theorem 1 in \cite{janson1986moments}, as $c\rightarrow 0$, we have $\expt_A\left[(\tau_G(0))^{2}\right]=O(1)$. Finally, by Markov's inequality and since $1/\mu^* - 1/\theta > 0$, for each $G \in \mathcal{H}_{A,*}^{\mW}$, as $c\rightarrow 0$,
$$\prob_A \left(\tau_{G}(\widehat{T}_{A,G}(c)) \ge \frac{L_c}{\mu^*} - \frac{L_c}{\theta} -1\right) = O(L_c^{-2}).$$
The proof is complete.
\end{proof}

\section{Asymptotic ESS of stopping times associated with multidimensional random walks}\label{sec: Asymptotic ESS of stopping times associated with multidimensional random walks}

Let $\bd{V}_1, \bd{V}_2,\ldots$ be a sequence of $d$ dimensional i.i.d.~random vectors, and denote by $\bd{R}_n = \bd{V}_1 + \bd{V}_2 + \cdots + \bd{V}_n$ for $n \geq 1$ the associated random walk. For each $j \in [d]$, we denote by $\mu_j := \expt[\bd{V}_{1,j}]$, where $\bd{V}_{1,j}$ represents the $j$-th component of $\bd{V}_1$. Moreover, we denote by $\tilde{\bd{V}}_{t,j} := \bd{V}_{t,j}-\mu_j$ for $t \geq 1$, and then
\begin{align}\label{aux:rw}
    \bd{R}_{n,j}-n\mu_j = \sum_{t=1}^{n} \tilde{\bd{V}}_{t,j}, \quad \text{ for } n \geq 1.
\end{align}

We consider the stopping time
\begin{equation}\label{eq: T_b}
T_b := \inf\left\{n \geq 1: \min_{j\in[d]}\bd{R}_{n,j} \geq b \right\},
\end{equation}
that is, the first time that all components of the random walk exceed a threshold. In this section, we assume
\begin{equation*}
0< \mu_1 \leq \mu_2 \leq \cdots \leq \mu_d.
\end{equation*}
Let $1 \leq r_{*} \leq d$ be the unique integer such that
\begin{equation}
    \label{def:r_star}
\mu_1 = \mu_2 = \ldots = \mu_{r_{*}} < \mu_{r_{*} + 1},
\end{equation}
where $\mu_{d+1} := +\infty$.
We refer to the case $r_{*}=1$ as the asymmetric setup, and to $r_{*}\neq 1$ as the symmetric setup.

In this section, for some $q\ge 2$ to be specified, we assume that
\begin{equation}\label{eq: bounded moment}
\expt\left(|\tilde{\bd{V}}_{1,j}|^q\right)<\infty \text{ for all $j\in[d]$}.
\end{equation}

\begin{theorem}\label{thm: rw_asymm}
Consider the asymmetric setup, i.e., $r_{*}=1$. Assume \eqref{eq: bounded moment} holds for some $q > 2\sqrt{2}$. 
Then, as $b\to\infty$,
\[
\expt(T_b)=\frac{b}{\mu_1}+O(1).
\]
\end{theorem}

\begin{proof}
We apply Theorem~4.2 in \cite{nagai2006nonlinear}. Specifically, let $Z_n := \min_{j\in[d]}\bd{R}_{n,j}$ for $n\ge 1$, and rewrite $\{Z_n : n \ge 1\}$ as a perturbed random walk:
\begin{equation}\label{eq: a perturbed random walk asymmetric}
Z_n = S_n + \xi_n, \quad \text{where} \quad
S_n := \bd{R}_{n,1}, \quad \text{ and } \quad
\xi_n := \min\{0,\, \bd{R}_{n,2} - \bd{R}_{n,1}, \ldots, \bd{R}_{n,d} - \bd{R}_{n,1}\}.
\end{equation}
Then $T_b$ in \eqref{eq: T_b} can be written as
\[
T_b = \inf\{n \ge 1 : Z_n \ge b\} = \inf\{n \ge 1 : S_n + \xi_n \ge b\}.
\]
Further, since $q > 2\sqrt{2}$, there exists some $\alpha^* \in (1/2,1]$ such that
\begin{align*}
    \frac{2}{q} \leq \alpha^* < \frac{q}{4}.
\end{align*}

Lemma \ref{lemma: rho regular in zhang 2006 asymmetric} shows that the process $\{\xi_n:n\ge 1\}$ is $1$-regular with parameters $\mu = \mu_1$, $\alpha= \alpha^*$, and $p=1$  (see Definition \ref{def: rho regular}) and  the following constants:
\begin{equation*}
\delta_0 = \frac{1}{2},\;\; \theta = \frac{\mu}{2},\;\;   \omega_0 = \frac{\mu}{4},\;\;   K = 3,\;\;  \theta^* = \frac{\mu}{2}.
\end{equation*}
The verification of (A.16) in \cite{dragalin2000multihypothesis} implies that condition~(4.2) in \cite{nagai2006nonlinear} holds for   $\delta_0 = 1/2$. Furthermore, since $\alpha^* \geq 2/q$, we have $2/\alpha^* \leq q$, and thus $\expt[|\tilde{\bd{V}}_{1,1}|^{2/\alpha^*}] < \infty$.

Therefore, by Theorem~4.2 in \cite{nagai2006nonlinear},
\[
\mu_1 \, \expt(T_b) = b - \expt(\zeta_{n_b}) + O(1),
\]
where $n_b := \lfloor b/\mu_1 \rfloor$ and $\zeta_n := (\xi_n \wedge \theta n^{\alpha}) \vee (-\theta^* n^{\alpha})$ for $n \geq 1$. Moreover, Lemma \ref{lemma: rho regular in zhang 2006 asymmetric} shows that $\expt(\zeta_{n_b}) = o(1)$ as $b \to \infty$. The proof is complete.
\end{proof}

Next, we turn to the symmetric setup, i.e., $r_{*}\ge 2$. Define
$$
\tilde{\bd{V}}_{1,[r_*]} := \left(\tilde{\bd{V}}_{1,1},\ldots,\tilde{\bd{V}}_{1,r_*}\right)^{\top}, 
\qquad 
\bSigma_{[r_*]} := \expt\left(\tilde{\bd{V}}_{1,[r_*]}\tilde{\bd{V}}_{1,[r_*]}^\top\right)\in\mR^{r_*\times r_*},
$$
so that $\bSigma_{[r_*]}$ is the covariance matrix of $\tilde{\bd{V}}_{1,[r_*]}$.
Let $(Y_1,\ldots,Y_{r_*})^{\top}$ be a zero-mean Gaussian random vector with covariance matrix $\bSigma_{[r_*]}$. Finally, define
\begin{equation}\label{eq: h}
h_{[r_*]} := \expt\left[\max\{Y_1,\ldots,Y_{r_*}\}\right].
\end{equation}

\begin{theorem}\label{thm:rw_sym}
Consider the symmetric setup, i.e., $r_{*} \ge 2$. Let $\epsilon \in (0,1/2)$ be a constant, and  assume
\eqref{eq: bounded moment} holds for some $q$ satisfying
\begin{equation}\label{eq: q_cond symmetric}
 q >\frac{1}{\epsilon} \quad \text{ and } \quad q \geq 3. 
\end{equation}
Then as $b \to \infty$,
$$\expt (T_b) = \frac{b}{\mu_1} + \frac{h_{[r_*]}\sqrt{b}}{\mu_1^{3/2}} + O(b^{1/4+\epsilon/2}).$$
\end{theorem}

\begin{proof}
We apply Theorem 4.2 in \cite{nagai2006nonlinear}. We denote by $Z_n = \min_{j\in[d]}\bd{R}_{n,j}$ for $n\ge 1$, and we rewrite $\{Z_n: n \ge 1\}$ as a perturbed random walk:
\begin{equation}\label{eq: a perturbed random walk symmetric}
Z_n = S_n + \xi_n, \text{ where }
S_n := n \mu_1,\ \xi_n := \min\{\bd{R}_{n,1} -  n\mu_1, \cdots, \bd{R}_{n,d} - n\mu_1\}
\end{equation}
Then $T_b$ in \eqref{eq: T_b} can be written as
$$
T_b = \inf\{n\geq 1: Z_n \geq b\}=\inf\{n \geq 1: S_n + \xi_n \geq b\}.
$$
Lemma \ref{lemma: rho regular in zhang 2006 symmetric} shows that the process $\{\xi_n:n\ge 1\}$ is $n^{\alpha/2}$-regular with parameters $\mu = \mu_1$, $\alpha= 1/2+\epsilon$, and $p=1$, and  constants
\begin{equation*}
\delta_0 = \frac{1}{2},\;\; \theta = \frac{\mu}{2},\;\;   \omega_0 = \frac{\mu}{4},\;\;   K = 3,\;\;  \theta^* = \frac{\mu}{2}.
\end{equation*}
See Definition \ref{def: rho regular}. The verification of (A.16) in \cite{dragalin2000multihypothesis} shows that (4.2) in \cite{nagai2006nonlinear} holds with $\delta_0 = 1/2$. Furthermore, since $S_n = n\mu_1$, this is a deterministic random walk with increment $\mu_1$.

Thus, due to Theorem 4.2 in \cite{nagai2006nonlinear}, we have
$$\mu_1\expt (T_b) = b -\expt(\zeta_{n_b}) + O(b^{\alpha/2}),$$
where $n_b := \lfloor b/\mu_1 \rfloor$ and $\zeta_n := (\xi_n \wedge \theta n^{\alpha}) \vee (-\theta^* n^{\alpha})$ for $n \geq 1$. Moreover, Lemma \ref{lemma: rho regular in zhang 2006 symmetric} shows that $\expt(\zeta_{n_b}) = -h_{[r_*]}\sqrt{n_b} + O(1)$ as $b \to \infty$. The proof is complete.
\end{proof}

\subsection{Supporting Lemmas}
We consider a function $\rho(\cdot): \{1,2,\cdots\}\mapsto[0,\infty)$, such that
\begin{equation*}
1\le \rho(x)=o(x) \text{ as }x\rightarrow\infty,\quad \sup_{x\ge 1}\sup_{x\le t\le 3x}\rho(t)/\rho(2x)<\infty.
\end{equation*}
We recall the definition of $\rho$-regularity for a stochastic process $\{\xi_n \in \mathbb{R} : n \ge 1\}$ from \cite{nagai2006nonlinear}.

\begin{definition}\label{def: rho regular}
$\{\xi_n\in\mR:n\ge 1\}$ is called $\rho$-regular with parameters $\mu > 0$, $1/2<\alpha\le 1$ and $p\ge 1$, and constants 
\begin{equation*}
\delta_0, \theta>0,\;\;\theta < \mu \;\text{ for } \alpha=1, K>0, \omega_0>0, 0<\theta^*<K\mu,
\end{equation*}
if the following holds:
\begin{align*}
&\prob(\max_{\delta_0n<k\le n}\xi_k>\theta n^{\alpha}) = o(\rho(n)/n)^p,
\\&\sum_{k\ge n+Kn^{\alpha}}k^{p-1}\prob(\xi_k\le -(k-n)\mu +\omega_0k^{\alpha})=o(\rho(n))^p,
\\&\left\{\max_{0\le k\le Mn^{\alpha}} \left(\frac{|\zeta_n-\xi_{n+k}|\wedge n^{\alpha}}{\rho(n)}\right)^p: n\ge 1\right\} \text{ is uniformly integrable for all } M > 0,
\end{align*}
where we define
\begin{equation}\label{eq: zeta}
\zeta_n := (\xi_n\wedge\theta n^{\alpha})\vee (-\theta^*n^{\alpha}) = 
\begin{cases}
\theta n^{\alpha} &\text{ for }\xi_n>\theta n^{\alpha}\\
\xi_n &\text{ for }-\theta^* n^{\alpha}\le \xi_n\le \theta n^{\alpha}\\
-\theta^* n^{\alpha} &\text{ for }\xi_n<-\theta^* n^{\alpha}
\end{cases}.
\end{equation}
\end{definition}
\begin{remark}
For the second condition, since $(k-n)/k^{\alpha}$ is increasing in $k$ for $k \ge n + K n^{\alpha}$ and $1/2< \alpha\le 1$, we have
\begin{equation}\label{eq: second condition remark}
-(k-n)\mu + \omega_0 k^{\alpha} \le \left(- \frac{K\mu}{(K+1)^{\alpha}} + \omega_0\right) k^{\alpha}.     
\end{equation}    
\end{remark}

For $p>0$, we denote the $\ell_p$ norm of a random variable $X\in\mR$ by  $\|X\|_p = \left[\expt(|X|^p)\right]^{1/p}$. For $a,b \in \mR$, let $a\vee b = \max\{a,b\}$.

\begin{lemma}\label{lemma: rho regular in zhang 2006 asymmetric}
Consider the asymmetric setup, i.e., $r_{*}=1$. Assume \eqref{eq: bounded moment} holds for some $q > 2$, and let $\alpha^*$ be any constant such that $\alpha^* \in (1/2,q/4)$.  For each integer $n \geq 1$, define
\begin{equation*}
\rho(n)=1. 
\end{equation*}
Then, the following holds for $\{\xi_n:n\ge 1\}$ defined in \eqref{eq: a perturbed random walk asymmetric} and $\{\zeta_n: n \geq 1\}$ in \eqref{eq: zeta}:
\begin{enumerate}
\item $\{\xi_n:n\ge 1\}$ is $\rho$-regular with parameters $\mu = \mu_1$, $\alpha= \alpha^*$, and $p=1$ and constants
\begin{equation}\label{eq: verify constants asymmetric}
\delta_0 = \frac{1}{2},\;\; \theta = \frac{\mu}{2},\;\;   \omega_0 = \frac{\mu}{4},\;\;   K = 3,\;\;  \theta^* = \frac{\mu}{2}.
\end{equation}
\item $\expt(\zeta_n) = o(1)$ as $n\rightarrow\infty$.  
\end{enumerate}

\end{lemma}
\begin{proof}
In this proof, we denote by $C$ a constant that does not depend on $n$, but may depend on 
$q$ and $\{\mu_j,\|\tilde{\bd{V}}_{1,j}\|_q: 1 \leq j \leq d\}$. 

\noindent\textbf{A Tail Probability Bound.}
By definition \eqref{eq: a perturbed random walk asymmetric}, $\xi_n \leq 0$ for $n \geq 1$. Furthermore,  due to union bound and Corollary \ref{cor: R_nj-R_n1<0}, there exists a constant $C > 0$, such that for $n \geq 1$,
\begin{equation}\label{eq: xi_n tail prob asymmetric}
\prob(\xi_n \neq 0) = \prob(\xi_n < 0) \le \sum_{2\le j\le d}\prob(\bd{R}_{n,j}-\bd{R}_{n,1}<0) \le Cdn^{-q/2}.
\end{equation}

\noindent \textbf{First condition in Definition \ref{def: rho regular}.}
Due to definition \eqref{eq: a perturbed random walk asymmetric}, we have $\max_{1\le k\le n}\xi_k \leq 0$, which implies that
\begin{align*}
\prob\left(\max_{\delta_0 n \leq k \leq n}\xi_k \ge \theta n^{\alpha}\right) = 0.
\end{align*}

\noindent\textbf{Second condition in Definition \ref{def: rho regular}.} 
Due to \eqref{eq: second condition remark} and \eqref{eq: verify constants asymmetric}, we have
$$-(k-n)\mu +\omega_0k^{\alpha} \leq -\frac{\mu}{2} k^{\alpha}<0.$$
Then by \eqref{eq: xi_n tail prob asymmetric}, there exists a constant $C > 0$  such that for each $k \ge n + K n^{\alpha}$, we have,
\begin{align*}
\prob(\xi_k\le -(k-n)\mu +\omega_0k^{\alpha}) &\le \prob\left( \xi_k < 0 \right)\le Cdk^{-q/2}.
\end{align*}
Thus, for $p=1$ we have
$$\sum_{k=n + K n^{\alpha}}^{\infty} k^{p-1} \prob(\xi_k\le -(k-n)\mu +\omega_0k^{\alpha}) \le \sum_{k=n + K n^{\alpha}} Cd k^{-q/2} = o(1),$$
where for the last inequality $q>2$.

\noindent\textbf{Third condition in Definition \ref{def: rho regular}.}
First, due to definition \eqref{eq: a perturbed random walk asymmetric}, note that
\begin{equation*}
\begin{aligned}
|\zeta_n - \xi_{n+k}| &\leq |\zeta_n - \xi_n| + |\xi_n - \xi_{n+k}|
\leq |\zeta_n -\xi_n|\idf\{\xi_n \neq 0\}   + |\xi_n - \xi_{n+k}|\idf\left\{\bigcup_{k=1}^{Mn^{\alpha}}\{\xi_{n+k} \neq 0\}\right\}.
\end{aligned}  
\end{equation*}
As a result, we have
\begin{equation*}
\begin{aligned}
&|\zeta_n - \xi_{n+k}|\wedge n^{\alpha} \le a_n+b_n, \text{ where we define }
\\& a_n :=  n^{\alpha} \idf\{\xi_n \neq 0\}, \quad \ b_n = n^{\alpha} \idf\left\{\bigcup_{k=1}^{Mn^{\alpha}}\{\xi_{n+k} \neq 0\}\right\}.    
\end{aligned}
\end{equation*}
We show that $\{a_n, b_n: n\ge 1\}$ are uniformly integrable. Due to Theorem 4.6.3 in \cite{durrett2019probability}, it is sufficient to prove $\expt[a_n] \to 0$ and $\expt[b_n] \to 0$. 

By \eqref{eq: xi_n tail prob asymmetric},  there exists a constant $C > 0$ such that 
\begin{align*}
&\expt( a_n )\le n^{\alpha}\prob(\xi_n \neq 0) \le Cdn^{\alpha-q/2}\rightarrow 0 \textup{ as } n\rightarrow\infty,
\\&\expt( b_n )\le n^{\alpha}\sum_{k=1}^{Mn^{\alpha}}\prob(\xi_{n+k}\neq 0) \le C(M+1)dn^{2\alpha-q/2}\rightarrow 0 \textup{ as } n\rightarrow\infty,
\end{align*}
where we use the fact that $\alpha = \alpha^* < q/4$. Thus, the proof for part 1 is complete. 

\medskip
\noindent\textbf{Part 2. Expectation of $\zeta_n$.} 
Due to definition \eqref{eq: a perturbed random walk asymmetric} and \eqref{eq: zeta}, we have
$$|\zeta_n|\le |\xi_n|\wedge \frac{\mu}{2} n^{\alpha} \le \frac{\mu}{2} n^{\alpha}\idf\{\xi_n\neq 0\}.$$
Thus, by \eqref{eq: xi_n tail prob asymmetric}, there exists a constant $C>0$, such that 
\begin{align*}
\expt(|\zeta_n|) \le \frac{\mu}{2} n^{\alpha}\prob(\xi_n\neq 0)\le Cdn^{\alpha-q/2} = o(1),
\end{align*}
where for the last inequality we use the fact that $\alpha = \alpha^* < q/4 < q/2$.
Then the proof is complete.
\end{proof}

\begin{lemma}\label{lemma: rho regular in zhang 2006 symmetric}
Consider the symmetric case, i.e., $r_{*}\ge 2$. 
 Let $\epsilon \in (0,1/2)$ be a constant, and  assume
\eqref{eq: bounded moment} holds for some $q$ satisfying \eqref{eq: q_cond symmetric}. Define
\begin{equation*}
\alpha^* := 1/2 + \epsilon, \quad \text{ and } \quad
\rho(n):=n^{\alpha^*/2} \; \text{ for integer } n \geq 1.
\end{equation*}
Then, the following hold for $\{\xi_n:n\ge 1\}$ defined in \eqref{eq: a perturbed random walk symmetric} and $\zeta_n$ in \eqref{eq: zeta}:
\begin{enumerate}
\item $\{\xi_n:n\ge 1\}$ is $\rho$-regular with parameters $\mu = \mu_1$, $\alpha= \alpha^*$, and $p=1$, and constants
\begin{equation}\label{eq: verify constants symmetric}
\delta_0 = \frac{1}{2},\;\; \theta = \frac{\mu}{2},\;\;   \omega_0 = \frac{\mu}{4},\;\;   K = 3,\;\;  \theta^* = \frac{\mu}{2}.
\end{equation}
\item $\expt(\zeta_n) = -h_{[r_*]}\sqrt{n}+O(1) \text{ as } n\rightarrow\infty$, where $h_{[r_*]}$ is defined in \eqref{eq: h}.
\end{enumerate}
\end{lemma}
\begin{proof}
In this proof, we denote by $C$ a constant that does not depend on $n$, but may depend on 
$q$, $\{\mu_j,\|\tilde{\bd{V}}_{1,j}\|_q: 1 \leq j \leq d\}$. Since $\epsilon \in (0,1/2)$, due to the definition in \eqref{eq: q_cond symmetric}, we have 
\begin{equation}\label{aux:q_sym_conds}
q \geq 3,\qquad q > \frac{1}{\epsilon},\qquad
q> \frac{3}{4\epsilon} - \frac{1}{2},\qquad q > \frac{1}{2\epsilon} + \frac{1}{2},\quad \text{ and } \quad q > \frac{3}{2} + 3\epsilon= 3 \alpha.     
\end{equation}

\noindent\textbf{Tail Probability Bounds.}
For each $n \geq 1$, we define
\begin{equation}\label{eq: tilde xi}
\begin{aligned}
    \hat{\xi}_n:= \min\{\bd{R}_{n,1} -  n\mu_1, \cdots, \bd{R}_{n,r_{*}} - n\mu_1\},\quad
    \check{\xi}_n := \min\{\bd{R}_{n,r_{*}+1} -  n\mu_1, \cdots, \bd{R}_{n,d} - n\mu_1\},
\end{aligned}   
\end{equation}
where we use the convention $\min\{\emptyset\} := +\infty$. Then, by definition \eqref{eq: a perturbed random walk symmetric}, we have $\xi_n = \min\left\{ \hat{\xi}_n,\; \check{\xi}_n\right\}$. Due to union bound and Corollary \ref{cor: R_nj-R_n1<0},  there exists a constant $C > 0$, such that for $n \geq 1$ we have
\begin{align} \label{eq: hat vs check xi_n}
    \prob(\check{\xi}_n < \hat{\xi}_n) = \sum_{j=r_{*}+1}^{d} \sum_{i = 1}^{r_{*}} \prob\left(\bd{R}_{n,j} - \bd{R}_{n,i} < 0\right) \leq Cr_{*}(d-r_{*})n^{-q/2}.
\end{align}
Since $\mu_1 = \cdots = \mu_{r_{*}}$ in definition \eqref{eq: tilde xi}, for each $1 \leq k \leq n$,
\begin{align*}
|\hat{\xi}_k|&\le\max_{j\in[r_{*}]} |\bd{R}_{k,j}-k\mu_j| \le \sum_{j\in[r_{*}]}\left(\max_{1\le k\le n}|\bd{R}_{k,j}-k\mu_j|\right).
\end{align*}
Thus, due to triangle inequality and Lemma \ref{lemma: bounded moment}, there exists a constant $C > 0$, such that we have
\begin{equation}\label{eq: hat xi_n expectation symmetric}
\left\|\max_{1\le k\le n}|\hat{\xi}_k|\right\|_{q} \le \sum_{j\in[r_{*}]}\left\|\max_{1\le k\le n}|\bd{R}_{k,j}-k\mu_j|\right\|_{q}\le Cr_{*} n^{1/2},
\end{equation}
and by Markov's inequality for any $\delta>0$ we have
\begin{equation}\label{eq: hat xi_n tail prob symmetric}
\prob\left(\max_{1\le k\le n}|\hat{\xi}_k|\ge \delta\right)\le \left\|\max_{1\le k\le n}|\hat{\xi}_k|\right\|_{q}^q/\delta^q \le Cr_{*}^q n^{q/2}/\delta^q.  
\end{equation}

\noindent\textbf{First condition in Definition \ref{def: rho regular}.}
Since $\xi_n = \min\left\{ \hat{\xi}_n,\; \check{\xi}_n\right\}$ and due to \eqref{eq: hat xi_n tail prob symmetric}, there exists a constant $C>0$ such that
\begin{align*}
n^{-(\alpha/2-1)}\prob\left(\max_{\delta_0 n \leq k \leq n}\xi_k \ge \theta n^{\alpha}\right) \le
n^{-(\alpha/2-1)}\prob\left(\max_{1 \leq k \leq n}|\hat{\xi}_k| \ge \theta n^{\alpha}\right) 
\le Cr_{*}^qn^{-(\alpha/2-1)-(\alpha-1/2)q}.
\end{align*}
Since $q>(1-\alpha/2)/(\alpha-1/2) = 3/(4\epsilon) - 1/2$ due to \eqref{aux:q_sym_conds}, the last term is $o(1)$, and thus
$$
\prob\left(\max_{\delta_0 n \leq k \leq n}\xi_k \ge \theta n^{\alpha}\right) = o((n^{\alpha/2}/n)^p),
$$
where we recall that $p=1$. 

\noindent\textbf{Second condition in Definition \ref{def: rho regular}.}   Due to \eqref{eq: second condition remark} and \eqref{eq: verify constants symmetric}, we have
$$-(k-n)\mu +\omega_0k^{\alpha} \leq -\frac{\mu}{2} k^{\alpha}.$$
Further, due to union bound, \eqref{eq: hat vs check xi_n} and \eqref{eq: hat xi_n tail prob symmetric}, there exists a constant $C > 0$, such that for $k \ge n + K n^{\alpha}$,
\begin{align*}
&\prob(\xi_k\le -(k-n)\mu +\omega_0k^{\alpha})\le \prob(\xi_k\le -\frac{\mu}{2} k^{\alpha}) \le \prob(|\hat{\xi}_k|\ge \frac{\mu}{2} k^{\alpha}) + \prob(\check{\xi}_k < \hat{\xi}_k)
\\&\le Cr_{*}^qk^{-(\alpha-1/2)q} + Cr_{*}(d-r_{*})k^{-q/2}
\le Cr_{*}^q(d-r_{*})k^{-(\alpha-1/2)q},
\end{align*}
which, since $p=1$, implies that 
\begin{align*}
\sum_{k=n + K n^{\alpha}}^{\infty} k^{p-1} \prob(\xi_k\le -(k-n)\mu +\omega_0k^{\alpha})
\le Cr_{*}^q(d-r_{*})\sum_{k=n + K n^{\alpha}}^{\infty}k^{-(\alpha-1/2)q}.
\end{align*}
Since $(\alpha-1/2)q  = \epsilon q > \epsilon (1/\epsilon) = 1$ due to \eqref{aux:q_sym_conds}, we have
$$
\sum_{k=n + K n^{\alpha}}^{\infty} k^{p-1} \prob(\xi_k\le -(k-n)\mu +\omega_0k^{\alpha}) = o(1) = o((n^{\alpha/2})^p).
$$

\noindent\textbf{Third condition in Definition \ref{def: rho regular}.}
Note that on the event $\{\check{\xi}_n \ge \hat{\xi}_n\}\cap\{\check{\xi}_{n+k} \ge \hat{\xi}_{n+k}\}$ we have  $\xi_{n+k} = \hat{\xi}_{n+k}$ and $\zeta_n = \hat{\zeta}_n$, where we denote by $\hat{\zeta}_n := (\hat{\xi}_n\wedge\theta n^{\alpha})\vee (-\theta^*n^{\alpha})$.

Due to the triangle inequality and union bound, it follows that
\begin{align*}
&|\zeta_n - \xi_{n+k}|  \leq |\hat{\zeta}_n - \hat{\xi}_n| + |\hat{\xi}_n - \hat{\xi}_{n+k}| + (|\zeta_n - \xi_{n+k}|)
\left(\mathbbm{1}\{\check{\xi}_n < \hat{\xi}_n\} + \mathbbm{1}\{\check{\xi}_{n+k} < \hat{\xi}_{n+k}\}\right),
\end{align*}
As a result, for any $M > 0$,
\begin{align*}
\max_{1\le k\le Mn^{\alpha}}|\zeta_n - \xi_{n+k}|\wedge n^{\alpha} \le a_n+b_n+c_n,
\end{align*}
where we define
\begin{align*}
& a_n := n^{\alpha} \mathbbm{1}\{|\hat{\xi}_n| \geq (\mu/2) n^{\alpha}\},\qquad b_n := \max_{1\le k\le Mn^{\alpha},\ j \in [r_*]} |\bd{R}_{n+k,j}-(n+k)\mu_j - (\bd{R}_{n,j}-n\mu_j)|,
\\& c_n = \max_{1\le k\le Mn^{\alpha}}\left\{ n^{\alpha}\left(\mathbbm{1}\{\check{\xi}_n < \hat{\xi}_n\} + \mathbbm{1}\{\check{\xi}_{n+k} < \hat{\xi}_{n+k}\}\right)\right\}
\end{align*}

We first show that $\{a_n/n^{\alpha/2}:n\ge 1\}$ is uniformly integrable. Due to Theorem 4.6.3 in \cite{durrett2019probability}, it is sufficient to prove $\expt(a_n/n^{\alpha/2})\rightarrow 0 \text{ as } n\rightarrow\infty$.
Due to \eqref{eq: hat xi_n tail prob symmetric}, there exists a constant $C > 0$, such that we have
\begin{align*}
&\expt(|a_n/n^{\alpha/2}|)
= n^{\alpha/2} \prob\left(|\hat{\xi}_n| \geq (\mu/2) n^{\alpha}\right)
\le n^{\alpha/2} \frac{C r_*^q n^{q/2}}{((\mu/2) n^{\alpha})^q}
= C r_*^q n^{\alpha/2 + q/2 - \alpha q}.
\end{align*}
Since $q > (\alpha/2)/(\alpha-1/2) = 1/(4\epsilon) + 1/2$ due to \eqref{aux:q_sym_conds}, we have $\expt(a_n/n^{\alpha/2}) = o(1)$.

Next, we show that $\{b_n/n^{\alpha/2}:n\ge 1\}$ is uniformly integrable. Since the increments are i.i.d., for each $n \geq 1$
$$
b_n \text{ has the same distribution as } \max_{1\le k\le Mn^{\alpha}, \; j \in [r_*]} |\bd{R}_{n,j}-n\mu_j|.
$$
Then, due to \eqref{eq: hat xi_n expectation symmetric},  there exists a constant $C > 0$, such that for each $n \geq 1$,
\begin{align*}
&\expt[|b_n/n^{\alpha/2}|^q]=\left\|\max_{1\le k\le Mn^{\alpha}, \; j \in [r_*]} |\bd{R}_{k,j}-k\mu_j|\right\|_q^q/n^{\alpha q/2} \le Cr_{*}^q M^{q/2},
\end{align*}
which implies $\sup_{n \geq 1}\expt\left[\left(|b_n/n^{\alpha/2}|\right)^q\right] <  \infty$. Then, due to Theorem 4.6.2 in \cite{durrett2019probability}, $\{b_n/n^{\alpha/2}:n\ge 1\}$ is uniformly integrable.

Finally, we show that $\{c_n/n^{\alpha/2}:n\ge 1\}$ is uniformly integrable. Due to Theorem 4.6.3 in \cite{durrett2019probability}, it is sufficient to prove $\expt\left(c_n/n^{\alpha/2}\right) \to 0$.  Due to \eqref{eq: hat vs check xi_n} and union bound, there exists a constant $C>0$, such that 
\begin{align*}
\expt(c_n/n^{\alpha/2})\le \sum_{k=1}^{Mn^{\alpha}} n^{\alpha/2} \left[\prob\left(\check{\xi}_n < \hat{\xi}_n\right) + \prob\left(\check{\xi}_{n+k} < \hat{\xi}_{n+k}\right) \right] \le Cr_{*}(d-r_{*}) M n^{3\alpha/2-q/2} = o(1),
\end{align*}
where for the last inequality we use the fact that $q>3\alpha$ due to \eqref{aux:q_sym_conds}. The proof for part 1 is complete.

\smallskip
\noindent\textbf{Part 2. Expectation of $\zeta_n$.} Due to the definition of $\hat{\zeta}_n$ above, and since \eqref{eq: hat vs check xi_n} holds, there exists a constant $C>0$, such that  
\begin{align*}
    \left|\expt[\zeta_n] - \expt[\hat{\zeta}_n] \right| \leq \expt\left[\mu n^{\alpha} \idf\{\check{\xi}_n< \hat{\xi}_n\}\right] \le Cr_{*}(d-r_{*})n^{\alpha-q/2}.
\end{align*}
Since $q\ge 2\alpha$ due to \eqref{aux:q_sym_conds}, we have $\expt[\zeta_n]=\expt[\hat{\zeta}_n] + O(1)$.

Moreover, due to \eqref{eq: hat xi_n expectation symmetric}, there exists a constant $C>0$, such that  
\begin{align*}
\left|\expt(\hat{\zeta}_n-\hat{\xi}_n)\right| \le \expt\left[|\hat{\xi}_n|\idf\left\{|\hat{\xi}_n|\ge (\mu/2)n^{\alpha}\right\}\right]\le \frac{\expt|\hat{\xi}_n|^q}{[(\mu/2)n^{\alpha}]^{q-1}}\le C(r_*)^q n^{-(\alpha-1/2)q+\alpha}.
\end{align*}
Since $q\ge \alpha/(\alpha-1/2) = 1/(2\epsilon)+1$ due to \eqref{aux:q_sym_conds}, we have $\expt(\hat{\zeta}_n)= \expt(\hat{\xi}_n) + O(1)$.

Finally, recall $h_{[r_*]}$ in \eqref{eq: h}. Due to Lemma \ref{lemma: E(zeta)},   
\begin{align*}
\left|\expt(\hat{\xi}_n) - (-h_{[r_*]}\sqrt{n})\right| = O(1),
\end{align*}
which implies that $\expt[\zeta_n] = - h_{[r_*]}\sqrt{n} + O(1)$. The proof is complete.
\end{proof}

\subsection{Fundamental Result}
\begin{lemma}\label{lemma: bounded moment}
Assume \eqref{eq: bounded moment} holds with some $q\ge 2$. Then, there exists an absolute constant $C>0$, such that for each $j\in[d]$, we have
$$\left\|\max_{1 \leq k \leq n}|\bd{R}_{k,j}-k\mu_j|\right\|_q\le C    q \left\|\tilde{\bd{V}}_{1,j}\right\|_q n^{1/2}.$$
\end{lemma}
\begin{proof} 
Due to Rosenthal's inequality (see e.g. Theorem 1.5.13 of \cite{de2012decoupling}), there exists an absolute constant $C>0$ such that for each $j\in[d]$ and $q\ge 2$ we have
$$\left\|\max_{1 \leq k \leq n}\sum_{t=1}^{k} \tilde{\bd{V}}_{t,j}\right\|_q\vee\left\|\max_{1 \leq k \leq n}\sum_{t=1}^{k} (-\tilde{\bd{V}}_{t,j})\right\|_q
\leq C q\left(
\left\|\sum_{t=1}^n \tilde{\bd{V}}_{t,j}\right\|_2 + \left\|\max_{1\leq t \leq n} |\tilde{\bd{V}}_{t,j}|\right\|_q
\right).$$
Due to independence and since $q \geq 2$, we have
\begin{align*}
    \left\|\sum_{t=1}^n \tilde{\bd{V}}_{t,j}\right\|_2 = \left\|\tilde{\bd{V}}_{1,j}\right\|_2 n^{1/2} \leq \left\|\tilde{\bd{V}}_{1,j}\right\|_q  n^{1/2},  \text{ and }  
    \left\|\max_{1\leq t \leq n} |\tilde{\bd{V}}_{t,j}|\right\|_q\le \left[\sum_{t=1}^n\expt(|\tilde{\bd{V}}_{t,j}|^q)\right]^{1/q}\le \left\|\tilde{\bd{V}}_{1,j}\right\|_q  n^{1/2}.
\end{align*}
Note that  for any $1\le k\le n$ and $j\in[d]$ we have
\begin{align*}
&|\bd{R}_{k,j}-k\mu_j| = \max\left\{\sum_{t=1}^k \tilde{\bd{V}}_{t,j}, -\sum_{t=1}^k \tilde{\bd{V}}_{t,j}  \right\},
\end{align*}
which implies that for each $j \in [d]$,
\begin{align*}
    \left\|\max_{1 \leq k \leq n}|\bd{R}_{k,j}-k\mu_j|\right\|_q \le 2 \left\|\max_{1 \leq k \leq n}\sum_{t=1}^{k} \tilde{\bd{V}}_{t,j}\right\|_q\vee\left\|\max_{1 \leq k \leq n}\sum_{t=1}^{k} (-\tilde{\bd{V}}_{t,j})\right\|_q \leq 4C q \left\|\tilde{\bd{V}}_{1,j}\right\|_q  n^{1/2}.
\end{align*}
Then the proof is complete.
\end{proof}
\begin{corollary}\label{cor: R_nj-R_n1<0}
Assume \eqref{eq: bounded moment} holds for some $q \geq 2$. Let $i < j \in [d]$ such that $\mu_i < \mu_j$. Then, there exists an absolute  constant $C > 0$  such that  
$$\prob(\bd{R}_{n,j}-\bd{R}_{n,i}\leq 0)\le \left(\frac{C q}{\mu_j - \mu_i} \left(\left\|\tilde{\bd{V}}_{1,j}\right\|_q+\left\|\tilde{\bd{V}}_{1,i}\right\|_q\right) \right)^q n^{-q/2}.$$
\end{corollary}
\begin{proof} 
Note that
$$\prob(\bd{R}_{n,j}-\bd{R}_{n,i}\leq 0)= \prob\left(
\sum_{t=1}^{n}(\tilde{\bd{V}}_{t,j} - \tilde{\bd{V}}_{t,i}) \leq  -n(\mu_j - \mu_i)
\right).
$$
By Markov's inequality and Rosenthal's inequality (see the proof of Lemma \ref{lemma: bounded moment}), there exists an absolute constant $C > 0$ such that
\begin{align*}
    \prob(\bd{R}_{n,j}-\bd{R}_{n,i}\leq 0) \leq    \frac{\left[C q \left(\left\|\tilde{\bd{V}}_{1,j}\right\|_q + \left\|\tilde{\bd{V}}_{1,i}\right\|_q\right) n^{1/2}\right]^q}{\left[n(\mu_j - \mu_i)\right]^q}.
\end{align*}
Then the proof is complete.
\end{proof}

Recall the definition of random walks in \eqref{aux:rw} and $r_*$ in \eqref{def:r_star}. Furthermore, recall that $\bSigma_{[r_*]}$ is the covariance matrix of $\tilde{\bd{V}}_{1,[r_*]}$, and  that $h_{[r_*]}$ is defined in \eqref{eq: h}.

\begin{lemma}\label{lemma: E(zeta)} 
Assume \eqref{eq: bounded moment} holds for $q=3$. Then there exists a finite constant $C>0$ depending only on the positive eigenvalues of $\bSigma_{[r_*]}$, $r_*$ and $\{\expt[|\tilde{\bd{V}}_{1,j}|^3],j\in[r_*]\}$, such that
$$\left|\expt\left[\min_{j\in [r_*]}(\bd{R}_{n,j}-n\mu_j)\right] - (-h_{[r_*]})\sqrt{n}\right| \le C.$$
\end{lemma}
\begin{proof} 
Consider the spectral decomposition
$$
\bSigma_{[r_*]} = {U} \Lambda {U}^\top, 
$$
where $\Lambda := \text{diag}(\lambda_1,\lambda_2,\ldots,\lambda_{r_*})$ is a diagonal matrix with $\lambda_1 \geq \lambda_2 \geq \ldots \geq \lambda_{r_*} \geq 0$, and $U \in \mathbb{R}^{r_* \times r_*}$ is an orthonormal matrix, that is, $U U^{\top}$ equals the identity matrix $\bd{I}_{r_* \times r_*}$.

Denote by $s = \textup{rank}(\bSigma_{[r_*]})$ the rank of covariance matrix $\bSigma_{[r_*]}$. By definition, $s \leq r_*$, $\lambda_{1}\ge\cdots\ge\lambda_{s} > 0$ and $\lambda_{s+1}=\cdots=\lambda_{r_*}=0$. Let $U_1 \in \mathbb{R}^{r_* \times s}$ be the first $s$ columns of $U$, and 
 $U_2\in \mathbb{R}^{r_* \times (s-r_*)}$ the last $s-r_*$ columns of $U$.

By definition, 
$$
\cov(U_2^\top \tilde{\bd{V}}_{1,[r_*]} ) = \expt\left[ U_2^\top \tilde{\bd{V}}_{1,[r_*]} \left(U_2^\top \tilde{\bd{V}}_{1,[r_*]}\right)^\top\right] = U_2^\top \bSigma_{[r_*]} U_2 = \boldsymbol{0}_{(r_*-s) \times (r_* - s)},
$$
which implies that
\begin{equation}
    \label{aux:prob1_column}
\prob\left(\tilde{\bd{V}}_{1,[r_*]}  \in \text{column}(U_1)\right) = 1,
\end{equation}
where $\text{column}(U_1)$ is the column span of $U_1$.

Next, let $\Lambda_1 = \text{diag}(\lambda_1,\ldots,\lambda_s) \in \mathbb{R}^{s \times s}$ be the diagonal matrix of all positive eigenvalues, we define
$$
\mathcal{B}_0 := {\Lambda_1}^{-1/2} U_1^{\top} \in \mathbb{R}^{s \times r_*},\quad \mathcal{B}_1 := U_1  {\Lambda_1}^{1/2} \in \mathbb{R}^{r_* \times s}.
$$
Note that by elementary calculation,
\begin{align}\label{aux:prop_B0B1}
    \mathcal{B}_0\mathcal{B}_1 = \bd{I}_{s\times s},\quad \mathcal{B}_0\bSigma_{[r_*]} \mathcal{B}_0^\top=\bd{I}_{s\times s}, \quad
    \mathcal{B}_1\mathcal{B}_0 u = u \text{ for any }\;\; u\in\textup{column}(U_1).
\end{align}

In addition,  for $t \geq 1$, we define
\begin{align*}
&\bar{\bd{V}}_t := \mathcal{B}_0 (-\tilde{\bd{V}}_{t,[r_*]}) \in \mathbb{R}^{s},\quad \bar{\bd{R}}_n := \sum_{t=1}^n \bar{\bd{V}}_t/\sqrt{n} \in \mathbb{R}^{s},\quad \bd{Y} := \mathcal{B}_1 \bd{Z} \in \mathbb{R}^{r_*}, \text{ where } \bd{Z} \sim \mathcal{N}(\bd{0}_s,\bd{I}_{s\times s}).     
\end{align*}
Then, by definition and \eqref{aux:prop_B0B1}, for $j \in [r_*]$,
\begin{equation}\label{eq: E(zeta) 1}
\begin{aligned}
&\expt\bar{\bd{V}}_1 = \bd{0}_s,\quad \cov(\bar{\bd{V}}_1) = \bd{I}_{s\times s}, \quad \bar{\bd{R}}_{n,j} = \mathcal{B}_0 \left[-(\bd{R}_{n,j}-n\mu_j)/\sqrt{n} \right], 
 \quad   \bd{Y} \sim \mathcal{N}(\bd{0}_d,\bSigma_{[r_*]}).
\end{aligned}
\end{equation}
Then, due to  \eqref{aux:prob1_column} and \eqref{aux:prop_B0B1}, we have for each $j \in [r_*]$,
\begin{align*}
    (\bd{R}_{n,j}-n\mu_j)/\sqrt{n} = \mathcal{B}_1\mathcal{B}_0 (\bd{R}_{n,j}-n\mu_j)/\sqrt{n} = -\mathcal{B}_1 \bar{\bd{R}}_{n,j}.
\end{align*}
We define a function $g(u) := \max_{j\in [r_*]}(\mathcal{B}_1 u)_j$ for any $u \in \mR^s$. Then
\begin{align*}
    \min_{j\in [r_*]}\left((\bd{R}_{n,j}-n\mu_j)/\sqrt{n}\right) = -g(\bar{R}_n).
\end{align*}
Furthermore, in view of  $h_{[r_*]}$  defined in \eqref{eq: h}, 
$$
h_{[r_*]} = \expt[\max\{Y_1,\ldots,Y_{r_*}\}] = \expt[\max\{(\mathcal{B}_1 \bd{Z})_1,\ldots,(\mathcal{B}_1 \bd{Z})_{r_*}\}] = \expt[g(\bd{Z})],
$$
which implies that 
\begin{align*}
    \left|\expt \left[\min_{j\in [r_*]}\left((\bd{R}_{n,j}-n\mu_j)/\sqrt{n}\right)\right] - (-h_{[r_*]})\right| = \left|\expt[g(\bar{\bd{R}}_n)] - \expt[g(\bd{Z})]\right|.
\end{align*}

Finally, we apply normal approximation to bound the above term. Denote by $\|u\|$ the Euclidean norm of $u \in \mathbb{R}^{s}$. Since $|\max_{j\in [r_*]} u_j-\max_{j\in [r_*]} v_j| \le \|u-v\|$ for any $u,v\in\mR^s$, we have
$$
\frac{|g(u)-g(v)|}{\|u-v\|} \le \frac{\|\mathcal{B}_1u- \mathcal{B}_1 v\|}{\|u-v\|}\le \lambda_{\max}(\mathcal{B}_1\mathcal{B}_1') = \lambda_{\max}(\bSigma):= M_1(g) \text{ for any } u,v \in\mR^s.
$$
Now, due to (3.5) in \cite{raivc2018multivariate} and \eqref{eq: E(zeta) 1}, 
\begin{align*}
& |\expt g(\bar{\bd{R}}_n) - \expt g(\bd{Z})|
 \le M_1(g)\sum_{t=1}^n\expt\left(\|\bar{\bd{V}}_t/\sqrt{n}\|^2 \min\{
4.5, (11.1+0.83\log(s))\|\bar{\bd{V}}_t/\sqrt{n}\|\}\right)
\\&\le \log(r_*) \lambda_{\max}(\bSigma)\expt(\|\bar{\bd{V}}_1\|^3)n^{-1/2} 
\\&\le \log(r_*) \lambda_{\max}(\bSigma)\left(\lambda_{\max}({\Lambda_1}^{-1})\right)^{3/2}\expt(\|\tilde{\bd{V}}_{1,[r_*]}\|^3)n^{-1/2},
\end{align*}
where the last inequality is due to $\|\bar{\bd{V}}_1\|^2\le \lambda_{\max}({\Lambda_1}^{-1})\|\tilde{\bd{V}}_{1,[r_*]}\|^2$. Then the proof is complete.
\end{proof}

\subsection{On the Validity of Theorem 3.3 in \cite{dragalin2000multihypothesis}} \label{sec: Theorem 3.3 in dragalin is not correct}

We note that the proof of Theorem 3.3 in \cite{dragalin2000multihypothesis} appears to be incorrect. In (A.14) and (A.15) therein, they define
\begin{align*}
& \xi(n)= \sqrt{n}(\zeta(n)-h_{r,i}), \;\;\text{ with }\;\;
\zeta(n) := \frac{1}{\sqrt{n}}\max_{1\le k\le r}[\lambda_{k,i}+Z_{\langle M-r-1+k\rangle}(n)-n\mu_{i[M-1]}],
\end{align*}
where $i$ is a fixed hypothesis, $\{\lambda_{k,i},1 \leq k \leq r \}$ are constants, $\{Z_{\langle M-r-1+k\rangle}(n): 1 \leq k \leq r\}$ is an $r$-dimensional random walk with $\expt[Z_{\langle M-r-1+k\rangle}(1)] = \mu_{i[M-1]}$ for $1 \leq k \leq r$,
 and 
$$
h_{r,i} := \expt\left[ \max\{Y_{1},\ldots,Y_{r}\}\right], \;\; \text{ where } \;\; (Y_1,\ldots,Y_r)^\top \sim \mathcal{N}(0,\Sigma),
$$
and $\Sigma$ is the covariance matrix of the random vector $\{Z_{\langle M-r-1+k\rangle}(1), 1 \leq k \leq r\}$.

They use non-linear renewal theory based on Theorem 3 in \cite{zhang1988nonlinear}, where the condition (A.19) in \cite{dragalin2000multihypothesis} must hold, which is stated in the following:
\begin{equation}\label{eq: uniform integrable in dragalin}
    \{\max_{1\le i\le n}|\xi(n+i)|: n\ge 1\} \text{ is uniform integrable. }
\end{equation}

However, \emph{this condition fails to hold}. Specifically, by the central limit theorem and the continuous mapping theorem, as $n\to\infty$,
\[
\zeta_n-h_{r,i}\ \text{ converges in distribution to } \max\{Y_1,\ldots,Y_r\}-h_{r,i}.
\]
Since the limit is nondegenerate, $\xi(n)= \sqrt{n}(\zeta(n)-h_{r,i})$ is \emph{not} uniformly integrable.

In this work, instead of applying Theorem 3 in \cite{zhang1988nonlinear} by verifying \eqref{eq: uniform integrable in dragalin}, we invoke Theorem 4.2 in \cite{nagai2006nonlinear} and verify the uniform integrability of
\begin{equation*}
\left\{\max_{0\le i\le Mn^{\alpha}}\left(\frac{|\xi(n+i)-\xi(n)|}{\rho(n)}\right):n\ge1\right\},
\end{equation*}
where $\rho(n)=n^{\alpha/2}$ for a suitable $\alpha\in(1/2,1]$. See the previous two subsections for details.

\section{Additional Classes of Sequential Procedures and Discussion}
\label{app:discussions}

\subsection{Known Number of Signals} \label{app:kns}
In this subsection, we consider Class \ref{problem: kns}, $\Delta_m^{\textup{kns}}(\alpha)$, where the number of signals is known to be $m$, and the \emph{classical} misclassification rate is controlled. We start by reviewing the Gap rule, denoted by $\delta_G(b) = (T_G(b),D_G(b))$, proposed in   \cite{song2017asymptotically}, where $b>0$ is a threshold.

 Recall the LLRs $\{\lambda_t^k : k\in [K],\, t=1,2,\ldots\}$ defined in \eqref{eq:LLR}. For each time $t\ge 1$, denote by
\begin{equation}\label{eq: blamda}
\blambda_t^{(1)}\le \cdots \le \blambda_t^{(K)}
\end{equation}
the ordered statistics of the LLRs $\{\lambda_t^k : k\in [K]\}$ (without taking absolute values). Moreover, denote the corresponding stream indices by $\{\bar{j}_1(t),\ldots,\bar{j}_K(t)\}$ such that $\lambda_t^{\bar{j}_i(t)}=\blambda_t^{(i)}$ for $i\in[K]$. Then 
  for $b > 0$, the Gap rule $\delta_G(b)$ is defined as follows:
\begin{equation}\label{eq: gap rule}
\begin{aligned}
    T_G(b) := \min\{t\ge 1: \blambda_t^{(K-m+1)}-\blambda_t^{(K-m)}\ge b\}, \qquad
    D_G(b) := \{\bar{j}_{K-m+1}(t), \cdots, \bar{j}_K(t)\}.
\end{aligned}
\end{equation}
That is $\delta_G(b)$ stops when the gap between the $(K-m+1)$-th and $(K-m)$-th ordered statistics is larger than the threshold $b$, and the null hypothesis are rejected for the streams with the $m$ largest LLRs upon stopping.

It is shown in Theorem 3.1 of \cite{song2017asymptotically} that if we select
\begin{equation}\label{gap:b}
    b_{\alpha} = |\log\alpha| + \log [m(K-m)],
\end{equation}
then $\delta_G(b_{\alpha}) \in \Delta_m^{\textup{kns}}(\alpha)$. Furthermore, by Theorem~5.3 of \cite{song2017asymptotically}, with this choice of threshold, the family $\{\delta_G(b_{\alpha})\}$ is first-order asymptotically optimal as $\alpha \to 0$. 
The next theorem strengthens this optimality to the second order.
\begin{theorem}\label{thm: second-order optimal of gap rule}
Suppose Assumption \ref{assumption: lorden} holds. Then, the Gap rule $\{\delta_G(b_{\alpha})\}$ is second-order asymptotically optimal in $\{\Delta_m^{\textup{kns}}(\alpha)\}$ as $\alpha \to 0$, 
 that is, for each $A \subset [K]$,
\begin{equation*}
\limsup_{\alpha \to {0}}
\Bigl(
\expt_A[T_G(b_{\alpha})]
-
T_A^{\textup{min}}\left(\Delta_m^{\textup{kns}}(\alpha)\right) 
\Bigr)
< \infty.
\end{equation*}

\end{theorem}
\begin{proof}
The proof is in Appendix \ref{sec: proof of kns}. 
\end{proof}

\begin{remark}
In the proof of above theorem, we apply Theorem \ref{thm: main result} with 
\begin{equation}\label{eq: loss kns}
 \mW(D\mid A)=
\begin{cases}
    1 &\;\text{ if }\;\; D \neq A \\
    0 &\;\text{ else} 
\end{cases}, \quad \text{ for } A, D \subset [K] \text{ with } |A| = |D| = m.
\end{equation} 
\end{remark}

Next, we consider asymptotic approximations to the smallest achievable ESS,
$T_A^{\textup{min}}\left(\Delta_m^{\textup{kns}}(\alpha)\right)$, for each $A \subset [K]$ as $\alpha \to 0$. It is shown in Theorem 5.3 of \cite{song2017asymptotically} that for each $A \subset [K]$, the first-order asymptotic performance satisfies
\begin{align*}
    T_A^{\textup{min}}\left(\Delta_m^{\textup{kns}}(\alpha)\right) = \frac{|\log(\alpha)|}{\KL_{A,A}^{\mW}}(1+o(1)),
\end{align*}
where $\KL_{A,A}^{\mW}$ is defined in \eqref{def:KL_A_D_star} with $\mW$ given in \eqref{eq: loss kns}, that is,
\begin{align*}
    \KL_{A,A}^{\mW} = \min\{\KL(f_A|f_C):  C \in \mathcal{A} \text{ and } C \neq A\}.
\end{align*}

In this work, we provide second-order accurate approximations. Specifically, since $\mW$ in \eqref{eq: loss kns} is a zero-one loss, by Remark \ref{remark: zero-one loss}, each $A \in \mathcal{A}$ admits a unique most favorable subset with repsect to $\mW$, as in Definition \ref{def: dragalin 1}, with $D_A^{\mW} = A$ and $ \KL_{A,*}^{\mW} =  \KL_{A,A}^{\mW}$.

\begin{theorem}\label{thm: second-order min ESS kns}
Suppose Assumption \ref{assumption: lorden} holds. Let $\mW$ be defined in \eqref{eq: loss kns}, and  $A \in \mathcal{A}$ be an arbitrary signal subset.
\begin{enumerate}[label=(\alph*), ref=\theassumption $(\alph*)$]
\item 
Consider the asymmetric case $r_A^{\mW} = 1$. Assume \eqref{eq: bounded moment assumption} holds for $q=3$.
Then, as $\alpha\rightarrow 0$,  
$$
T_A^{\textup{min}}\left(\Delta_{m}^{\mathrm{kns}}(\alpha)\right)= \frac{|\log \alpha|}{\KL_{A,A}^{\mW}}+O(1).
$$
\item 
Consider the symmetric case $r_A^{\mW} \geq 2$. Let $\epsilon \in (0,1/2)$, and assume
\eqref{eq: bounded moment assumption} holds for some $q$ satisfying \eqref{eq: main text q_cond symmetric}. Then, as $\alpha\rightarrow 0$,  
$$
T_A^{\textup{min}}\left(\Delta_{m}^{\mathrm{kns}}(\alpha)\right)= \frac{|\log \alpha|}{\KL_{A,A}^{\mW}}+\frac{h_{A,A}^{\mW}\sqrt{|\log \alpha|}}{(\KL_{A,A}^{\mW})^{3/2}}+O\left((\log \alpha)^{1/4+ \epsilon/2}\right).
$$
\end{enumerate}
\end{theorem}
\begin{proof}
The proof is similar to that of Theorem \ref{thm: second-order min ESS gmr} and is therefore omitted.
\end{proof}

\subsection{Additional Subclasses of Sequential Procedures}
\label{app:further}

In this subsection, we review additional subclasses of sequential procedures, for which we can strengthen existing first-order results in the literature to the second order by applying Theorem \ref{thm: main result} and Theorem \ref{thm: min ESS characterization}. Since the arguments are similar to those for Class \ref{problem: gmr}--\ref{problem: kns}, we omit the details.

First, we assume that the number of signals is bounded between $\ell$ and $u$, and control the classical familywise error rates \cite{song2017asymptotically,he2021asymptotically}. 

\begin{problem}\label{problem: bns}
Let $\ell,u$ be integers such that $0 \le \ell < u \le K$. For tolerance levels $\alpha,\beta \in (0,1)$, we define a class of procedures controlling familywise error rates under this setting as
\begin{equation*}
\begin{aligned}
&\mathcal{A} = \{A \subset [K] : \ell \le |A| \le u\},\\
&\Delta_{\ell,u}^{\text{bns}}(\alpha,\beta)
=
\left\{
(T,D) \in \Delta:
\max_{A \in \mathcal{A}} \prob_A\bigl(|D \setminus A| \geq 1\bigr) \le \alpha,
\;\;
\max_{A \in \mathcal{A}} \prob_A\bigl(|A \setminus D| \geq 1\bigr) \le \beta
\right\}.
\end{aligned}
\end{equation*}
\end{problem}

\begin{remark}
When $\ell = 0$ and $u = K$, the class $\Delta_{\ell,u}^{\text{bns}}(\alpha,\beta)$ coincides with $\Delta_{1,1}^{\mathrm{gfr}}(\alpha,\beta)$; see Class \ref{problem: gfr}.
\end{remark}

In \cite{song2017asymptotically}, a procedure called the Gap-intersection rule, denoted by $\delta_{GI}(a,b)$, is proposed, where $a,b>0$ are thresholds. It is shown in \cite{song2017asymptotically} that, with a proper choice of $a,b$, $\delta_{GI}(a,b)$ belongs to $\Delta_{\ell,u}^{\text{bns}}(\alpha,\beta)$. In addition, as $\alpha \vee \beta \to 0$, the family $\{\delta_{GI}(a,b)\}$ is first-order optimal, and the first-order minimum ESS is characterized for each $A \in \mathcal{A}$.

Now, define the following \emph{zero-one} loss (see Remark \ref{remark: zero-one loss}):
\begin{equation*}
 \mW(D\mid A)=
\begin{cases}
    1 &\;\text{ if }\;\; D \neq A \\
    0 &\;\text{ else} 
\end{cases}, \quad \text{ for } A, D \subset [K] \text{ with } \ell \leq |A|, |D| \leq u,
\end{equation*} 
By applying Theorem \ref{thm: main result} and Theorem \ref{thm: min ESS characterization} with this loss function $\mW$, we can establish the second-order optimality of $\{\delta_{GI}(a,b)\}$ and characterize $T_A^{\min}(\Delta_{\ell,u}^{\text{bns}}(\alpha,\beta))$ for each $A \in \mathcal{A}$ as $\alpha \vee \beta \to 0$ such that \eqref{eq:alpha_beta_same_rate} holds.

Second, we control the false discovery rate (FDR) and the false non-discovery rate (FNR), and assume that the number of signals is known to be either $m$ or between $\ell$ and $u$ \cite{he2021asymptotically}. Recall the definition of $\Delta^{\mathrm{fdr}}(\alpha,\beta)$  in Class \ref{problem: fdr}.

\begin{problem}\label{problem: kns fdr fnr}
Given tolerance levels $\alpha,\beta \in (0,1)$ and an integer $m \in [1,K)$, define
\begin{equation*}
\mathcal{A} := \{A \subset [K]: |A| = m\},\quad \text{ and }\quad
\Delta_{m}^{\mathrm{fdr}}(\alpha,\beta)
:= \Delta^{\mathrm{fdr}}(\alpha,\beta).
\end{equation*}
\end{problem}

\begin{problem}\label{problem: bns fdr fnr}
Given tolerance levels $\alpha,\beta \in (0,1)$ and  integers $0\leq \ell < u \leq K$, define
\begin{equation*}
\mathcal{A} = \{A \subset [K] : \ell \le |A| \le u\},\quad \text{ and }\quad 
\Delta_{\ell,u}^{\mathrm{fdr}}(\alpha,\beta)
= \Delta^{\mathrm{fdr}}(\alpha,\beta).
\end{equation*}
\end{problem}

In \cite{he2021asymptotically}, it is shown that the Gap rule in \eqref{eq: gap rule} (resp.~the Gap-intersection rule mentioned above), with properly designed thresholds, is first-order asymptotically optimal for the class $\Delta_{m}^{\mathrm{fdr}}(\alpha,\beta)$ (resp.~$\Delta_{\ell,u}^{\mathrm{fdr}}(\alpha,\beta)$). Moreover, the minimum ESS is characterized to the first order. By arguments similar to those in Subsection \ref{sec: fdr}, we can apply Theorem \ref{thm: main result} to establish the corresponding second-order optimality, and Theorem \ref{thm: min ESS characterization} to obtain a second-order accurate approximation to the minimum ESS.

Finally, we note that \cite{he2021asymptotically} also considers control of the \emph{marginal} FDR and \emph{marginal} FNR under prior knowledge about the number of signals. We refer the interested reader there for the definitions and first-order optimal procedures, and note that our results apply to the corresponding classes of sequential procedures as well.

\section{Proofs under Specific Error Metrics and Information Structure}\label{sec: proof of optimality}
\subsection{Results regarding the Lorden's rule}
In this subsection, let $\mW : \mathcal{A} \times \mathcal{A} \to [0,\infty)$ be a general loss function, and let $\pi_0$ be the uniform distribution on $\mathcal{A}$.
 We present results regarding the rule
$\delta_{\mathrm{Ld}}(c,\mW) = (T_{\mathrm{Ld}}(c,\mW), D_{\mathrm{Ld}}(c,\mW))$
in \eqref{eq: lorden procedure}, which will be used later.

\begin{lemma}\label{lemma: A_0}
Fix a cost $c > 0$. On the event $\{T_{\mathrm{Ld}} < \infty\}$, for each $A_0 \in \mathcal{A}$,
\[
\sum_{i \in D_0} \lambda_{T_{\mathrm{Ld}}}^i
-
\sum_{j \in A_0} \lambda_{T_{\mathrm{Ld}}}^j
>
\log \left[c^{-1}\pi_0(A_0)\mW(D_{\mathrm{Ld}} \mid A_0)\right],
\]
where $D_0 := \{k \in [K] : \lambda_{T_{\mathrm{Ld}}}^k \ge 0\}$ is the set of stream indices with nonnegative LLRs at stopping, $\log(0) := -\infty$ and $(T_{\mathrm{Ld}}, D_{\mathrm{Ld}})$ is short for $(T_{\mathrm{Ld}}(c,\mW), D_{\mathrm{Ld}}(c,\mW))$.
\end{lemma}
\begin{proof}
Recall LLRs $\{\lambda_t^k: k\in [K]\}$ in \eqref{eq:LLR}. We focus on the event $\{T_{\mathrm{Ld}}(c,\mW) < \infty\}$.
For each stream $k \in [K]$ and subset $A \subset [K]$, we define
$$
\Lambda_t^k := \exp\{\lambda_t^k\} = \prod_{s=1}^t  \frac{f_1^k(\mathbf{X}_s^k)}{f_0^k(\mathbf{X}_s^k)}.,
\quad \text{ and } \quad
\Lambda_t^A := \prod_{k \in A} \Lambda_t^k.
$$
Due to the definition of $D_0$, we have 
\begin{equation}
    \label{obs:D0_prop}
    \Lambda_{T_{\mathrm{Ld}}}^A \leq \Lambda_{T_{\mathrm{Ld}}}^{D_0},\quad \text{ for each } A \in \mathcal{A}.
\end{equation}

By the definition of $\delta_{\mathrm{Ld}}(c,\mW)$ in \eqref{eq: lorden procedure}, we have
$
\sum_{A\in\mathcal{A}}\pi_{T_{\mathrm{Ld}}}(A)\mW(D_{\mathrm{Ld}}\mid A) < c,
$
which, in view of the definition of $\pi_t(A)$ in \eqref{eq: posterior}, is equivalent to
$$
\sum_{A\in\mathcal{A}}f_{A,[T_{\mathrm{Ld}}]}(A)\mW(D_{\mathrm{Ld}}\mid A) < c \sum_{B\in\mathcal{A}}f_{B,[T_{\mathrm{Ld}}]}.
$$ 
Moreover, dividing both sides by $\prod_{s=1}^{T_{\mathrm{Ld}}} \prod_{k \in [K]} f_0^k(\mathbf{X}_s^k)$, we obtain
\begin{equation*}
\sum_{A\in\mathcal{A}} \Lambda_{T_{\mathrm{Ld}}}^A \mW(D_{\mathrm{Ld}}\mid A) < c \sum_{B\in\mathcal{A}} \Lambda_{T_{\mathrm{Ld}}}^B.
\end{equation*}

Fix an arbitrary $A_0\in \mathcal{A}$. On one hand,
$$\sum_{A\in\mathcal{A}} \Lambda_{T_{\mathrm{Ld}}}^A \mW(D_{\mathrm{Ld}}\mid A)\ge \Lambda_{T_{\mathrm{Ld}}}^{A_0}\mW(D_{\mathrm{Ld}}\mid A_0).$$
On the other, due to \eqref{obs:D0_prop}, we have
$$
\sum_{B\in\mathcal{A}}\Lambda_{T_{\mathrm{Ld}}}^B \le \sum_{B\in\mathcal{A}}\Lambda_{T_{\mathrm{Ld}}}^{D_0} = |\mathcal{A}|\Lambda_{T_{\mathrm{Ld}}}^{D_0} = (\pi_0(A))^{-1}\Lambda_{T_{\mathrm{Ld}}}^{D_0}.
$$
Therefore, combining the above three inequalities, we obtain
\[
\frac{\Lambda_{T_{\mathrm{Ld}}}^{D_0}}{\Lambda_{T_{\mathrm{Ld}}}^{A_0}}
>
c^{-1}\pi_0(A_0)\mW(D_{\mathrm{Ld}} \mid A_0),
\]
and the proof is completed by taking logarithms on both sides.
\end{proof}

Recall the definition of $\bar{j}_k(t)$ in Appendix \ref{app:kns}.
\begin{lemma}\label{lemma: |A_0| known}
Fix a cost $c > 0$. Suppose that $\mathcal{A} = \{A \subset [K]: |A| = m\}$, where $m \in [1,K)$ is an integer. On the event that $\{T_{\mathrm{Ld}} < \infty\}$, for each $A_0 \in \mathcal{A}$, we have
$$
\sum_{i\in D_0}\lambda_{T_{\mathrm{Ld}}}^i-\sum_{j\in A_0}\lambda_{T_{\mathrm{Ld}}}^j > \log [c^{-1}\pi_0(A_0)\mW(D_{\mathrm{Ld}}\mid A_0)],
$$
where $D_0 := \{\bar{j}_{K-m+1}(T_{\mathrm{Ld}}),\ldots,\bar{j}_{K}(T_{\mathrm{Ld}})\}$ is the collection of stream indices with $m$ largest LLRs at $T_{\mathrm{Ld}}$,  $\log(0) := -\infty$ and $(T_{\mathrm{Ld}}, D_{\mathrm{Ld}})$ is short for $(T_{\mathrm{Ld}}(c,\mW), D_{\mathrm{Ld}}(c,\mW))$.
\end{lemma}
\begin{proof}
Due to the definition of $D_0$, we have 
\begin{equation*}
    \Lambda_{T_{\mathrm{Ld}}}^A \leq \Lambda_{T_{\mathrm{Ld}}}^{D_0},\quad \text{ for each } A \in \mathcal{A}.
\end{equation*}
The remaining arguments are similar to those in the proof of Lemma \ref{lemma: A_0}, and are therefore omitted.
\end{proof}

\subsection{Generalized Misclassification Rates}\label{app: proof of gmr}
In this subsection, we present the proofs of the results in Subsection \ref{sec: gmr}. 
We recall the definition of the stopping rule $T_S(b)$ in \eqref{eq: sum intersection rule}, the loss function $\mW$ in \eqref{eq: loss gmr}, and the associated rule $T_{\mathrm{Ld}}(c,\mW)$ in \eqref{eq: lorden procedure}. Furthermore, $\pi_0$ denotes the uniform distribution on $\mathcal{A} = 2^{[K]}$.

\begin{lemma} \label{lemma: stop earlier sum intersection}
Let the cost $c\in (0, 2^{-K})$ and $b =  \log(2^{-K}c^{-1})$. Then,   
$$T_S(b)\le T_{\mathrm{Ld}}(c,\mW).$$
\end{lemma}
\begin{proof}
Since $c$, $b$, and $\mW$ are fixed, we suppress their dependence in $T_S(b)$ and $T_{\mathrm{Ld}}(c,\mW)$.
Suppose that $T_{\mathrm{Ld}} < \infty$ since otherwise the conclusion trivially holds. We show that the stopping criterion of $T_S$ is met at time $T_{\mathrm{Ld}}$, that is,
\begin{equation}
    \label{gmr:aux_comp}
\sum_{i=1}^{m_0}\tlambda_{T_{\mathrm{Ld}}}^{(i)} \geq b = \log(2^{-K} c^{-1}),
\end{equation}
where we recall the ordered statistics $\{\tlambda_t^{(i)}: i\in[K]\}$ in \eqref{eq: ordered statistics sum intersection}.
This immediately implies $T_S \le T_{\mathrm{Ld}}$.

For each $i \in [K]$, let $\ell_i$ denote the index of the stream attaining $\tilde{\lambda}_t^{(i)}$, that is,
\[
|\lambda_t^{\ell_i}| = \tilde{\lambda}_t^{(i)},
\]
where ties are broken arbitrarily. Further, define $B_0 := \{\ell_1, \ldots, \ell_{m_0}\}$ as the set of the $m_0$ streams with the smallest absolute LLRs. Let $D_0 := \{k \in [K] : \lambda_{T_{\mathrm{Ld}}}^k \ge 0\}$ denote the set of streams with nonnegative LLRs at time $T_{\mathrm{Ld}}$.

We define the following subset 
$$A_0 := (B_0\cap D_{\mathrm{Ld}}^c)\cup(B_0^c\cap D_0),$$
as illustrated in Figure \ref{fig: A_0 sum intersection}. We claim that
\begin{equation}\label{eq: B_0 2}
    |A_0 \triangle D_{\mathrm{Ld}}|\ge m_0 \quad \text{ and } \quad A_0\cap B_0^c = D_0\cap B_0^c.
\end{equation}
The second part above follows directly from De Morgan's laws. For the first part, it suffices to prove that $A_0 \triangle D_{\mathrm{Ld}} \supseteq B_0$. To this end, fix an arbitrary $k \in B_0$. By the definition of $A_0$, if $k \in A_0$, then $k \in D_{\mathrm{Ld}}^{c}$, i.e., $k \in A_0 \setminus D_{\mathrm{Ld}}$; if $k \notin A_0$, then $k \notin D_{\mathrm{Ld}}^{c}$, i.e., $k \in D_{\mathrm{Ld}} \setminus A_0$. Therefore, $k \in A_0 \triangle  D_{\mathrm{Ld}}$.

\begin{figure}[!t]
\centering
\resizebox{200pt}{!}{
\begin{circuitikz}
\tikzstyle{every node}=[font=\Large]
\draw  (6.25,22) rectangle (15.75,21.25);
\draw [short] (11.5,22) -- (11.5,21.25);
\draw [dashed] (11.5,21.25) -- (11.5,16);
\node [font=\Large] at (13.5,22.5) {$D_{\mathrm{Ld}}$};
\node [font=\Large] at (8.75,22.5) {$D_{\mathrm{Ld}}^c$};
\node [font=\Large] at (8.75,21.6) {$0$};
\node [font=\Large] at (13.5,21.6) {$1$};

\draw  (6.25,20) rectangle (15.75,19.25);
\draw [<->, >=Stealth, dashed] (9.5,19) -- (13.5,19);
\node [font=\Large] at (10.75,18.5) {$D_0$};
\draw [short] (9.5,20) -- (9.5,19.25);
\draw [short] (13.5,20) -- (13.5,19.25);
\node [font=\Large] at (8,19.6) {$-$};
\node [font=\Large] at (10.5,19.6) {$+$};
\node [font=\Large] at (12.5,19.6) {$+$};
\node [font=\Large] at (14.5,19.6) {$-$};
\draw [dashed] (9.5,19.25) -- (9.5,16);
\draw [dashed] (13.5,19.25) -- (13.5,16);

\draw  (6.25,18) rectangle (15.75,17.25);
\node [font=\Large] at (10.75,16.5) {$B_0$};
\draw [<->, >=Stealth, dashed] (8.5,17) -- (12.5,17);
\draw [short] (8.5,18) -- (8.5,17.25);
\draw [short] (12.5,18) -- (12.5,17.25);
\draw [dashed] (8.5,17.25) -- (8.5,16);
\draw [dashed] (12.5,17.25) -- (12.5,16);

\draw  (6.25,16) rectangle (15.75,15.25);
\draw [short] (8.5,16) -- (8.5,15.25);
\draw [short] (9.5,16) -- (9.5,15.25);
\draw [short] (11.5,16) -- (11.5,15.25);
\draw [short] (12.5,16) -- (12.5,15.25);
\draw [short] (13.5,16) -- (13.5,15.25);
\node [font=\Large] at (7.5,15.6) {$0$};
\node [font=\Large] at (9,15.6) {$1$};
\node [font=\Large] at (10.5,15.6) {$1$};
\node [font=\Large] at (12,15.6) {$0$};
\node [font=\Large] at (13,15.6) {$1$};
\node [font=\Large] at (14.5,15.6) {$0$};
\draw [<->, >=Stealth, dashed] (8.5,15) -- (11.5,15);
\draw [<->, >=Stealth, dashed] (12.5,15) -- (13.5,15);
\draw [dashed] (10.5,15) -- (11.25,14.55);
\draw [dashed] (13,15) -- (12.25,14.55);
\node [font=\Large] at (11.75,14.5) {$A_0$};

\end{circuitikz}
}
\caption{Visualization of the subsets appearing in the proof of the second-order optimality of the Sum-Intersection rule.}
 \label{fig: A_0 sum intersection}
\end{figure}
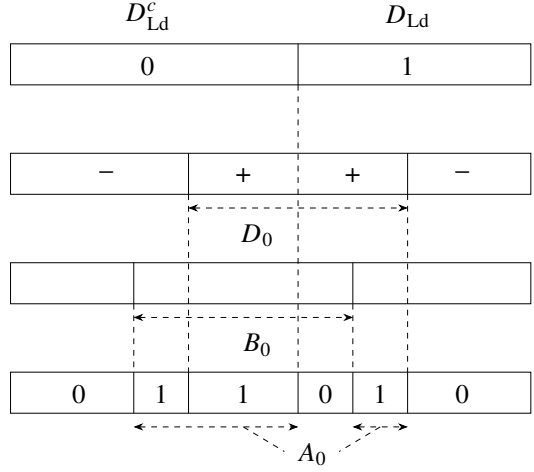

Next, by Lemma \ref{lemma: A_0}, we have
$$\sum_{i\in D_0}\lambda_{T_{\mathrm{Ld}}}^i- \sum_{j\in A_0}\lambda_{T_{\mathrm{Ld}}}^j > \log [c^{-1}\pi_0(A_0)\mW(D_{\mathrm{Ld}}\mid A_0)].$$ 
Due to the second part in \eqref{eq: B_0 2}, we have
\begin{align*}
\sum_{i\in D_0}\lambda_{T_{\mathrm{Ld}}}^i- \sum_{j\in A_0}\lambda_{T_{\mathrm{Ld}}}^j = \sum_{i\in D_0\cap B_0}\lambda_{T_{\mathrm{Ld}}}^i - \sum_{j\in A_0\cap B_0}\lambda_{T_{\mathrm{Ld}}}^j
\le \sum_{k\in B_0}|\lambda_{T_{\mathrm{Ld}}}^k| = \sum_{i=1}^{m_0}\tlambda_{T_{\mathrm{Ld}}}^{(i)}.
\end{align*}
Due to the first part in \eqref{eq: B_0 2} and the definition of $\mW$ in \eqref{eq: loss gmr}, we have $\mW(D_{\mathrm{Ld}}\mid A_0) = 1$. This proves \eqref{gmr:aux_comp} and completes the proof since $\pi_0(A_0) = 2^{-K}$.
\end{proof}

Recall the definition of $\KL(f_A \mid f_C)$ in \eqref{def:KL_fA_fC}, of $\mathcal{H}_D^{\mW}$ in \eqref{def:H_D_W}, and of $\KL_{A,D}^{\mW}$ in \eqref{def:KL_A_D_star}.

\begin{proof}[Proof of Lemma \ref{lemma: misclassification rate satisfies assumption}]
Fix an arbitrary $A\subset [K]$. Let $\tilde{C} \in \mathcal{H}_{A}^{\mW}$ (i.e., $|\tilde{C} \triangle A| \geq m_0$) be such that $\KL(f_A \mid f_{\tilde{C}}) = \KL_{A,A}^{\mW}$. It is sufficient to show that for any $D\neq A$, there exists a subset $C^*\in \mathcal{H}_{D}^{\mW}$ (i.e., $|C^* \triangle  D| \geq m_0$) such that
\begin{equation}\label{eq: misclassification rate satisfies assumption 2}
\KL(f_A\mid f_{C^*})<\KL(f_A\mid f_{\tilde{C}}).
\end{equation}

We denote
\begin{equation*}
\begin{aligned}
&L_1 = A \setminus D,\qquad L_2 = A \cap D,\qquad L_3 = D \setminus A,\qquad L_4 = D^{c} \setminus A,\\
&\Gamma_1 = L_1 \setminus \tilde{C},\qquad \Gamma_2 = L_2 \setminus \tilde{C},\qquad 
\Gamma_3 = L_3 \cap \tilde{C},\qquad \Gamma_4 = L_4 \cap \tilde{C},
\end{aligned}
\end{equation*}
as illustrated in Figure \ref{fig: C^* and tilde C}. Then  it follows that
\begin{equation*}
A \setminus \tilde{C} = \Gamma_1 \cup \Gamma_2,\qquad
\tilde{C} \setminus A = \Gamma_3 \cup \Gamma_4,\qquad
\tilde{C} \triangle A =
\Gamma_1 \cup \Gamma_2 \cup \Gamma_3 \cup \Gamma_4,\qquad D \triangle A = L_1 \cup L_3.
\end{equation*}
Since   $D \neq A$, and $|\tilde{C} \triangle A| \geq m_0$, we have
\begin{equation}\label{eq: misclassification rate satisfies assumption 0}
|L_1| + |L_3| \geq 1,\qquad
|\Gamma_1| + |\Gamma_2| + 
|\Gamma_3| + 
|\Gamma_4| \geq m_0. 
\end{equation}

\noindent \textbf{Case 1.}  $\Gamma_1\cup\Gamma_3 \neq  \emptyset$. In this case, we define 
$$
C^* = L_1\cup(L_2\setminus \Gamma_2)\cup  \Gamma_4,
$$
and it follows that
\begin{equation*}
A\setminus C^* = {\Gamma}_2, \qquad   C^*\setminus A = {\Gamma}_4,  \qquad
C^*\triangle D = L_1\cup {\Gamma}_2\cup L_3\cup {\Gamma}_4.
\end{equation*}
On one hand, since $L_1,L_2,L_3,L_4$ are disjoint, we have 
\begin{align*}
|C^*\triangle D|  = |L_1| + |{\Gamma}_2| + |L_3| + |{\Gamma}_4| \ge |\Gamma_1| + |\Gamma_2| + |\Gamma_3| + |\Gamma_4| \ge m_0,
\end{align*}
where the last inequality is due to \eqref{eq: misclassification rate satisfies assumption 0}. On the other hand, in view of \eqref{def:KL_fA_fC},
\begin{align*}
\KL(f_A\mid f_{\tilde{C}}) - \KL(f_A\mid f_{C^*})& = \sum_{k\in A\setminus \tilde{C}}\mathcal{I}_k^1 + \sum_{k\in \tilde{C}\setminus A}\mathcal{I}_k^0 - (\sum_{k\in A\setminus C^*}\mathcal{I}_k^1 + \sum_{k\in C^*\setminus A}\mathcal{I}_k^0)\\
& = 
\sum_{k\in\Gamma_1}\mathcal{I}_k^1 + \sum_{k\in \Gamma_3}\mathcal{I}_k^0
>0,
\end{align*}
where the last inequality is due to $\Gamma_1 \cup \Gamma_3 \neq \emptyset$ and Assumption \ref{assumption: lorden}. Thus \eqref{eq: misclassification rate satisfies assumption 2} is established in Case 1.

\medskip

\noindent \textbf{Case 2}. $\Gamma_1 \cup \Gamma_3 = \emptyset$. In this case, due to \eqref{eq: misclassification rate satisfies assumption 0}, we have
\begin{equation} \label{aux:G2G4_m0}
    |\Gamma_2| + |\Gamma_4| \geq m_0.
\end{equation}
If  $|L_1| + |L_3| \geq |\Gamma_2| + |\Gamma_4|$, we define $\tilde{\Gamma}_2 = \tilde{\Gamma}_4 = \emptyset$. Otherwise,
let
$\tilde{\Gamma}_2 \subset \Gamma_2$ and $\tilde{\Gamma}_4 \subset \Gamma_4$ be such that
$|\tilde{\Gamma}_2| + |\tilde{\Gamma}_4|
=
|\Gamma_2| + |\Gamma_4| - |L_1| - |L_3|
$.
 Due to \eqref{aux:G2G4_m0} and \eqref{eq: misclassification rate satisfies assumption 0}, we always have
\begin{align}\label{aux:diff_non_empty}
    \left(\Gamma_2 \setminus \tilde{\Gamma}_2\right) \cup
     \left(\Gamma_4 \setminus \tilde{\Gamma}_4\right) \neq \emptyset.
\end{align}
Now, we define 
$$
C^* = L_1\cup(L_2\setminus\tilde{\Gamma}_2)\cup \tilde{\Gamma}_4,
$$
and it follows that
\begin{equation*}
A\setminus C^* = \tilde{\Gamma}_2, \qquad   C^*\setminus A = \tilde{\Gamma}_4,  \qquad
C^*\triangle D = L_1\cup\tilde{\Gamma}_2\cup L_3\cup\tilde{\Gamma}_4.
\end{equation*}
On one hand, since $L_1,L_2,L_3,L_4$ are disjoint and due to \eqref{aux:G2G4_m0}, we have 
\begin{align*}
|C^*\triangle D| = |L_1| + |\tilde{\Gamma}_2| + |L_3| + |\tilde{\Gamma}_4| \ge  |\Gamma_2| + |\Gamma_4| \ge m_0,
\end{align*}
On the other hand, in view of \eqref{def:KL_fA_fC},
\begin{align*}
\KL(f_A\mid f_{\tilde{C}}) - \KL(f_A\mid f_{C^*}) &= \sum_{k\in A\setminus \tilde{C}}\mathcal{I}_k^1 + \sum_{k\in \tilde{C}\setminus A}\mathcal{I}_k^0 - (\sum_{k\in A\setminus C^*}\mathcal{I}_k^1 + \sum_{k\in C^*\setminus A}\mathcal{I}_k^0) \\
& = 
\sum_{k\in\Gamma_2\setminus\tilde{\Gamma}_2}\mathcal{I}_k^1 + \sum_{k\in \Gamma_4\setminus\tilde{\Gamma}_4}\mathcal{I}_k^0 
>0,
\end{align*}
where the last inequality follows from \eqref{aux:diff_non_empty} and Assumption \ref{assumption: lorden}. Thus \eqref{eq: misclassification rate satisfies assumption 2} is established in Case 2, and the proof is complete.
\end{proof}

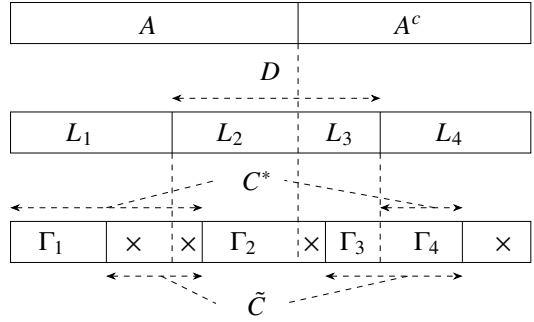
\begin{figure}[!t]
\centering
\resizebox{200pt}{!}{
\begin{circuitikz}
\tikzstyle{every node}=[font=\Large]
\draw  (6.25,22) rectangle (15.75,21.25);
\draw [short] (11.5,22) -- (11.5,21.25);
\draw [dashed] (11.5,21.25) -- (11.5,17.25);
\node [font=\Large] at (8.75,21.6) {$A$};
\node [font=\Large] at (13.5,21.6) {$A^c$};

\draw  (6.25,20) rectangle (15.75,19.25);
\draw [<->, >=Stealth, dashed] (9.2,20.25) -- (13,20.25);
\draw [short] (9.2,20) -- (9.2,19.25);
\draw [short] (13,20) -- (13,19.25);
\draw [dashed] (9.2,19.25) -- (9.2,17.25);
\draw [dashed] (13,19.25) -- (13,17.25);
\node [font=\Large] at (7.5,19.6) {$L_1$};
\node [font=\Large] at (10.25,19.6) {$L_2$};
\node [font=\Large] at (12.25,19.6) {$L_3$};
\node [font=\Large] at (14.25,19.6) {$L_4$};
\node [font=\Large] at (11,20.75) {$D$};

\draw  (6.25,18) rectangle (15.75,17.25);
\draw [short] (8,18) -- (8,17.25);
\draw [short] (9.75,18) -- (9.75,17.25);
\draw [short] (12,18) -- (12,17.25);
\draw [short] (14.5,18) -- (14.5,17.25);
\node [font=\Large] at (7,17.6) {$\Gamma_1$};
\node [font=\Large] at (8.5,17.6) {$\times$};
\node [font=\Large] at (9.5,17.6) {$\times$};
\node [font=\Large] at (10.5,17.6) {$\Gamma_2$};
\node [font=\Large] at (11.75,17.6) {$\times$};
\node [font=\Large] at (12.5,17.6) {$\Gamma_3$};
\node [font=\Large] at (13.85,17.6) {$\Gamma_4$};
\node [font=\Large] at (15.25,17.6) {$\times$};
\draw [<->, >=Stealth, dashed] (6.25,18.25) -- (9.75,18.25);
\draw [<->, >=Stealth, dashed] (13,18.25) -- (14.5,18.25);
\draw [dashed] (8,18.25) -- (10,18.75);
\draw [dashed] (14,18.25) -- (11.25,18.75);
\node [font=\Large] at (10.75,18.75) {$C^*$};
\draw [<->, >=Stealth, dashed] (8,17) -- (9.75,17);
\draw [<->, >=Stealth, dashed] (12,17) -- (14.5,17);
\draw [dashed] (9,17) -- (10,16.5);
\draw [dashed] (13.5,17) -- (11.25,16.5);
\node [font=\Large] at (10.75,16.5) {$\tilde{C}$};
\end{circuitikz}
}
\caption{Visualization of the subsets appearing in the proof of the ESS approximation under the generalized misclassification rate.}
\label{fig: C^* and tilde C}
\end{figure}

\subsection{Generalized Familywise Error Rates}\label{sec: proof of gfr}
In this subsection, we present the proofs of the results in Subsection \ref{sec: leap rule}. 
We recall the definition of the stopping rule $T_L(a,b)$ in \eqref{eq: leap rule}, the loss function $\mW$ in \eqref{eq: loss gfr}, and the associated rule $T_{\mathrm{Ld}}(c,\mW)$ in \eqref{eq: lorden procedure}. Furthermore, $\pi_0$ denotes the uniform distribution on $\mathcal{A} = 2^{[K]}$.

\begin{proof}[Proof of Theorem \ref{thm: second-order optimal of leap rule}] 
Let $\bell := (\alpha,\beta)$, and we define
\begin{align}\label{eq: leap_constants_def}
\begin{split}
c_{\bell} &:= 2^{-K} e^{-(a_{\bell} \vee b_{\bell})} = 2^{-K}\min\left\{2^{-m_2}\binom{K}{m_2}^{-1}\beta, 2^{-m_1}\binom{K}{m_1}^{-1}\alpha\right\}, \\
L &:= (1+C) 2^{K} \max\left\{2^{m_2}\binom{K}{m_2}, 2^{m_1} \binom{K}{m_1}\right\},
\end{split}
\end{align}
where we recall $a_{\bell}, b_{\bell}$ are defined in \eqref{eq: leap_thresholds} and the constant $C$ in \eqref{eq:alpha_beta_same_rate}.  
Then the conclusion follows immediately from Theorem \ref{thm: main result} once we verify condition \eqref{assumption: lorden 1}--\eqref{eq: second-order optimal sufficient condition 2} hold with the loss function $\mW$ in \eqref{eq: loss gfr}, cost $c_{\bell}$ and constant $L$.

We start with condition \eqref{assumption: lorden 1}. For each $D \subset [K]$, let $A = D^{c}$. Then $D\setminus A = D$ and $A\setminus D = D^{c}$. Since $|D| + |D^{c}| = K \ge m_1 + m_2$, it follows that either $|D\setminus A| \ge m_1$ or $|A\setminus D| \ge m_2$. By the definition of $\mW$ in \eqref{eq: loss gfr}, we therefore have $\mW(D\mid A) > 0$. Hence, condition \eqref{assumption: lorden 1} holds.

Next, we verify condition \eqref{eq: second-order optimal sufficient condition 1}. By the definition of $c_{\bell}$ in \eqref{eq: leap_constants_def}, we have
$$
\max\{a_{\bell}, b_{\bell}\} = \log\left(2^{-K} c_{\bell}^{-1}\right).
$$
Then, by Lemma \ref{lemma: stop earlier leap}, $
T_L(a_{\bell}, b_{\bell}) \le T_{\mathrm{Ld}}(c_{\bell},\mW)$, 
which verifies condition \eqref{eq: second-order optimal sufficient condition 1}.

Finally, we focus on condition \eqref{eq: second-order optimal sufficient condition 2}. Fix an arbitrary procedure $\delta=(T,D)\in\Delta_{m_1,m_2}^{\textup{gfr}}(\bell)$ in \eqref{problem: gfr}. By its definition and the definition of $\mW$ in \eqref{eq: loss gfr}, 
\begin{align*}
    \ie(\delta; \mW) &\le \sum_{A\in\mathcal{A}}\pi_0(A) \left(\prob_A(|D\setminus A|\ge m_1)+\prob_A(|A\setminus D|\ge m_2)\right) 
    \\&\le  (\alpha+\beta)\sum_{A\in\mathcal{A}}\pi_0(A) =  (\alpha+\beta).
\end{align*}
Note that due to condition \eqref{eq:alpha_beta_same_rate}
$$
\alpha + \beta \leq (1+C) \min\{\alpha,\beta\} \leq L c_{\bell},
$$
where the last inequality is due to the definition of $c_{\bell}$ and $L$ in \eqref{eq: leap_constants_def}. Thus, condition \eqref{eq: second-order optimal sufficient condition 2} is verifies, the proof is complete.
\end{proof}

 \begin{lemma}\label{lemma: stop earlier leap}
Let the cost $c\in (0, 2^{-K})$, and assume that the thresholds $0< a,b\le \log(2^{-K} c^{-1})$. Then
$$T_L(a,b)\le T_{\mathrm{Ld}}(c,\mW).$$
\end{lemma}
\begin{proof} 
Since $c$, $a,b$, and $\mW$ are fixed, we suppress their dependence in $T_L(a,b)$ and $T_{\mathrm{Ld}}(c,\mW)$.
Suppose that $T_{\mathrm{Ld}} < \infty$ since otherwise the conclusion trivially holds. We show that the stopping criterion of $T_L$ is met at time $T_{\mathrm{Ld}}$, that is, one of the following holds: either (i) for some $0 \leq i < m_1$,
\begin{equation}
    \label{fmr:aux_comp1}
    \sum_{j=1}^{m_1-i}\hlambda_{T_{\mathrm{Ld}}}^{(j)} \ge \log(2^{-K} c^{-1}),\quad \text{ and } \quad  
\sum_{j=i+1}^{i+m_2}\clambda_{T_{\mathrm{Ld}}}^{(j)} \ge \log(2^{-K} c^{-1}),
\end{equation}
or (ii) for some $0 \leq i < m_2$,
\begin{equation*}
\sum_{j= i+1}^{i+m_1}\hlambda_{T_{\mathrm{Ld}}}^{(j)} \ge \log(2^{-K} c^{-1}),\quad \text{ and } \quad   
\sum_{j=1}^{m_2- i}\clambda_{T_{\mathrm{Ld}}}^{(j)} \ge \log(2^{-K} c^{-1}) \end{equation*}
where we recall the ordered statistics $\{\hlambda_t^{(i)}\}$ and $\{\clambda_t^{(i)}\}$, and their corresponding stream indices $\{\hl_i(t),\cl_j(t)\}$ in Subsection \ref{sec: leap rule}.
This immediately implies $T_L \le T_{\mathrm{Ld}}$.

 Recall the LLRs $\{\lambda_t^k: k\in[K]\}$ in \eqref{eq:LLR}, and
 let $D_0 := \{k \in [K] : \lambda_{T_{\mathrm{Ld}}}^k \ge 0\}$ denote the set of streams with nonnegative LLRs at time $T_{\mathrm{Ld}}$. We define
\begin{equation}\label{eq: stop earlier leap 3}
\begin{aligned}
    &L_1 := \{k\in D_{\mathrm{Ld}}^c: \lambda_{T_{\mathrm{Ld}}}^k<0\} = D_{\mathrm{Ld}}^{c} \cap D_0^{c}, \quad L_2 := \{k\in D_{\mathrm{Ld}}^c: \lambda_{T_{\mathrm{Ld}}}^k\ge 0\} =   D_{\mathrm{Ld}}^{c} \cap D_0,
    \\& L_3 := \{k\in D_{\mathrm{Ld}}: \lambda_{T_{\mathrm{Ld}}}^k\ge 0\} =   D_{\mathrm{Ld}}  \cap D_0, \quad L_4 := \{k\in D_{\mathrm{Ld}}: \lambda_{T_{\mathrm{Ld}}}^k< 0\} =   D_{\mathrm{Ld}}  \cap D_0^{c}.
    \end{aligned}
\end{equation}
We focus on the case $|L_2| \le |L_4|$; the proof for the case $|L_2| > |L_4|$ is analogous and omitted. Denote 
$$
i^* := \min\{|L_4| - |L_2|,\;\; m_1 - 1\}.
$$

We denote by $B_0^+ = \{\hl_1(T_{\mathrm{Ld}}),\ldots,\hl_{m_1-i^*}(T_{\mathrm{Ld}})\}$ the set of stream indices corresponding to the first $m_1-i^*$ ordered statistics of the positive LLRs at time $T_{\mathrm{Ld}}$, and by $B_0^- = \{\cl_1(T_{\mathrm{Ld}}),\ldots,\cl_{m_2+i^*}(T_{\mathrm{Ld}})\}$ the set corresponding to the first $m_2+i^*$ ordered statistics of the absolute values of the nonpositive LLRs at time $T_{\mathrm{Ld}}$.  We define the following subsets:
 \begin{equation}\label{eq: stop earlier leap 4}
\begin{aligned}
&A_0^+ := L_2 \cup (L_3 \setminus B_0^+),\quad \text{ and } \quad  A_0^-:= \tilde{B}^- \cup L_2 \cup L_3,\quad \text{ where }
\\& \tilde{B}^- = 
\begin{cases}
B_0^-\cap L_1 &\text{ if } \;\;|B_0^-\cap L_1|\le m_2\\
\text{ any subset of } B_0^-\cap L_1 \text{ with cardinality } m_2 & \text{ otherwise}
\end{cases},
\end{aligned}
\end{equation}
as illustrated in Figure \ref{fig: A_0 leap}. By definition, we have 
\begin{equation}
  A_0^+ \subset D_0 \subset A_0^{-},\quad   D_0 \setminus A_0^+ = B_0^+\cap L_3, \quad \text{ and } \quad A_0^-\setminus D_0 = \tilde{B}^-. \label{aux:A_0_D_0}
\end{equation}
Furthermore, by Lemma \ref{lemma: leap rule claim}, 
\begin{equation}\label{eq: stop earlier leap 2}
|D_{\mathrm{Ld}}\setminus A_0^+| \ge m_1, \quad \text{ and }  \quad |A_0^-\setminus D_{\mathrm{Ld}}| \ge m_2.
\end{equation}

Now, due to Lemma \ref{lemma: A_0},  we have
$$\sum_{i\in D_0}\lambda_{T_{\mathrm{Ld}}}^i-\sum_{j\in A_0}\lambda_{T_{\mathrm{Ld}}}^j > \log[ c^{-1}\pi_0(A_0)\mW(D_{\mathrm{Ld}}\mid A_0)], \text{ for } A_0 \in \{A_0^+, A_0^-\}.
$$
On one hand, by the definition of $\mW$ in \eqref{eq: loss gfr}
and  \eqref{eq: stop earlier leap 2}, it follows that
$$
\mW(D_{\mathrm{Ld}}\mid A_0^{+}) = 1,\quad \text{ and} \quad 
\mW(D_{\mathrm{Ld}}\mid A_0^{-}) = 1.
$$
On the other hand,   due to   \eqref{aux:A_0_D_0},  it follows that
\begin{equation*}
\sum_{i\in D_0}\lambda_{T_{\mathrm{Ld}}}^i-\sum_{j\in A_0}\lambda_{T_{\mathrm{Ld}}}^j =  
\begin{cases}
    \sum_{k\in B_0^+\cap L_3}\lambda_{T_{\mathrm{Ld}}}^k & \text{ for } A_0=A_0^+\\
    \sum_{k\in \tilde{B}^-}|\lambda_{T_{\mathrm{Ld}}}^k| & \text{ for } A_0=A_0^-
\end{cases}.
\end{equation*}
As a result, due to the definition of $B_0^+, \tilde{B}^-$, we have 
\begin{align*}
   & \sum_{j=1}^{m_1-i^*}\hlambda_{T_{\mathrm{Ld}}}^{(j)} = \sum_{k\in B_0^+}\lambda_{T_{\mathrm{Ld}}}^k\ge \sum_{k\in B_0^+\cap L_3}\lambda_{T_{\mathrm{Ld}}}^k > \log[ c^{-1}\pi_0(A_0^{+})], \\
   &\sum_{j=1+i^*}^{m_2+i^*}|\clambda_{T_{\mathrm{Ld}}}^{(j)}| \ge \sum_{j\in\tilde{B}^-}|\clambda_{T_{\mathrm{Ld}}}^{(j)}| > \log[ c^{-1}\pi_0(A_0^{-})].
\end{align*}
Since $\pi(A) = 2^{-K}$ for each $A \subset [K]$, it implies that 
\eqref{fmr:aux_comp1} holds with $i = i^*$. The proof is complete.
\end{proof}

\begin{figure}[!t]
\centering
\resizebox{200pt}{!}{
\begin{circuitikz}
\tikzstyle{every node}=[font=\Large]
\draw  (6.25,22) rectangle (15.75,21.25);
\draw [short] (11.5,22) -- (11.5,21.25);
\draw [dashed] (11.5,21.25) -- (11.5,15.25);
\draw [dashed] (13,19.25) -- (13,16);
\node [font=\Large] at (13.5,22.5) {$D_{\mathrm{Ld}}$};
\node [font=\Large] at (8.75,22.5) {$D_{\mathrm{Ld}}^c$};
\node [font=\Large] at (8.75,21.6) {$0$};
\node [font=\Large] at (13.5,21.6) {$1$};

\draw  (6.25,20) rectangle (15.75,19.25);
\draw [<->, >=Stealth, dashed] (9.2,20.25) -- (13,20.25);
\draw [short] (9.2,20) -- (9.2,19.25);
\draw [short] (13,20) -- (13,19.25);
\draw [dashed] (9.2,19.25) -- (9.2,16);
\node [font=\Large] at (7.5,19.6) {$-$};
\node [font=\Large] at (10.25,19.6) {$+$};
\node [font=\Large] at (12.25,19.6) {$+$};
\node [font=\Large] at (14.25,19.6) {$-$};
\node [font=\Large] at (7.5,18.75) {$L_1$};
\node [font=\Large] at (10,18.75) {$L_2$};
\node [font=\Large] at (12.25,18.75) {$L_3$};
\node [font=\Large] at (14.5,18.75) {$L_4$};
\node [font=\Large] at (11,20.75) {$D_0$};

\draw  (6.25,18) rectangle (15.75,17.25);
\draw [short] (9.75,18) -- (9.75,17.25);
\draw [short] (12.5,18) -- (12.5,17.25);
\draw [dashed] (12.5,17.25) -- (12.5,16);
\node [font=\Large] at (8,17.6) {$-$};
\node [font=\Large] at (10.5,17.6) {$+$};
\node [font=\Large] at (12,17.6) {$+$};
\node [font=\Large] at (13.85,17.6) {$-$};
\draw [<->, >=Stealth, dashed] (9.75,17) -- (12.5,17);
\node [font=\Large] at (10.75,16.5) {$B_0^+$};

\draw  (6.25,16) rectangle (15.75,15.25);
\draw [short] (9.2,16) -- (9.2,15.25);
\draw [short] (11.5,16) -- (11.5,15.25);
\draw [short] (12.5,16) -- (12.5,15.25);
\draw [short] (13,16) -- (13,15.25);
\node [font=\Large] at (7,15.6) {0};
\node [font=\Large] at (10,15.6) {1};
\node [font=\Large] at (12,15.6) {0};
\node [font=\Large] at (12.75,15.6) {1};
\node [font=\Large] at (13.85,15.6) {0};
\draw [<->, >=Stealth, dashed] (9.2,15) -- (11.5,15);
\draw [<->, >=Stealth, dashed] (12.5,15) -- (13,15);
\draw [dashed] (10.5,15) -- (11.25,14.55);
\draw [dashed] (12.75,15) -- (12.25,14.55);
\node [font=\Large] at (11.75,14.5) {$A_0^+$};
\end{circuitikz}
}
\hspace{0.14cm}
\resizebox{200pt}{!}{
\begin{circuitikz}
\tikzstyle{every node}=[font=\Large]
\draw  (6.25,22) rectangle (15.75,21.25);
\draw [short] (11.5,22) -- (11.5,21.25);
\draw [dashed] (11.5,21.25) -- (11.5,15.25);
\node [font=\Large] at (13.5,22.5) {$D_{\mathrm{Ld}}$};
\node [font=\Large] at (8.75,22.5) {$D_{\mathrm{Ld}}^c$};
\node [font=\Large] at (8.75,21.6) {$0$};
\node [font=\Large] at (13.5,21.6) {$1$};

\draw  (6.25,20) rectangle (15.75,19.25);
\draw [<->, >=Stealth, dashed] (9.2,20.25) -- (13,20.25);
\draw [short] (9.2,20) -- (9.2,19.25);
\draw [short] (13,20) -- (13,19.25);
\draw [dashed] (9.2,19.25) -- (9.2,18);
\draw [dashed] (13,19.25) -- (13,16);
\node [font=\Large] at (7.5,19.6) {$-$};
\node [font=\Large] at (10.25,19.6) {$+$};
\node [font=\Large] at (12.25,19.6) {$+$};
\node [font=\Large] at (14.25,19.6) {$-$};
\node [font=\Large] at (7.5,18.75) {$L_1$};
\node [font=\Large] at (10,18.75) {$L_2$};
\node [font=\Large] at (12.25,18.75) {$L_3$};
\node [font=\Large] at (14.5,18.75) {$L_4$};
\node [font=\Large] at (11,20.75) {$D_0$};

\draw  (6.25,18) rectangle (15.75,17.25);
\draw [short] (8,18) -- (8,17.25);
\draw [short] (9.2,18) -- (9.2,17.25);
\draw [short] (13,18) -- (13,17.25);
\draw [short] (15,18) -- (15,17.25);
\draw [dashed] (8,17.25) -- (8,16);
\node [font=\Large] at (7,17.6) {$-$};
\node [font=\Large] at (8.5,17.6) {$-$};
\node [font=\Large] at (10.25,17.6) {$+$};
\node [font=\Large] at (12.25,17.6) {$+$};
\node [font=\Large] at (13.85,17.6) {$-$};
\node [font=\Large] at (15.4,17.6) {$-$};
\draw [<->, >=Stealth, dashed] (8,17) -- (9.2,17);
\draw [<->, >=Stealth, dashed] (13,17) -- (15,17);
\draw [dashed] (8.5,17) -- (10.25,16.55);
\draw [dashed] (13.85,17) -- (11,16.55);
\node [font=\Large] at (10.75,16.5) {$B_0^-$};

\draw  (6.25,16) rectangle (15.75,15.25);
\draw [short] (8,16) -- (8,15.25);
\draw [short] (13,16) -- (13,15.25);
\node [font=\Large] at (7,15.6) {0};
\node [font=\Large] at (9.75,15.6) {1};
\node [font=\Large] at (12,15.6) {1};
\node [font=\Large] at (13.85,15.6) {0};
\draw [<->, >=Stealth, dashed] (8,15) -- (13,15);
\node [font=\Large] at (10.5,14.5) {$A_0^-$};
\end{circuitikz}
}
\caption{Visualization of the subsets appearing in the proof of the second-order optimality of the Leap rule.
}\label{fig: A_0 leap}
\end{figure}
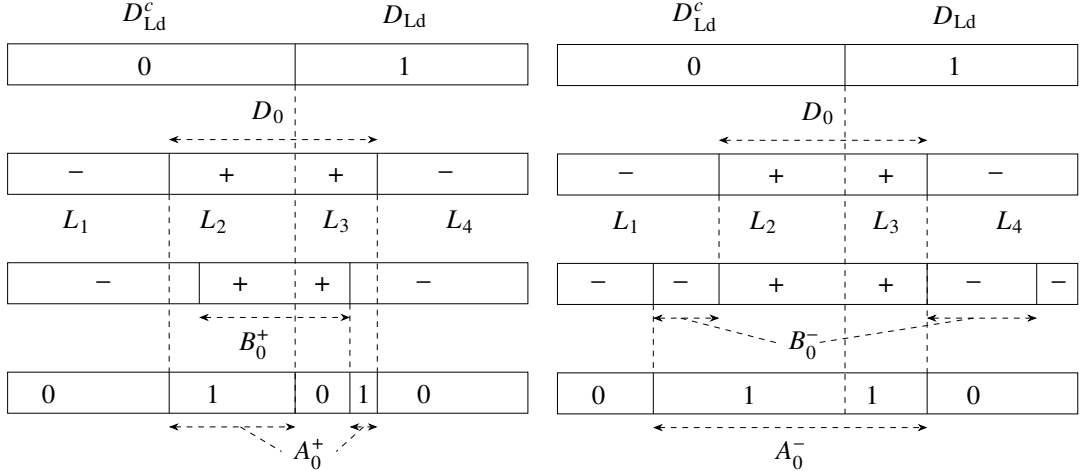

\begin{lemma}\label{lemma: leap rule claim}
Consider the subsets $L_1,L_2,L_3,L_4,D_0$ in \eqref{eq: stop earlier leap 3}, and $A_0^+, A_0^-, B_0^+, B_0^-, \tilde{B}^-$ in \eqref{eq: stop earlier leap 4}.  We assume $|L_2|\le |L_4|$ and denote by $i^*:=|L_4|-|L_2|$. Then
\begin{equation*}
 |D_{\mathrm{Ld}}\setminus A_0^+| \ge m_1, \quad \text{ and } \quad  |A_0^-\setminus D_{\mathrm{Ld}}| \ge m_2.
\end{equation*}
\end{lemma}
\begin{proof} 
Since $B_0^+\subset L_2\cup L_3$, $B_0^-\subset L_1\cup L_4$ and $L_1,L_2,L_3,L_4$ are disjoint, we have $B_0^+\cap L_3 = B_0^+\setminus L_2$ and $B_0^-\cap L_1 = B_0^-\setminus L_4$, which implies
$$|B_0^+\cap L_3| \ge |B_0^+|-|L_2| = m_1-i^*-|L_2| = m_1-|L_4|,$$ 
$$|B_0^-\cap L_1| \ge |B_0^-|-|L_4| = m_2+i^*-|L_4| = m_2-|L_2|.$$
In addition, due to the definition of $\tilde{B}^-$, if $|B_0^-\cap L_1|\le m_2$ we have $\tilde{B}^- = B_0^-\cap L_1$; otherwise, we have $|\tilde{B}^-| =  m_2$.

Next, since $D_{\mathrm{Ld}} = L_3 \cup L_4$, $A_0^+ = L_2 \cup (L_3\setminus B_0^+)$ and $A_0^- = \tilde{B}^-\cup L_2\cup L_3$, it follows that
\begin{align*}
D_{\mathrm{Ld}}\setminus A_0^+ = (B_0^+\cap L_3) \cup L_4, \quad \text{ and } \quad A_0^-\setminus D_{\mathrm{Ld}} = \tilde{B}^- \cup L_2.
\end{align*}

As a result, we have 
$$|D_{\mathrm{Ld}}\setminus A_0^+| = |B_0^+\cap L_3| + |L_4|\ge m_1, \quad \text{ and } \quad  |A_0^-\setminus D_{\mathrm{Ld}}| = |\tilde{B}^-| + |L_2|\ge m_2.$$
This completes the proof.
\end{proof}

\begin{proof}[Proof of Lemma \ref{gfr:sufficient_conditions_unique}]
For part (a), note that when $m_1=m_2=1$, the loss function $\mW$ in \eqref{eq: loss gfr} is a zero-one loss as defined in Remark \ref{remark: zero-one loss}. The conclusion therefore follows directly from Remark \ref{remark: zero-one loss}.

Next, we focus on part (b), where $\mathcal{I}_0^k=\mathcal{I}_1^k=\mathcal{I}^* \text{ for } k\in[K]$ and $m_1=m_2=m^*$.  Let $A \subset [K]$ be such that  $\min\{|A|,|A^c|\}\ge m^*$. 
Recall the definition of $\mW$ in \eqref{eq: loss gfr} and of $\KL_{A,D}^{\mW}$ in \eqref{def:KL_A_D_star}.
It is clear that $\KL_{A,A}^{\mW} = m^*\mathcal{I}^*$. Thus, it is sufficient to prove that for any $D\neq A$, there exists a subset $C^*\in\mathcal{H}_{D}^{\mW}$, such that 
\begin{equation*}
\KL(f_A\mid f_{C^*})< m^*\mathcal{I}^*.
\end{equation*}
We denote by 
\begin{equation*}
\begin{aligned}
&L_1 = A\setminus D, \ L_2 = A\cap D, \ L_3 = D\setminus A, \ L_4 = D^c\setminus A.
\end{aligned}
\end{equation*}
Since $\min\{|A|,|A^c|\}\ge m^*$, we have $\sum_{i=1}^4 |L_i| = |A| + |A^c| \ge 2m^*$, which implies either $|L_2|+|L_3|\ge m^*$ or $|L_1|+|L_4|\ge m^*$. 
We provide the proof for the case $|L_2|+|L_3|\ge m^*$; the other case is analogous.

\noindent \textbf{Case 1.} $L_3\neq \emptyset$. In this case, we define 
$$
C^* = L_1\cup(L_2\setminus\Gamma_2), \;\text{ with } \Gamma_2\subset L_2 \;\text{ such that } |\Gamma_2| = \max\{m^*-|L_3|,0\},
$$
where we note that such $\Gamma_2$ always exists since $|L_2|\ge m^*-|L_3|$. Note that 
$$
D\setminus C^* = \Gamma_2 \cup L_3 \quad 
A\setminus C^* = \Gamma_2, \quad \text{ and } \quad  C^* \setminus A= \emptyset.
$$
On one hand, it implies that $|D\setminus C^*| = |\Gamma_2|+|L_3|\ge m^*$, and thus $C^*\in\mathcal{H}_{D}^{\mW}$. On the other, since in this case $|\Gamma_2| < m^*$,
$$\KL(f_A\mid f_{C^*}) = |\Gamma_2|\mathcal{I}^*< m^*\mathcal{I}^*.$$
The proof under Case 1 is complete.

\noindent \textbf{Case 2.} $L_3 =  \emptyset$. In this case, since $D \neq A$, we must have $L_1 \neq \emptyset$. Furthermore, we have $|L_2| \geq m^*$ and $|L_4| = |A^c| \ge m^*$. 
We define
$$
C^* = L_1\cup L_2\cup \Gamma_4, \text{ with }\Gamma_4\subset L_4 \text{ such that } |\Gamma_4| = \max\{m^*-|L_1|,0\},
$$
where we note that such $\Gamma_4$ always exists since $|L_4| \ge m^*$. Note that
\begin{align*}
    C^* \setminus D = L_1 \cup \Gamma_4, \quad
    C^*\setminus A = \Gamma_4, \quad \text{ and } \quad A \setminus C^* = \emptyset
\end{align*}
On one hand, it implies that $|C^*\setminus D| = |L_1|+|\Gamma_4|\ge m^*$, and thus $C^*\in\mathcal{H}_{D}^{\mW}$. On the other, since $|L_1| >0$, we have that $|\Gamma_4| < m^*$, and thus
$$\KL(f_A\mid f_{C^*}) = |\Gamma_4|\mathcal{I}^*< m^*\mathcal{I}^*.$$
The proof under Case~2 is complete, which concludes the proof of part~(b).
\end{proof}

\begin{figure}[!t]
\centering
\resizebox{200pt}{!}{
\begin{circuitikz}
\tikzstyle{every node}=[font=\Large]
\draw  (6.25,22) rectangle (15.75,21.25);
\draw [short] (11.5,22) -- (11.5,21.25);
\draw [dashed] (11.5,21.25) -- (11.5,17.25);
\node [font=\Large] at (8.75,21.6) {$A$};
\node [font=\Large] at (13.5,21.6) {$A^c$};

\draw  (6.25,20) rectangle (15.75,19.25);
\draw [<->, >=Stealth, dashed] (9.2,20.25) -- (13,20.25);
\draw [short] (9.2,20) -- (9.2,19.25);
\draw [short] (13,20) -- (13,19.25);
\draw [dashed] (9.2,19.25) -- (9.2,17.25);
\draw [dashed] (13,19.25) -- (13,17.25);
\node [font=\Large] at (7.5,19.6) {$L_1$};
\node [font=\Large] at (10.25,19.6) {$L_2$};
\node [font=\Large] at (12.25,19.6) {$L_3$};
\node [font=\Large] at (14.25,19.6) {$L_4$};
\node [font=\Large] at (11,20.75) {$D$};

\draw  (6.25,18) rectangle (15.75,17.25);
\draw [short] (9.75,18) -- (9.75,17.25);
\draw [short] (14.5,18) -- (14.5,17.25);
\node [font=\Large] at (10.5,17.6) {$\Gamma_2$};
\node [font=\Large] at (13.85,17.6) {$\Gamma_4$};
\draw [<->, >=Stealth, dashed] (6.25,18.25) -- (11.5,18.25);
\draw [<->, >=Stealth, dashed] (13,18.25) -- (14.5,18.25);
\draw [dashed] (8.5,18.25) -- (10,18.75);
\draw [dashed] (14,18.25) -- (11.25,18.75);
\node [font=\Large] at (10.75,18.75) {$C^*$};
\draw [<->, >=Stealth, dashed] (6.25,17) -- (9.75,17);
\node [font=\Large] at (8,16.5) {$C^*$};
\end{circuitikz}
}
\caption{Visualization of the subsets appearing in the proof of the ESS approximation under the generalized familywise error rates.}
\label{fig: C^* and tilde C familywise}
\end{figure}
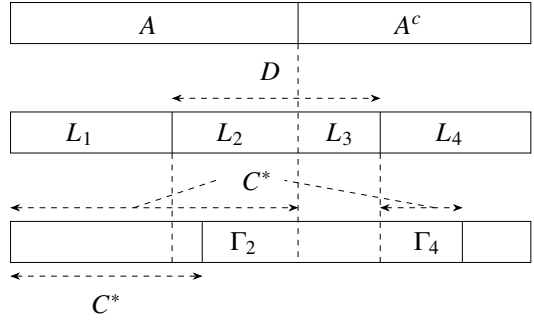

\subsection{False Discovery and Non-discovery Rate}\label{sec: proof of fdr}
\begin{proof}[Proof of Theorem \ref{thm: second-order optimal of intersection rule}]
As shown in Subsection 2.2 of \cite{he2021asymptotically}, for any decision rule $D$, we have the following inequalities: for each $A \subset [K]$,
\begin{equation*}
\begin{aligned}
&\frac{1}{K}\prob_A(|D\setminus A|\ge 1) \le \expt_A\left[\frac{|D\setminus A|}{|D|\vee 1}\right] \le \prob_A(|D\setminus A|\ge 1),
\\&\frac{1}{K}\prob_A(|A\setminus D|\ge 1) \le \expt_A\left[\frac{|A\setminus D|}{|D|\vee 1}\right]\le \prob_A(|A\setminus D|\ge 1),
\end{aligned}
\end{equation*}
which implies the following relationship between
Class \ref{problem: fdr} and Class \ref{problem: gfr} with $m_1 =m _2 =1$:
\begin{equation*}
    \Delta_{1,1}^{\mathrm{gfr}}(\alpha,\beta)\subset \Delta^{\textup{fdr}}(\alpha,\beta)\subset \Delta_{1,1}^{\mathrm{gfr}}(K\alpha,K\beta).
\end{equation*}
Then, by definition, it follows that for $\alpha,\beta \in (0,1)$ and each $A \subset [K]$,
\begin{equation}
    \label{eq: T_comp}
    T_A^{\textup{min}}\left(\Delta_{1,1}^{\mathrm{gfr}}(\alpha,\beta)\right) \geq 
     T_A^{\textup{min}}\left(\Delta^{\textup{fdr}}(\alpha,\beta)\right)
    \geq T_A^{\textup{min}}\left(\Delta_{1,1}^{\mathrm{gfr}}(K\alpha,K\beta)\right).
\end{equation}

Furthermore, when $m_1 = m_2 = 1$, the Intersection rule $\delta_I(a,b)$ in \eqref{def:intersection_rule} coincides with the Leap rule $\delta_L(a,b)$ in \eqref{eq: leap rule}, and the thresholds $(a_{\alpha,\beta}, b_{\alpha,\beta})$ in \eqref{fdr:a_b} are the same as $(a_{\alpha,\beta}, b_{\alpha,\beta})$ in \eqref{eq: leap_thresholds}. Then by Theorem \ref{thm: second-order optimal of leap rule}, there exists a constant $C' > 0$ that does not depend on $\alpha,\beta$, such that for each $A \subset [K]$, and $\alpha,\beta \in (0,1)$ such that \eqref{eq:alpha_beta_same_rate} holds, we have
\begin{align*}
 \expt_A[T_I(a_{K\alpha,K\beta},b_{K\alpha,K\beta})]
-
T_A^{\textup{min}}\left(\Delta_{1,1}^{\mathrm{gfr}}(K\alpha,K\beta)\right) \leq C'.
\end{align*}
Then due to \eqref{eq: T_comp}, we immediately have
\begin{align*}
&     \expt_A[T_I(a_{\alpha,\beta},b_{\alpha,\beta})] - T_A^{\textup{min}}\left(\Delta^{\textup{fdr}}(\alpha,\beta)\right) \leq      \expt_A[T_I(a_{\alpha,\beta},b_{\alpha,\beta})] - T_A^{\textup{min}}\left(\Delta_{1,1}^{\mathrm{gfr}}(K\alpha,K\beta)\right) \\
     &\leq \expt_A[T_I(a_{\alpha,\beta},b_{\alpha,\beta})] - \expt_A[T_I(a_{K\alpha,K\beta},b_{K\alpha,K\beta})] + C'
\end{align*}
The proof is then complete in view of Lemma \ref{lemma:aux_fdr}.
\end{proof}

Denote by $\bell_{\alpha,\beta} = (a_{\alpha,\beta}, b_{\alpha,\beta})$. Consider the Intersection rule defined in \eqref{def:intersection_rule}.
\begin{lemma}\label{lemma:aux_fdr}
    Suppose Assumption \ref{assumption: lorden} holds, and 
    $\alpha,\beta \in (0,1)$ satisfy \eqref{eq:alpha_beta_same_rate}. Then, there exists a constant $L' > 0$, that does not depend on $\alpha,\beta$, such that for each $A \subset [K]$ and any $\alpha,\beta \in (0,1)$
    \begin{align*}
         \expt_A[T_I(\bell_{\alpha,\beta})] - \expt_A[T_I(\bell_{K\alpha,K\beta})] \leq L'.
    \end{align*}
\end{lemma}
\begin{proof}
Without loss of generality, assume that $A = [K]$, and 
that $\alpha, \beta \leq (K)^{-1}$. Then
\begin{equation}\label{eq: bell alpha, beta}
\begin{aligned}
    \bell_{K\alpha, K\beta} = (|\log(\beta)|, |\log(\alpha)|),\quad \text{ and } \quad 
     \bell_{\alpha, \beta} =(|\log(\beta)| + \log(K), |\log(\alpha)|+ \log(K)).
\end{aligned}    
\end{equation}
For $a,b > 0$, define
    \begin{align*}
    \tilde{T}_I(a,b) := \inf\left\{t \geq 1: \lambda_t^{k} \geq b, \text{ for } k \in [K]\right\}.
    \end{align*}
Then by the definition of  the Intersection rule  in \eqref{def:intersection_rule},  
\begin{align*}
 T_I(\bell_{\alpha,\beta})  \leq \tilde{T}_I(\bell_{\alpha,\beta} ),\quad 
  T_I(\bell_{K\alpha,K\beta})  \leq \tilde{T}_I(\bell_{K\alpha,K\beta} ).
\end{align*}
Furthermore, by the same argument as in Theorem \ref{thm: min ESS characterization},  as $\alpha \vee \beta \to 0$ such that \eqref{eq:alpha_beta_same_rate} holds, we have
\begin{align*}
    \expt_{[K]}\left[T_I(\bell_{K\alpha,K\beta})  \right]  =  \expt_{[K]}\left[\tilde{T}_I(\bell_{K\alpha,K\beta})\right] + O(1),
\end{align*}
which implies that
 \begin{align*}
         \expt_{[K]}[T_I(\bell_{\alpha,\beta})] - \expt_{[K]}[T_I(\bell_{K\alpha,K\beta})] \leq   \expt_{[K]}\left[\tilde{T}_I(\bell_{ \alpha, \beta})\right] - \expt_{[K]}\left[\tilde{T}_I(\bell_{K\alpha,K\beta})\right] +O(1).
    \end{align*}
Finally, we show that as $\alpha \vee \beta \to 0$ ,  $\expt_{[K]}\left[\tilde{T}_I(\bell_{ \alpha, \beta})\right] - \expt_{[K]}\left[\tilde{T}_I(\bell_{K\alpha,K\beta})\right] = O(1)$. We define
\begin{align*}
\tau^* := \tilde{T}_I(\log(K), \log(K)) = \inf\left\{t \geq 1: \lambda_t^{k} \geq  \log(K), \text{ for } k \in [K]\right\}.
\end{align*}
It is clear that $\tilde{L} := \expt_{[K]}\left[\tau^*\right] < \infty$, which is a constant. Moreover, by definition,
$$
\tilde{T}_I(\bell_{ \alpha, \beta}) \leq \tau^* + \widehat{T}_I(\bell_{ K\alpha, K\beta}),
$$
where in view of \eqref{eq: bell alpha, beta} we define
\begin{align*}
    \widehat{T}_I(\bell_{ K\alpha, K\beta}) := \inf\left\{ t \geq \tau^*: \lambda_t^{k} -  \lambda_{\tau^*}^{k} \geq |\log(\alpha)| \text{ for } k \in [K]\right\} - \tau^*.
\end{align*}
Since $\{\mathbf{X}_t^{[K]}: t \geq 1\}$ are i.i.d., we have $\widehat{T}_I(\bell_{ K\alpha, K\beta})$ have the same distribution under $\prob_{[K]}$ as $\tilde{T}_I(\bell_{K\alpha, K\beta})$. As a result,  
\begin{align*}
    \expt_{[K]}\left[\tilde{T}_I(\bell_{ \alpha, \beta})\right] \leq  \expt_{[K]}\left[\tau^*\right] + \expt_{[K]}\left[ \widehat{T}_I(\bell_{ K\alpha, K\beta})\right] \leq \tilde{L} + \expt_{[K]}\left[\tilde{T}_I(\bell_{K\alpha,K\beta})\right],
\end{align*}
which concludes the proof.
\end{proof}

\subsection{Known Number of Signals}\label{sec: proof of kns}
In this subsection, we present the proofs of the results in Appendix \ref{app:kns}. 
We recall the definition of the stopping rule $T_G(b)$ in \eqref{eq: gap rule}, the loss function $\mW$ in \eqref{eq: loss kns}, and the associated rule $T_{\mathrm{Ld}}(c,\mW)$ in \eqref{eq: lorden procedure}. Furthermore, $\pi_0$ denotes the uniform distribution on $\mathcal{A} = \{A \subset [K]: |A| = m\}$.

\begin{proof}[Proof of Theorem \ref{thm: second-order optimal of gap rule}]
We define
\begin{align}\label{eq: gap_constants_def}
c_{\alpha} :=  \binom{K}{m}^{-1} e^{-b_{\alpha}} = \binom{K}{m}^{-1}[m(K-m)]^{-1}\alpha, \qquad
L := \binom{K}{m}m(K-m).
\end{align}
where we recall $b_{\alpha}$ is defined in \eqref{gap:b}.  
Then the conclusion follows immediately from Theorem \ref{thm: main result} once we verify condition \eqref{assumption: lorden 1}--\eqref{eq: second-order optimal sufficient condition 2} hold with the loss function $\mW$ in \eqref{eq: loss kns}, cost $c_{\alpha}$ and constant $L$.

We start with condition \eqref{assumption: lorden 1}.  Since $1\leq m < K$, for each $D \in \mathcal{A}$, there exists $i \in D$ and $i' \in D^c$. Now, let $A := (D\setminus \{i\}) \cup \{i'\}$. Then clearly $|A| =m$, and thus $A \in \mathcal{A}$. Then by the definition \eqref{eq: loss kns}, we have $W(D|A) = 1$, which verifies condition \eqref{assumption: lorden 1}.

Next, we consider condition \eqref{eq: second-order optimal sufficient condition 1}. Since  $b_{\alpha}=\log \left[\left(\binom{K}{m}c_{\alpha}\right)^{-1}\right]$, by Lemma \ref{lemma: stop earlier gap}, we have
$$
T_G(b_{\alpha})\le T_{\mathrm{Ld}}(c_{\alpha},\mW).
$$

Finally, we focus on condition \eqref{eq: second-order optimal sufficient condition 2}. For an arbitrary procedure $\delta=(T,D)\in\Delta_m^{\textup{kns}}(\alpha)$, by the definitions in \eqref{eq: problem kns}, \eqref{eq: loss kns} and \eqref{eq: gap_constants_def}, we have
\begin{align*}
    \ie(\delta; \mW) &= \sum_{|A|=m} \pi_0(A) \prob_A(|D\triangle A|\geq 1)\le \alpha\sum_{|A|=m} \pi_0(A) = \alpha = L c_{\alpha}.
\end{align*}
The proof is complete.
\end{proof}

Recall the loss function $\mW$ in \eqref{eq: loss kns},   the definition of $\blambda_t^{(k)}$   in \eqref{eq: blamda} and the definition of $\bar{j}_k(t)$ in Appendix \ref{app:kns}.

\begin{lemma}\label{lemma: stop earlier gap}
Assume the cost $c\le \binom{K}{m}^{-1}$. Then for the threshold $b = \log\left[\left(\binom{K}{m}c\right)^{-1}\right]$, we have 
$$
T_G(b)\le T_{\mathrm{Ld}}(c,\mW).
$$
\end{lemma}
\begin{proof} 
Since $c$, $b$, and $\mW$ are fixed, we suppress their dependence in $T_G(b)$ and $T_{\mathrm{Ld}}(c,\mW)$. Suppose that $T_{\mathrm{Ld}} < \infty$ since otherwise the conclusion trivially holds. We show that the stopping criterion of $T_G$ is met at time $T_{\mathrm{Ld}}$, that is,
\begin{equation}
    \label{kns:aux_comp}
\blambda_{T_{\mathrm{Ld}}}^{(K-m+1)} - \blambda_{T_{\mathrm{Ld}}}^{(K-m)} \geq b.
\end{equation}
This immediately implies $T_G \le T_{\mathrm{Ld}}$.

Let $D_0 := \{\bar{j}_{K-m+1}(T_{\mathrm{Ld}}),\ldots,\bar{j}_{K}(T_{\mathrm{Ld}})\}$ be the collection of stream indices with $m$ largest LLRs at $T_{\mathrm{Ld}}$, 
We define
$$
A_0 := (D_0\setminus\{\bar{j}_{K-m+1}(T_{\mathrm{Ld}})\})\cup\{\bar{j}_{K-m}(T_{\mathrm{Ld}})\},
$$
which implies that
$$
D_0 \setminus A_0 = \{\bar{j}_{K-m+1}(T_{\mathrm{Ld}})\}, \qquad A_0 \setminus D_0 =\{\bar{j}_{K-m}(T_{\mathrm{Ld}})\}. 
$$
By Lemma \ref{lemma: |A_0| known}, we have
$$
\sum_{i\in D_0}\lambda_{T_{\mathrm{Ld}}}^i- \sum_{j\in A_0}\lambda_{T_{\mathrm{Ld}}}^j > \log[c^{-1}\pi_0(A_0)\mW(D_{\mathrm{Ld}}\mid A_0)].
$$
On one hand, we have
$$\sum_{i\in D_0}\lambda_{T_{\mathrm{Ld}}}^i- \sum_{j\in A_0}\lambda_{T_{\mathrm{Ld}}}^j = \blambda_{T_{\mathrm{Ld}}}^{(K-m+1)}-\blambda_{T_{\mathrm{Ld}}}^{(K-m)}.$$
On the other hand, since $D_{\mathrm{Ld}} \neq A_0$, by the definition of $\mW$ in \eqref{eq: loss kns}, $\mW(D_{\mathrm{Ld}}\mid A_0)=1$. Since $\pi_0(A_0) = \left(\binom{K}{m}\right)^{-1}$, it follows that \eqref{kns:aux_comp} holds, and the proof is complete.
\end{proof}

\section{Numerical results for the asymmetric case}
\label{app:addition_simulation}

In this Appendix, we present additional numerical results under the generalization misclassification rate. Recall the setup in Section \ref{sec:numerical_results}. Consider a nonhomogeneous and asymmetric setup, where
\[
K=20,\qquad \mu_1=1/2,\ \mu_2=\cdots=\mu_K=1,\qquad \sigma_1^2=\cdots=\sigma_K^2=1,
\]
so that the first hypothesis is much harder to solve than the others. We take $m_0=1$ and consider the empty signal subset $A=\emptyset$. In this case,
\[
\mathcal{C}_{\emptyset}^{\mW}=\{\{1\}\},\qquad r_{\emptyset}^{\mW}=1,\quad \text{ and } \quad \KL_{\emptyset,*}^{\mW}=\KL_{\emptyset,\emptyset}^{\mW}=1/8.
\]
By \eqref{eq: first_order_gmr} and Theorem \ref{thm: second-order min ESS gmr}(a), the first- and second-order approximations coincide and
\begin{align*}
T_{\emptyset}^{\textup{min}}\left(\Delta_{m_0}^{\mathrm{gmr}}(\alpha)\right)=\textup{SO}_{\alpha}+O(1),\quad \text{where}\quad \textup{SO}_{\alpha}:=8|\log\alpha|.
\end{align*}
We present the numerical results in Figure \ref{fig: asymmetric_SI_20_1}. The leftmost plot displays the ESS $\expt_{\emptyset}[T_S(b_{\alpha}^*)]$ (square markers), together with the second-order approximation $\textup{SO}_{\alpha}$ (circle markers), as functions of $|\log_{10}(\alpha)|$ as $\alpha\in(0,1)$ varies. The middle and rightmost plots display the difference $\expt_{\emptyset}[T_S(b_{\alpha}^*)]-\textup{SO}_{\alpha}$ and the ratio $\expt_{\emptyset}[T_S(b_{\alpha}^*)]/\textup{SO}_{\alpha}$ as functions of $|\log_{10}(\alpha)|$. Although Theorems \ref{thm: second-order optimal of sum intersection rule} and \ref{thm: second-order min ESS gmr}(a) only imply that $\expt_{\emptyset}[T_S(b_{\alpha}^*)]-\textup{SO}_{\alpha}=O(1)$, the middle panel suggests that this difference vanishes as $\alpha\to0$. It would be of interest to establish this refinement rigorously.

\begin{figure}[!t]
\centering
\includegraphics[width=0.95\textwidth]{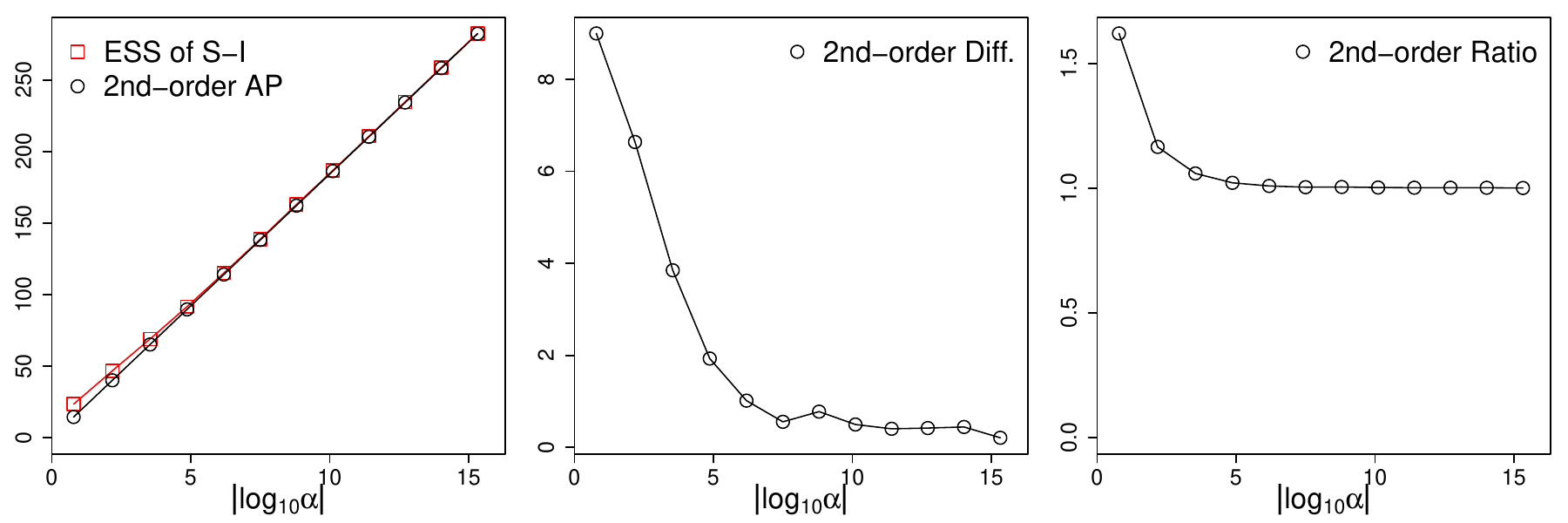}
\caption{Asymmetric Case: $K=20, m_0=1$. The x-axis in all panels is $|\log_{10}\alpha|$. In the leftmost plots, ``S-I'' denotes the Sum-Intersection rule and ``AP'' denotes the asymptotic approximation  to the smallest achievable ESS. In the middle plots, ``Diff.'' denotes the difference between the ESS of the S-I rule and the two approximations. In the rightmost plots, ``Ratio'' denotes the ratio of the ESS of the S-I rule to each of the two approximations.}
\label{fig: asymmetric_SI_20_1}
\end{figure}

\end{appendix}

\end{document}